\documentclass[reqno]{amsart} 

\usepackage[T1]{fontenc}
\usepackage[utf8]{inputenc}
\usepackage[english]{babel}
\usepackage{newpxtext} 
\usepackage{eulervm}

\usepackage{setspace} 
\usepackage{adjustbox}

\usepackage{etoolbox}
	\apptocmd{\sloppy}{\hbadness 10000\relax}{}{} 

\AtBeginDocument{ 
   \def\MR#1{}
} 

\usepackage{mathtools}  
	\mathtoolsset{showonlyrefs} 

\usepackage{amssymb,textcomp,mathrsfs,stmaryrd,braket}
\usepackage{bm} 

\usepackage{tikz}
	\usetikzlibrary{decorations.markings,decorations.pathreplacing,arrows,graphs,intersections,positioning,calc}
	\tikzstyle{vertex}=[circle,draw,inner sep=0pt,minimum size=12pt] 
	\newcommand{\vertex}{\node[vertex]}
	\tikzset{
	every picture/.style={thick,>=latex,->,node distance=2cm} 
	}

\usepackage{float} 

\theoremstyle{plain}
	\newtheorem{thm}{Theorem}[section]
	\newtheorem*{thm*}{Theorem}
	\newtheorem{prop}{Proposition}[section]
	\newtheorem*{prop*}{Proposition}
	\newtheorem{corol}{Corollary}[section]
	\newtheorem*{corol*}{Corollary}
	\newtheorem{lem}{Lemma}[section]
	\newtheorem*{lem*}{Lemma}

\theoremstyle{definition}
	\newtheorem{defn}{Definition}[section]

\theoremstyle{remark}
	\newtheorem{rem}{Remark}[section]

\newcommand{\addQEDstyle}[2]{\AtBeginEnvironment{#1}{\pushQED{\qed}\renewcommand{\qedsymbol}{#2}}
\AtEndEnvironment{#1}{\popQED}} 
	\addQEDstyle{rem}{$\triangle$}
	\addQEDstyle{rem*}{$\triangle$} 
	\addQEDstyle{exmp}{$\triangle$}
	\addQEDstyle{exmp*}{$\triangle$} 

\usepackage{mymacros}

\usepackage{hyperref}
	\hypersetup{colorlinks, linkcolor=red, urlcolor=blue, citecolor=blue}

\begin{document}
	\title[Simply-laced quantum connections]{Simply-laced quantum connections generalising KZ} 

	\author[G. Rembado]{Gabriele Rembado}

	\address[G. Rembado]{Department of Mathematics of Orsay, University of Paris-Sud/Saclay, Rue Michel Magat, 91405, Orsay, France}

	\curraddr{Hausdorff Centre for Mathematics, University of Bonn, 60 Endenicher Allee, 53115, Bonn, Germany}

	\email{gabriele.rembado@hcm.uni-bonn.de}

	\thanks{This research was completed while the author was a PhD candidate at the \emph{Laboratoire de Mathématiques d'Orsay} (University of Paris-Sud/Saclay).}



	\begin{abstract}
		We construct a new family of flat connections generalising the KZ connection, the Casimir connection and the dynamical connection. These new connections are attached to simply-laced graphs, and are obtained via quantisation of time-dependent Hamiltonian systems controlling the isomonodromic deformations of meromorphic connections on the sphere.
	\end{abstract}

	{\let\newpage\relax\maketitle} 

	\setcounter{tocdepth}{1} 
	\tableofcontents

	\section*{Introduction}

Choose positive integers $m,n$, and set $\mathfrak{g} \coloneqq \mathfrak{gl}_n(\mathbb{C})$, equipped with the nondegenerate pairing $\mathfrak{g} \otimes \mathfrak{g} \to \mathbb{C}$ given by the trace. 
Let then $\mathbf{B} \coloneqq \mathbb{C}^m \setminus \{\diags\}$ be the configuration space of ordered $m$-tuples of point in $\mathbb{C}$ with standard complex coordinates $t_i$, and $U(\mathfrak{g})$ the universal enveloping algebra of $\mathfrak{g}$. 
Finally, consider the trivial bundle $U(\mathfrak{g})^{\otimes m} \times \mathbf{B} \to \mathbf{B}$. 

The universal Knizhnik--Zamolodchikov equations (KZ) are a system of linear differential equations for a local section $\psi$ of this bundle. 
To write them, define $\Omega \in \mathfrak{g} \otimes \mathfrak{g}$ to be symmetric tensor corresponding to the identity $\Id_{\mathfrak{g}} \in \mathfrak{g} \otimes \mathfrak{g}^*$ under the duality $\mathfrak{g} \simeq \mathfrak{g}^*$ induced by the nondegenerate pairing, and $\Omega^{(ij)} \in \End\big(U(\mathfrak{g})^{\otimes m}\big)$ the action of $\Omega$ by left multiplication on the $i$th and $j$th slot of the $m$-fold tensor power $U(\mathfrak{g})^{\otimes m}$. 
Then the KZ equations read
\begin{equation}
	d\psi = \widehat{\varpi}\psi, \qquad \text{where} \qquad \widehat{\varpi} \coloneqq \sum_{1 \leq i \neq j \leq m} \Omega^{(ij)}\frac{dt_i - dt_j}{t_i - t_j} \, .
\end{equation}
This system originated as equations for correlators in the Wess--Zumino--Witten model for two-dimensional conformal field theory~\cite{knizhnik_zamolodchikov_1984_wess_zumino_witten}. 
Mathematically it amounts to a flat connection whose monodromy provides important representations of the $m$-string braid group---in view of the Kohno--Drinfeld theorem. 
Moreover the higher-genus analogue of the KZ connection, in the WZW model for 2d conformal field theory~\cite{tsuchiya_ueno_yamada_1989_conformal_field_theory}, is known to be equivalent to the Hitchin connection in geometric quantisation~\cite{beauville_laszlo_1994_conformal_blocks_and_generalised_theta_functions, laszlo_1998_hitchin_equals_wzw}. 

\vspace{5pt} 

It has been known for some years~\cite{reshetikhin_1992_kz_deformation_isomonodromy, harnad_1996_quantum_imd_and_the_kz_equations} that KZ can be obtained as a deformation quantisation of a system of nonlinear differential equations: the Schlesinger system~\cite{schlesinger_1905_ueber_die_loesungen_gewisser_linearer_differentialgleichungen}. This system is defined for matrices $R_1 \dotsc, R_m \in \mathfrak{g}$ which depend on a configuration of points in $\mathbb{C}$, and can be written in differential form as
\begin{equation}
	dR_i = \sum_{j \neq i} [R_i,R_j]\frac{dt_i - dt_j}{t_i - t_j} \, .
\end{equation}
These equations control the isomonodromic deformations of Fuchsian systems $d - \sum_i \frac{R_i}{z - t_i}dz$ on $\mathbb{C} P^1$, and are in turn determined by time-dependent Hamiltonians $H_i \colon \mathfrak{g}^m \times \mathbf{B} \to \mathbb{C}$, where
\begin{equation}
	H_i \coloneqq \sum_{j \neq i} \Tr(R_iR_j)\frac{dt_i -dt_j}{t_i - t_j} \, .
\end{equation}
The key idea behind the quantisation is that the function $\Tr(R_iR_j)$ on $\mathfrak{g}^m$ becomes $\Omega_{ij}$ under the Poincaré--Birkhoff--Witt map $\Sym(\mathfrak{g})^{\otimes m} \to U(\mathfrak{g})^{\otimes m}$. 

\vspace{5pt}

Almost two decades later various generalisations of KZ have appeared, such as the FMTV connection of Felder--Markov--Tarasov--Varchenko~\cite{felder_markov_tarasov_varchenko_2000_dynamical_connection}, and the DMT connection of De Concini and Millson--Toledano~Laredo~\cite{millson_toledano_laredo_2005_casimir_connection}, also known as the \emph{Casimir} connection, which is essentially an important special case of~\cite{felder_markov_tarasov_varchenko_2000_dynamical_connection}. 
The DMT connection may also be derived from isomonodromy via a simple deformation quantisation, but this time from an irregular isomonodromy problem~\cite{boalch_2002_g_bundles_isomonodromy_quantum_weyl_groups}. 

Importantly, in~\cite{felder_markov_tarasov_varchenko_2000_dynamical_connection} the space of times of KZ is increased by adding on the regular part of a Cartan subalgebra, and thus in the case of $\mathfrak{gl}_n(\mathbb{C})$ with $m$ marked points it becomes a product $\mathbb{C}^m \setminus \{\diags\} \times \mathbb{C}^n \setminus \{\diags\}$. 
The extra times correspond to the irregular isomonodromy times of the JMMS system of Jimbo--Miwa--M\^ori--Sato~\cite{jimbo_miwa_mori_sato_1980_density_matrix_bose_gas}. Harnad~\cite{harnad_1994_dual_isomonodromic_deformations} has shown that the two collections of times in the JMMS system may be swapped, and this classical duality between the Schlesinger system and its dual version underlies the quantum/Howe duality of~\cite{baumann_1999_q_weyl_group} used in~\cite{toledano_laredo_2002_kohno_drinfeld_for_quantum_weyl_groups} to relate the KZ to the DMT connection for $\mathfrak{gl}_n(\mathbb{C})$. 

\vspace{5pt}

More recently, the Hamiltonian theory of isomonodromy equations was extended~\cite{boalch_2012_simply_laced_isomonodromy_systems}, introducing \emph{simply-laced isomonodromy systems} (SLIMS). 
They involve $k$ collections of times, generalising the two collections of times in the JMMS system, and are attached to complete $k$-partite graphs plus some representation theoretic data. 
As a particular case they contain the JMMS system---corresponding to a complete bipartite graph---and further specialising they also include the Schlesinger system and its Harnad dual version---corresponding to a star-shaped graph. 

Moreover, the simply-laced isomonodromy systems extend a particular case of~\cite{jimbo_miwa_ueno_1981_monodromy_preserving_deformations}, which corresponds to complete $k$-partite graphs having at most one splayed node, and with all other nodes being one-dimensional (namely, one term of the ``master'' Equation~8.4 on page 33 of~\cite{boalch_2012_simply_laced_isomonodromy_systems} vanishes in the setup of~\cite{jimbo_miwa_ueno_1981_monodromy_preserving_deformations}). 
In turn, this was one of our motivations to consider the SLIMS: they are more symmetric than~\cite{jimbo_miwa_ueno_1981_monodromy_preserving_deformations}, because one can now permute all the parts, and in future work we plan to study the quantisation of the symmetries of~\cite{boalch_2012_simply_laced_isomonodromy_systems}, which should simultaneously generalise~\cite{nagoya_yamada_2014_symmetries_of_quantum_lax_equations} and the aforementioned~\cite{baumann_1999_q_weyl_group,toledano_laredo_2002_kohno_drinfeld_for_quantum_weyl_groups}. 

The quantisation of the isomonodromic deformation systems that occur at the intersection of the simply-laced isomonodromy systems and~\cite{jimbo_miwa_ueno_1981_monodromy_preserving_deformations} was constructed in~\cite{nagoya_sun_2011_confluent_kz_equations_with_poincare_rank_2_at_infinity}, but this is still far from including all complete $k$-partite graphs since at most one node can be splayed. Now we want to attain the general case without further intermediate steps, and thus we ask: 

\begin{center}
	\textbf{can one quantise the simply-laced isomonodromy systems, generalising the derivations of the KZ and DMT connections as quantisations of isomonodromy systems?} 
\end{center}

In this article we show that this is indeed possible, proving the following.

\begin{thm*}
	There exists a strongly flat nonautonomous quantum Hamiltonian system quantising the simply-laced isomonodromy systems of~\cite{boalch_2012_simply_laced_isomonodromy_systems}. 
	Moreover, after quantum Hamiltonian reduction the quantum system specialises to KZ and to systems which are semiclassically equivalent to DMT and FMTV.
\end{thm*}

Hence in brief we construct a new family of flat connections out of the deformation quantisation of isomonodromy systems on the Riemann sphere, which we call \emph{simply-laced quantum connections} (SLQC), thereby completing the following table: 

\vspace{5pt}

\noindent\adjustbox{max width=\textwidth}{
\begin{tabular}{ c | c  c  c  c }
	\hline \\
	Isomonodromy system & Schlesinger & Dual Schlesinger & JMMS & SLIMS \\[10pt]
	\hline \\
	Space of times & $\mathbb{C}^m \setminus \{\diags\}$ & $\mathbb{C}^n \setminus \{\diags\}$ & $\mathbb{C}^m \setminus \{\diags\} \times \mathbb{C}^n \setminus \{\diags\}$ & $\prod_1^k \mathbb{C}^{d_i} \setminus \{\diags\}$ \\[10pt]
	\hline \\
	Quantum connection & KZ & DMT & FMTV & SLQC \\[10pt]
	\hline 
\end{tabular}}

\subsection*{Layout of the article}

In \S~\ref{sec:classical_systems} we recall the construction of the simply-laced isomonodromy systems.
This involves a trivial symplectic fibration $\mathbb{F}_a = \mathbb{M} \times \mathbf{B} \to \mathbf{B}$, and results in a time-dependent Hamiltonian system $H \colon \mathbb{F}_a \to \mathbb{C}^I$. 
These nonautonomous systems are attached to complete $k$-partite quivers plus some decoration, and $\mathbb{M}$ is a space of representations of such quivers. 

In \S~\ref{sec:classical_potentials} we realise the Hamiltonians as traces of potentials on the quiver (classical potentials), and we study the Poisson bracket of such traces. 

In \S~\ref{sec:quantisation_algebras} we discuss the filtered quantisation of $\mathbb{M}$, which results in a noncommutative filtered algebra $A$: the Weyl algebra. 
We then upgrade this to a $\hslash$-deformation quantisation $\widehat{A}$ via the Rees construction, which comes with a semiclassical limit $\widehat{A} \to A_0$ onto the classical algebra $A_0 = \mathbb{C}[\mathbb{M}]$ of functions on the representation space.  

In \S~\ref{sec:quantisation_potentials} we define quantum potentials, which are related to the quantum algebras $A$ and $\widehat{A}$ similarly to how classical potentials are related to the algebra of functions on $\mathbb{M}$. 

In \S~\ref{sec:quantum_connections} we explain how to quantise the classical Hamiltonians $H_i \colon \mathbf{B} \to A_0$ to time-dependent quantum operators $\widehat{H}_i \colon \mathbf{B} \to \widehat{A}$, thereby defining the \emph{universal} simply-laced quantum connection. 
The simply-laced quantum connection is obtained at the quantum level $\hslash = 1$. 

In \S~\ref{sec:flatness} we prove the main result (Theorem~\ref{thm:quantum_flatness}) that the universal simply-laced quantum connection is strongly flat. 

In \S~\ref{sec:reduction_KZ} we show that the quantum Hamiltonian reduction of the simply-laced quantum connection yields the KZ connection, in the special case of a star-shaped graph with no irregular times. 

In \S~\ref{sec:reduction_DMT} we consider the Harnad-dual data of the previous section, and we show that the quantum Hamiltonian reduction of the simply-laced quantum connection yields a quantum system which is semiclassically equivalent to the DMT connection. 

In \S~\ref{sec:reduction_FMTV} we show that the quantum Hamiltonian reduction of the simply-laced quantum connection is semiclassically equivalent to the FMTV connection, in the case of a generic complete bipartite quiver. 

\vspace{5pt}

All vector spaces, manifold/varieties and algebras are tacitly defined over $\mathbb{C}$. 
All gradings and filtrations of algebras are over $\mathbb{Z}_{\geq 0}$; all filtrations are exhaustive and all algebras are finitely generated. 

The end of remarks/examples is signaled by a $\triangle$.

	\section{Simply-laced isomonodromy systems}
\label{sec:classical_systems}

In this section we define the simply-laced isomonodromy systems. 
They are systems of nonautonomous Hamiltonians attached to complete $k$-partite graphs plus some decoration, which we now introduce. 

Let $J$ be a finite set of cardinality $|J| = k \geq 2$, and $I$ another finite set provided with a surjection $\pi \colon I \twoheadrightarrow J$ with nonempty fibres. 
Write $I = \coprod_{j \in J} I^j$ for the induced partition of $I$, with parts $I^j \coloneqq \pi^{-1}(j)$, and let $\widetilde{\mathcal{G}}$ be the complete graph on nodes $J$.

\begin{defn}
	The complete $k$-partite graph on nodes $I$ is the graph $\mathcal{G}$ in which two nodes $i,j \in I$ are connected by an edge if and only if they lie in different parts of $I$. 
\end{defn}

Hence the adjacency of both graphs is completely determined by their set of nodes. 
Moreover, $\mathcal{G}$ is obtained from $\widetilde{\mathcal{G}}$ by splaying the nodes of the complete graph: one replaces $j$ with the finite set of nodes $I^j$, and connects them to every node outside $I^j$. 
Both $\widetilde{\mathcal{G}}$ and $\mathcal{G}$ are by definition simply-laced, i.e. without edge loops or repeated edges. 
We may equivalently think of these graphs as special types of quivers, by identifying an edge with a pair of opposite arrows. 
We keep the notation $\widetilde{\mathcal{G}}$ and $\mathcal{G}$ for the quivers corresponding to the two graphs, and from now on we work with them.

Now choose representation-theoretic data: attach finite-dimensional vector spaces $\{V_i\}_{i \in I}$ to the nodes of $\mathcal{G}$, and then associate the spaces $W^j \coloneqq \bigoplus_{i \in I^j} V_i$ to the nodes of $\widetilde{\mathcal{G}}$. 
Then, in addition to this usual data, fix an embedding 
\begin{equation}
	a \colon J \hookrightarrow \mathbb{C} \cup \{\infty\}, \quad j \longmapsto a_j \, .
\end{equation}
This assigns different elements of the complex projective line to the nodes of $\widetilde{\mathcal{G}}$.

\begin{defn}
	The embedding $a \colon J \hookrightarrow \mathbb{C} \cup \{\infty\}$ is called a reading of $\widetilde{\mathcal{G}}$. 
	The reading is generic if $\infty \not\in a(J)$, and degenerate otherwise; if the reading is degenerate, we write $\infty \coloneqq a^{-1}(\infty) \in J$ for the node sent to $\infty$. 

	We extend the reading of $\widetilde{\mathcal{G}}$ to a map $a \colon I \to \mathbb{C} \cup \{\infty\}$ by declaring the extension to be constant on each part $I^j$ of $I$, and we call this a reading of $\mathcal{G}$.
\end{defn}

Following~\cite{boalch_2012_simply_laced_isomonodromy_systems}, these data define a space of times
\begin{equation}
	\mathbf{B} \coloneqq \prod_{j \in J} \mathbb{C}^{I^j} \setminus \{\diags\} \subseteq \mathbb{C}^I \, ,
\end{equation}
and a vector space of representations of the quiver $\mathcal{G}$ in the $J$-graded vector space $V \coloneqq \bigoplus_{j \in J} W^j$:
\begin{equation}
	\mathbb{M} \coloneqq \Rep(\mathcal{G},V) = \bigoplus_{i \neq j \in J} \Hom\left(W^i,W^j\right) \, .
\end{equation}

Equivalently, $\mathbb{M} \subseteq \End(V)$ is the subspace of off-diagonal endomorphisms with respect to the block decomposition 
\begin{equation}
	\End(V) = \End\left(\bigoplus_{j \in J} W^j\right) = \bigoplus_{j \in J} \End(W^j) \oplus \bigoplus_{i \neq j \in J} \Hom(W^i,W^j) \, .
\end{equation}

Denote $B^{ij}\colon W^j \to W^i$ the linear maps defined by a representation, and let $X^{ij} = \phi_{ij}B^{ij}$ be the scalar multiplication of $B^{ij}$ by the nonvanishing complex number  
\begin{equation}
	\label{eq:weights}
	\phi_{ij} = -\phi_{ji} \coloneqq 
	\begin{cases}
		(a_i - a_j)^{-1}, & a_i,a_j \neq \infty \\
		1, & \quad a_i = \infty
	\end{cases} \, .
\end{equation}
Then we equip the vector space $\mathbb{M}$ with the symplectic form
\begin{equation}
	\label{eq:symplectic_form}
	\omega_a \coloneqq \frac{1}{2}\sum_{i \neq j \in J} \Tr\big(dX^{ij} \wedge dB^{ji}\big) \, . 
\end{equation}

Now define the space $\mathbb{F}_a$ as the product $\mathbb{F}_a \coloneqq \mathbb{M} \times \mathbf{B}$, where the dependence on the reading $a$ lies in the symplectic form~\eqref{eq:symplectic_form}. 
The canonical projection $\pi_a \colon \mathbb{F}_a \to \mathbf{B}$ makes $\mathbb{F}_a$ into a trivial symplectic fibration over the base $\mathbf{B}$, with fibre $(\mathbb{M},\omega_a)$. 
This symplectic fibration parametrises a family of meromorphic connections on a trivial vector bundle over the Riemann sphere, as follows. 

First, define $U^{\infty} \coloneqq \bigoplus_{j \neq \infty} W^j$ to be the natural complement to $W^{\infty}$ inside $V$, and write a generic endomorphism $\gamma \in \End(V)$ as 
\begin{equation}
	\gamma = 
	\begin{pmatrix}
		T^{\infty} & Q \\
		P & B + T
		\end{pmatrix} \, ,
\end{equation}
using the natural block decomposition for elements of $\End(V) = \End(W^{\infty} \oplus U^{\infty})$. 
This means that $W^{\infty} \in \End(T^{\infty})$ and $B, T \in \End(U^{\infty})$, whereas $Q \in \Hom(W^{\infty},U^{\infty})$ and $P \in \Hom(U^{\infty},W^{\infty})$. 
Moreover, we let $B$ (resp. $T$) be the off-diagonal (resp. diagonal) part of the restriction $\left.\gamma\right|_{U^{\infty}}$. 
Then separating the off-diagonal and the diagonal part of $\gamma$ leads to the global decomposition $\gamma = \Gamma + \widehat{T}$, where
\begin{equation}
	\Gamma \coloneqq 
	\begin{pmatrix}
		0 & P \\
		Q & B
	\end{pmatrix} 
	\in \mathbb{M}, \qquad \text{and} \qquad \widehat{T} \coloneqq 
	\begin{pmatrix}
		T^{\infty} & 0 \\
		0 & T
	\end{pmatrix} 
	\in \bigoplus_{j \in J} \End(W^j) \, .
\end{equation}

Now assume further that $\widehat{T}$ be semisimple, and let $W^j = \bigoplus_{i \in I^j} V_i$ be the eigenspace decomposition of $W^j$ with respect to the restriction $T^j \coloneqq \left.\widehat{T}\right|_{W^j}$. 
Hence the part $I^j \subseteq I$ of the $k$-partite set $I$ is an index set for the spectrum of $T^j$, and one may write 
\begin{equation}
	\widehat{T} = \sum_{i \in I} t_i \Id_i \, ,
\end{equation}
where $\Id_i \in \End(V)$ is the idempotent for $V_i \subseteq V$, and where $\{t_i\}_{i \in I^j} \subseteq \mathbb{C}$ is the spectrum of $T^j$. 

Thus the element $\{t_i\}_{i \in I} \in \mathbf{B}$ of the space of times encodes the spectrum of $\widehat{T}$. 
Numbers in the same part of $I$ are distinct, and varying this element inside $\mathbf{B}$ amounts to an \emph{admissible} deformation of spectrum of $\widehat{T}$, that is a deformation so that the eigenspace decomposition of $W^j$ be fixed for $j \in J$. 
Equivalently, distinct eigenvalues of $T^j$ may not coalesce along the deformation.

Finally, introduce the notation $\Id^j \in \End(V)$ for the idempotent for the subspace $W^j$ inside $V$, and define the endomorphism 
\begin{equation}
	A \coloneqq \sum_{j \neq \infty} a_j\Id^j \in \End(U^{\infty}) \, ,
\end{equation}
using the finite part $I \setminus I^{\infty} \to \mathbb{C}$ of the reading. 
With this notation introduced, we now construct a meromorphic connection on the the trivial vector bundle $U^{\infty} \times \mathbb{C}P^1 \to \mathbb{C}P^1$. 
Take $z$ to be a holomorphic coordinate which identifies $\mathbb{C}P^1 \simeq \mathbb{C} \cup \{\infty\}$, and write $R_i = Q_iP_i \in \End(U^{\infty})$ using the components 
\begin{center}
	\begin{tikzpicture}[node distance = 3cm]
		\node (a) [draw, minimum size = 30pt, circle] at (0,0) {$V_i$};
		\node (b) [draw, circle, , right of = a] {$U^{\infty}$};
		\path
		(a) edge [bend left = 20] node [above]{$Q_i$} (b)
		(b) edge [bend left = 20] node [below]{$P_i$} (a);
	\end{tikzpicture}
\end{center}
of $Q$ and $P$. 
Then we define the meromorphic connection 
\begin{equation}
	\label{eq:meromorphic_connection_slqc}
	\nabla \coloneqq d - \big(Az + B + T + Q(z - T^{\infty})^{-1}P\big)dz = d - \left(Az + B + T + \sum_{i \in I^{\infty}} \frac{R_i}{z - t_i}\right)dz \, .
\end{equation}
This connection has simple poles at the points $\{t_i\}_{i \in I^{\infty}} \subseteq \mathbb{C}$ with residues $R_i$, and a pole of order three at infinity when $A \neq 0$. 
It follows from the definition that simple poles exist only in the case of a degenerate reading, when the part $I^{\infty}$ is nonempty.

We now consider isomonodromic deformations of~\eqref{eq:meromorphic_connection_slqc}. 
This means by definition letting the spectral type of $\widehat{T}$ vary inside $\mathbf{B}$---i.e. varying the positions of the poles and the irregular coefficient $T$ at infinity---and look for a new off-diagonal term $\Gamma$ inside $\mathbb{M}$ such that the monodromy/Stokes data of the resulting meromorphic connection are the same as the starting one. 
Finding such deformations amounts to solving a system of nonlinear, first order differential equations, which are called \emph{isomonodromy equations}. 
The deformation parameters for the spectrum of $T$ are called \emph{irregular times}, since they correspond to deformations of irregular singularities. 
Simple poles instead are regular singularities, and thus the deformation parameters for the spectrum of $T^{\infty}$ are called \emph{regular times} in the isomonodromy literature.

Geometrically, isomonodromic families of connections define the leaves of an integrable nonlinear/Ehresmann symplectic connection inside the symplectic bundle $\pi_a \colon \mathbb{F}_a \to \mathbf{B}$.
Since the bundle is trivial over $\mathbf{B}$, it also carries a trivial Ehresmann connection, which is the pointwise span of the horizontal vector fields $\partial_{t_i}$ (for $i \in I$) associated to the global coordinates on $\mathbf{B}$. 
Importantly, the difference between the two connections can be integrated to a time-dependent Hamiltonian system $H_i \colon \mathbb{F}_a \to \mathbb{C}$. 
More precisely, if one denotes $\{\cdot,\cdot\}$ the Poisson bracket of the symplectic manifold $(\mathbb{M},\omega_a)$, then the (interesting) \emph{isomonodromy connection} is the pointwise span of the vector fields $X_i = \partial_{t_i} + \{H_i,\cdot\}$, where $\{H_i,\cdot\}$ is the vertical Hamiltonian vector field of the fibrewise restriction of $H_i$ to $\mathbb{M}$.

The definition of the Hamiltonians themselves is given by in coordinate-free fashion by constructing a horizontal 1-form $\varpi \in \Omega^0(\mathbb{F}_a,\pi_a^*T^*\mathbf{B})$ on the total space of the fibration. 
To write it down, set $\Xi \coloneqq \phi(\Gamma)$ and $X \coloneqq \phi(B)$, by applying the alternating weights~\eqref{eq:weights} componentwise. 
Then let $\delta(\Xi\Gamma)$ denote the diagonal part of $\Xi\Gamma$ in the decomposition $V = \bigoplus_{j \in J} W^j$, and set
\begin{equation}
	\widetilde{\Xi\Gamma} \coloneqq \ad_{\widehat{T}}^{-1} \bigl[ d\widehat{T},\Xi\Gamma \bigr] \, .
\end{equation}

With this notation introduced, one defines
\begin{equation}
	\label{eq:imd_connection}
	\varpi \coloneqq \frac{1}{2}\Tr\left(\widetilde{\Xi\Gamma}\delta(\Xi\Gamma)\right) - \Tr\left(\Xi\gamma\Xi d\widehat{T}\right) + \Tr\big(X^2TdT\big) + \Tr\big(PAQ T^{\infty}dT^{\infty}\big) \, .
\end{equation}
The Hamiltonians are now defined by $H_i \coloneqq \langle \varpi, \partial_{t_i} \rangle$, which means that $\varpi = \sum_{i \in I} H_idt_i$. 

Importantly, this time-dependent Hamiltonian system satisfies a strong version of integrability. 

\begin{thm}[\cite{boalch_2012_simply_laced_isomonodromy_systems}]
	\label{thm:classical_flatness}
	The isomonodromy system $H_i\colon \mathbb{F}_a \to \mathbb{C}$ is strongly flat, which means that
	\begin{equation}
		\{H_i,H_j\} = 0 = \frac{\partial H_i}{\partial t_j} - \frac{\partial H_j}{\partial t_i}, \qquad \text{for all } i,j \in I \, ,
	\end{equation}
	where $\{\cdot,\cdot\}$ is the symplectic Poisson bracket of $(\mathbb{M},\omega_a)$, computed by fibrewise restriction of the Hamiltonians to the vertical directions.
\end{thm}

We will also think to the Hamiltonian $H_i \colon \mathbb{M} \times \mathbf{B} \to \mathbb{C}$ as the data of a polynomial function on $\mathbb{M}$ for all choice of times in $\mathbf{B}$. 
In this viewpoint the Hamiltonian is a global section of the bundle of commutative algebras $A_0 \times \mathbf{B} \to \mathbf{B}$, where $A_0 \coloneqq \mathscr{O}_{\mathbb{M}}(\mathbb{M}) \simeq \Sym(\mathbb{M^*})$ is the Poisson algebra of regular/polynomial functions on the affine space $\mathbb{M}$, endowed with the above Poisson bracket $\{\cdot,\cdot\}$.

\begin{defn}
	The simply-laced isomonodromy system attached to the complete $k$-partite graph $\mathcal{G}$ on nodes $I$, to the vector spaces $\{V_i\}_{i \in I}$ and to the reading $a \colon I \to \mathbb{C} \cup \{\infty\}$ is the time-dependent Hamiltonian system~\eqref{eq:imd_connection}. The Hamiltonians $H_i$ are called the simply-laced Hamiltonians. 
\end{defn}

This is the classical Hamiltonian system we will quantise. To this end, we first express the Hamiltonians $H_i$ as traces of potentials in $\mathcal{G}$.

\section{Classical potentials}
\label{sec:classical_potentials}

Consider again the complete $k$-partite quiver $\mathcal{G}$ on nodes $I = \coprod_{j \in J} I^j$.

\begin{defn}
	A potential on $\mathcal{G}$ is a $\mathbb{C}$-linear combination of oriented cycles in $\mathcal{G}$, defined up to cyclic permutations of their arrows. The space of potentials is denoted $\mathbb{C}\mathcal{G}_{\cycl}$.
\end{defn}

Let now $\{V_i\}_{i \in I}$ be a family of finite-dimensional vector spaces, set $V \coloneqq \bigoplus_{i \in I} V_i$ and take a reading $a \colon J \hookrightarrow \mathbb{C} \cup \{\infty\}$. 
Then every potential $W \in \mathbb{C}\mathcal{G}_{\cycl}$ defines a polynomial function on $\mathbb{M} = \Rep(\mathcal{G},V)$, by taking the trace of its cycles in any given representation. 
Explicitly, write an oriented cycle of length $n \geq 0$---to be called an $n$-cycle---as $C = \alpha_n \dotsc \alpha_1$, where $\alpha_1, \dotsc, \alpha_n$ is a sequence of composable arrows in $\mathcal{G}$ (reading from right to left in the cycle). 
Then the function $\Tr(C) \colon \mathbb{M} \to \mathbb{C}$ is
\begin{equation}
	\Tr(C) = \Tr\big(X^{\alpha_n} \dotsm X^{\alpha_1}\big) \, ,
\end{equation}
and a time-dependent potential $W \colon \mathbf{B} \to \mathbb{C}\mathcal{G}_{\cycl}$ will define a nonautonomous Hamiltonian $\Tr(W) \colon \mathbf{B} \to A_0$ on $\mathbb{M}$ with space of times $\mathbf{B}$, where $A_0 = \Sym(\mathbb{M}^*)$ as above. 

The main point is that the simply-laced isomonodromy system~\eqref{eq:imd_connection} is made up of the traces of certain potentials on $\mathcal{G}$. 
To write them introduce the notation $I_i \coloneqq \pi^{-1}(\pi(i))$ for the part of $I$ containing the node $i$, and denote $\alpha_{ij}$ the arrow in $\mathcal{G}$ from the node $i$ to the node $j$. 
Then consider the following (possibly time-dependent) potentials on $\mathcal{G}$:
\begin{equation}
	\label{eq:classical_potentials}
	\begin{split}
		W_i(2) &\coloneqq \sum_{j \in I \setminus I_i} (t_i - t_j) \alpha_{ij}\alpha_{ji} \, , \\
		W_i(3) &\coloneqq \sum_{j,l \in I \setminus I_i \colon I_j \neq I_l} (a_j - a_l)\alpha_{il}\alpha_{lj}\alpha_{ji} \, , \\
		W_i(4) &\coloneqq \sum_{m \in I_i \setminus \{i\}} \sum_{j,l \in I \setminus I_i} \frac{(a_i - a_j)(a_i - a_l)}{t_i - t_m}\alpha_{ij}\alpha_{jm}\alpha_{ml}\alpha_{li} \, .
	\end{split}
\end{equation}

In plain words, the potential $W_i(2)$ is the sum of all oriented 2-cycles based at $i$, and thus its second node $j$ lives in a different part of $I$; similarly $W_i(3)$ is a linear combination of all oriented 3-cycles based at $i$, and thus its three nodes $i,j,l$ live in different parts of $I$; and $W_i(4)$ is a linear combination of all oriented 4-cycles based at $i$ which pass through a second node $m \neq i$ in the same part of $i$, and thus their remaining two nodes $j,l$ lie outside the common part of $i$ and $m$. 

\begin{prop}
\label{prop:classical_hamiltonians_are_traces}
	Assume the reading of $\mathcal{G}$ is nondegenerate. Then the simply-laced Hamiltonian $H_i$ is the trace of the time-dependent potential $W_i = W_i(2) + W_i(3) + W_i(4)$ in the variables $X^{ij}$, where $W_i(2), W_i(3)$ and $W_i(4)$ are as in~\eqref{eq:classical_potentials}.

	If the reading is instead degenerate then $H_i$ is the trace of a potential with the same cycles of $W_i$, but with different weights.
\end{prop}

This follows from an explicit expansion of the formula~\eqref{eq:imd_connection}. 
We denote $W_i$ the potential such that $H_i = \Tr(W_i)$ for $i \in I$, both in a degenerate and in a nondegenerate reading.

\begin{defn}
\label{def:isomonodromy_cycles}
	The potential $W_i$ is the isomonodromy potential at the node $i \in I$. The oriented cycles that appear in the isomonodromy potentials (described in words in the above paragraph) are the isomonodromy cycles.

	The isomonodromy 4-cycles are further divided in two types:
	\begin{enumerate}
		\item nondegenerate, if they pass through four distinct nodes of $\mathcal{G}$.
		
		\item degenerate, if they pass through three distinct nodes of $\mathcal{G}$.
	\end{enumerate}
\end{defn}

Figure~\ref{fig:isomonodromy_cycles} shows the isomonodromy cycles.

\begin{figure}[H]
	\begin{center}
		\begin{tikzpicture}
		\vertex (a) at (0,0) {};
		\vertex (b) [right of = a] {};
		\path
		(a) edge [bend left = 20] (b)
		(b) edge [bend left = 20] (a);
		\end{tikzpicture} 
		\qquad 
		\begin{tikzpicture}
			\vertex (a) at (0,0) {};
			\vertex (b) [right of = a] {};
			\vertex (c) at (1,1.7172) {};
			\path
			(a) edge (b)
			(b) edge (c)
			(c) edge (a);
		\end{tikzpicture} 
		\qquad 
		\begin{tikzpicture}
			\vertex (a) at (0,0) {};
			\vertex (b) [right of = a] {};
			\vertex (c) [above of = b] {};
			\vertex (d) [above of = a] {};
			\path
			(a) edge (b)
			(b) edge (c)
			(c) edge (d)
			(d) edge (a);
		\end{tikzpicture} 
		\qquad 
		\begin{tikzpicture}
			\vertex(a) at (0,0) {};
			\vertex (b) [right of = a] {};
			\vertex (c) at (1,1.7172) {};
			\path
			(a) edge [bend left = 20] (c)
			(c) edge [bend left = 20] (a)
			(b) edge [bend left = 20] (c)
			(c) edge [bend left = 20] (b);
		\end{tikzpicture}
		\caption{Isomonodromy cycles. \label{fig:isomonodromy_cycles}}
	\end{center}
\end{figure}

More precisely Figure~\ref{fig:isomonodromy_cycles} depicts 2-cycles, 3-cycles, nondegenerate 4-cycles and finally degenerate 4-cycles, in order from left to right. 
The degenerate 4-cycles can be described as the glueing of two 2-cycles at some node, which we call their \emph{centre}. 
The other two nodes are called \emph{peripheral}, and they lie in one and the same part of $I$---looking at the indices of $W_i(4)$ in~\eqref{eq:classical_potentials}. 
Beware that there are no restriction for 2-cycles and 3-cycles (all such oriented cycles in $\mathcal{G}$ are by definition isomonodromy cycles), but the 4-cycles we consider are only those that appear in the isomonodromy potentials $W_i(4)$ of~\eqref{eq:classical_potentials}. 

\begin{corol}
\label{cor:invariance}
	The simply-laced Hamiltonians $H_i$ are invariant for the natural action of simultaneous change of basis on the spaces $V_i \subseteq V$.
\end{corol}

This follows from Proposition~\ref{prop:classical_hamiltonians_are_traces}, plus the fact that the trace is a class function. 
Thus the group $\widehat{H} \coloneqq \prod_{i \in I} \GL(V_i)$, acting in a Hamiltonian fashion via simultaneous conjugations on $(\mathbb{M},\omega_a)$, preserves $H_i = \Tr(W_i)$.

\begin{rem}
\label{rem:quantisation_local_systems}
	It follows that the simply-laced isomonodromy system descends to a nonautonomous Hamiltonian system defined on the $\widehat{H}$-Hamiltonian reduction of $(\mathbb{M},\omega_a)$ at any coadjoint orbit $\breve{\mathcal{O}} \subseteq \Lie(\widehat{H})^*$. Further, this complex symplectic quotient is isomorphic (as holomorphic symplectic manifold) to the moduli space $\mathcal{M}^*_{\dR}$ of isomorphism classes of meromorphic connections~\eqref{eq:meromorphic_connection_slqc} defined on a trivial holomorphic vector bundle over the Riemann sphere---with local normal form at infinity prescribed by the unramified irregular type $Q = \frac{Az^2}{2} + Tz$; see~\cite[\S 9]{boalch_2012_simply_laced_isomonodromy_systems}.
	
	The moduli spaces depend on a point on the base $\mathbf{B}$, and as one varies it they assemble into a symplectic fibration $\widetilde{\mathcal{M}}^*_{\dR} \to \mathbf{B}$. Then the reduction of the simply-laced isomonodromy system makes this fibration into a local system by defining there a flat (symplectic) Ehresmann connection: the isomonodromy connection, the pull-back of the wild nonabelian Gau\ss{}--Manin connection in the fibration of wild character varieties along the map that monodromy/Stokes data (see~\cite{boalch_2001_symplectic_manifolds_and_isomonodromic_deformations}).

	Hence by quantising the reduced simply-laced isomonodromy system we are in effect quantising the isomonodromy connection, i.e. replacing a certain local system of symplectic manifold with a flat bundle constructed by fibrewise (deformation) quantisation. Some example of reductions will be discussed in \S\S~\ref{sec:reduction_KZ},~\ref{sec:reduction_DMT} and~\ref{sec:reduction_FMTV}.
\end{rem}

Finally, there is a natural Lie bracket $\{\cdot,\cdot\}_{\mathcal{G}}$ on $\mathbb{C}\mathcal{G}_{\cycl}$, given by the necklace Lie algebra structure (see e.g.~\cite{bocklandt_le_bruyn_2002_necklace_lie_algebras, etingov_2007_calogero_moser_systems}).  

\begin{defn}
\label{def:necklace}
	Pick two oriented cycles $C_1, C_2$ in $\mathcal{G}$. The Lie bracket $\{C_1,C_2\}_{\mathcal{G}} \in \mathbb{C}\mathcal{G}_{\cycl}$ is the unique potential such that 
	\begin{equation}
		\Tr\big(\{C_1,C_2\}_{\mathcal{G}}\big) = \{\Tr(C_1),\Tr(C_2)\} \in A_0 \, .
	\end{equation}
\end{defn}

The necklace Lie bracket admits the following description in terms of cutting and glueing of arrows. Write $C_1 = \alpha_n \dotsc \alpha_1$ and $C_2 = \beta_m \dotsc \beta_1$ for two cycles. Then one constructs a new cycle for every pair $(\alpha_i,\beta_j)$ such that $\alpha_i$ is opposite to $\beta_j$, by deleting the pair of arrows and glueing together what is left to obtain a new oriented cycle. 

To see this graphically, fix a pair $i,j$ such that $\alpha_i$ is opposite to $\beta_j$, and introduce the notation $t(\alpha), h(\alpha) \in I$ for the \emph{tail} (starting node) and the \emph{head} (end node) of an arrow $\alpha$ in $\mathcal{G}$, respectively. Set then $a = t(\beta_{j-1})$, $b = h(\beta_{j-1}) = h(\alpha_i)$, $c = h(\beta_j) = h(\alpha_{i-1})$, $d = h(\beta_{j+1})$, $e = t(\alpha_{i-1})$, $f = h(\alpha_{i+1})$, which are all nodes in $\mathcal{G}$. Then the local picture before deleting arrows looks like Figure~\ref{fig:necklace_I}.

\begin{figure}[H]
	\begin{center}
		\begin{tikzpicture}
			\node (a) at (0,0.5) {};
			\vertex (b) [right of = a, label = below:$a$] {};
			\vertex (c) at (2,2) [label = left:$b$] {};
			\vertex (d) [right of = c, label = right:$c$] {};
			\vertex (e) at (4,0.5) [label = below:$d$] {};
			\node (f) [right of = e] {};
			\node (g) at (7,3) {};
			\vertex (h) at (6,2) [label = right:$e$] {};
			\vertex (i) at (4,2.5) [label = above:$c$] {};
			\vertex (j) [left of = i, label = above:$b$] {};
			\vertex (k) [left of = c, label = left:$f$] {};
			\node (l) at (-1,3) {};
			\path
			(a) edge [densely dashed] (b)
			(b) edge node [left]{$\beta_{j-1}$} (c)
			(c) edge node [below]{$\beta_j$} (d)
			(d) edge node [right]{$\beta_{j+1}$} (e)
			(e) edge [densely dashed] (f)
			(g) edge [densely dashed] (h)
			(h) edge node [above]{$\alpha_{i-1}$} (i)
			(i) edge node [above]{$\alpha_i$} (j)
			(j) edge node [above]{$\alpha_{i+1}$} (k)
			(k) edge [densely dashed] (l);
		\end{tikzpicture}
		\caption{Cut-and-glue rule for the necklace Lie bracket: before cutting a pair of opposite arrows. \label{fig:necklace_I}}
	\end{center}
\end{figure}

After cutting and glueing one has instead the situation of Figure~\ref{fig:necklace_II}.

\begin{figure}[H]
	\begin{center}
		\begin{tikzpicture}
			\node (a) at (-1,0.5) {};
			\vertex (b) [right of = a, label = below:$a$] {};
			\vertex (c) at (1,2.5) [label = above:$b$] {};
			\vertex (d) at (-1,2) [label = left:$f$] {};
			\node (e) at (-2,3) {};
			\node (f) at (6,3) {};
			\vertex (g) at (5,2) [label = right:$e$] {};
			\vertex (h) at (3,2.5) [label = above:$c$] {};
			\vertex (i) [below of = h, label = below:$d$] {};
			\node (j) at (5,0.5) {};
			\path
			(a) edge [densely dashed] (b)
			(b) edge node [left]{$\beta_{j-1}$} (c)
			(c) edge node [above]{$\alpha_{i+1}$} (d)
			(d) edge [densely dashed] (e)
			(f) edge [densely dashed] (g)
			(g) edge node [above]{$\alpha_{i-1}$} (h)
			(h) edge node [right]{$\beta_{j+1}$} (i)
			(i) edge [densely dashed] (j);
		\end{tikzpicture}
		\caption{Cut-and-glue rule for the necklace Lie bracket: after cutting a pair of opposite arrows. \label{fig:necklace_II}}
	\end{center}
\end{figure}

One repeats this operation for all the pairs $(\alpha_i,\beta_j)$ of opposite arrows, and then $\{C_1,C_2\}_{\mathcal{G}}$ is obtained by summing these cycles with weights depending on the Poisson bracket of the coordinate functions associated to the arrows that have been cut out (see the computations in Appendix~\ref{sec:appendix}).

\begin{rem}
\label{rem:necklace_bracket}
	Conceptually, what happens is the following. The invariant regular functions on $\mathbb{M}$ for the action of $\widehat{H}$ consist of the Poisson subalgebra $A_0^{\widehat{H}} \subseteq A_0$ generated by traces of oriented cycles (see~\cite[Theorem~1]{le_bruyn_procesi_1990_semisimple_representations_of_quivers}; this is a consequence of Weyl's first fundamental theorem of invariant theory). Hence we have an injective map  $\Tr \colon \mathbb{C}\mathcal{G}_{\cycl} \hookrightarrow A_0^{\widehat{H}}$, and Definition~\ref{def:necklace} says that it is a morphism of Lie algebras. Furthermore if one upgrades the necklace Lie bracket to a Poisson bracket on $\Sym(\mathbb{C}\mathcal{G}_{\cycl})$---by enforcing the Leibniz rule---then the trace becomes an isomorphism of Poisson algebras. This is a restatement of the fact that the $\widehat{H}$-invariant functions are polynomials of traces of oriented cycles.
	
	We will present a quantum counterpart of this in \S~\ref{sec:quantisation_potentials}, independently from (but possibly similarly to)~\cite{schedler_2005_a_hopf_algebra_quantizing_a_necklace_lie_algebra}. In op. cit. a more refined Hopf algebra quantisation is considered, whereas in this paper coproducts are not considered.
\end{rem}

\section{Quantisation of Poisson algebras}
\label{sec:quantisation_algebras}

Consider again the commutative graded Poisson algebra $A_0 = \Sym(\mathbb{M}^*)$ of algebraic functions on $\mathbb{M}$. 
In this section we recall how to define a filtered quantisation of $A_0$, which is a particular instance of deformation quantisation. 
All this material is standard, and thus we omit proofs (see e.g.~\cite{etingov_schiffmann_1998_lectures_on_quantum_groups, schedler_2012_deformations_of_algebras}).

\subsection{Formal deformation quantisation}

Let $\hslash$ be a formal variable, and $(B,\{\cdot,\cdot\})$ a commutative Poisson algebra.

\begin{defn}
\label{def:deformation_quantisation_algebras}
	A one-parameter formal deformation quantisation of $B$ is a topologically free $\mathbb{C}\llbracket\hslash\rrbracket$-algebra $\widehat{A}$, together with an isomorphism $\widehat{A} \big\slash\hslash \widehat{A} \simeq B$, such that
	\begin{equation}
	\{f_0,g_0\} = \frac{1}{\hslash^d} \big(fg - gf\big) + \hslash \widehat{A}, 
	\end{equation}
	for arbitrary lifts $f,g \in \widehat{A}$ of $f_0, g_0 \in B$, and for a suitable integer $d \geq 1$. 
	The canonical projection $\sigma \colon \widehat{A} \to B$ is called the semiclassical limit.
\end{defn}

Definition~\ref{def:deformation_quantisation_algebras} is the type of quantisation we look for, but in our case there is a simplification: it is possible to construct a formal deformation quantisation by applying the Rees construction to a filtered quantisation.  

\subsection{Filtered quantisation and the Rees construction}

Consider a filtered associative algebra $A = \bigcup_{k \geq 0} A_{\leq k}$. 
Recall that the associated graded $\gr(A)$ of $A$ is the graded associative algebra with graded components $\gr(A)_k = A_{\leq k} \big\slash A_{\leq k - 1}$, for $k \geq 0$---where $A_{\leq -1} \coloneqq (0)$. 
The product of $\gr(A)$ is induced from that of $A$ by defining it on representatives. 

Let now $B$ be a graded associative algebra---not necessarily Poisson or commutative at this stage.

\begin{defn}
	A filtered deformation of $B$ is a filtered associative algebra $A$ together with an isomorphism $\gr(A) \simeq B$ of graded associative algebras.
\end{defn}

Analogously to Definition~\ref{def:deformation_quantisation_algebras}, the associative product of a filtered algebra induces a Poisson bracket on the associated graded, provided the latter is commutative. 
To state this fact denote by $\sigma_p \colon A \to \gr(A)_p$ the canonical projections for $p \geq 0$.

\begin{prop}
	Let $A = \bigcup_{k \geq 0} A_{\leq k}$ be a filtered associative algebra such that $\gr(A)$ is commutative. 
	Then there exists a maximal integer $d \geq 1$ such that $\bigl[A_{\leq k},A_{\leq l}\bigr] \subseteq A_{\leq k + l - d}$ for all $k,l \geq 0$. Furthermore, there is a canonical Poisson bracket $\{\cdot,\cdot\}_A$ defined on $\gr(A)$ by the formula:
	\begin{equation}
		\label{eq:filtered_quantisation}
		\{\sigma_k(f),\sigma_l(g)\}_A = \sigma_{k+l-d}\big(fg - gf\big) \, ,
	\end{equation}
	where $f \in A_{\leq k}$ and $g \in A_{\leq l}$.
\end{prop}

Finally, analogously to Definition~\ref{def:deformation_quantisation_algebras}, a filtered quantisation is a filtered deformation plus a compatibility of Poisson brackets. 
Let then $(B,\{\cdot,\cdot\})$ be a graded commutative Poisson algebra. 

\begin{defn}
	\label{def:filtered_quantisation_algebras}
	A filtered quantisation of $(B,\{\cdot,\cdot\})$ is a filtered deformation $A$ of $B$ together with an isomorphism $\gr(A) \simeq B$ of graded Poisson algebras.
\end{defn}

There is now a universal construction to pass from a filtered quantisation as in Definition~\ref{def:filtered_quantisation_algebras} to a formal deformation quantisation as in Definition~\ref{def:deformation_quantisation_algebras}. 
Let again $A$ be a filtered associative algebra, and $\hslash$ a formal variable.

\begin{defn}
	\label{def:rees_algebra}
	The Rees algebra $\Rees(A)$ of $A$ is the $\mathbb{C}[\hslash]$-algebra defined by
	\begin{equation}
		\begin{split}
			\Rees &(A) \coloneqq \bigoplus_{k \geq 0} A_{\leq k} \cdot \hslash^k \subseteq A[\hslash] \, .
		\end{split}
	\end{equation}
\end{defn}

\begin{prop}
	\label{prop:rees}
	Assume that $A$ is a filtered quantisation of $(B,\{\cdot,\cdot\})$. 
	Then the following $\mathbb{C}\llbracket \hslash \rrbracket$-algebra defines a (one-parameter) formal deformation quantisation of the commutative Poisson algebra $(B,\{\cdot,\cdot\})$ (forgetting the grading):
	\begin{equation}
		\widehat{A} \coloneqq \Set{\sum_{k \geq 0} f_k\hslash^k | f_k \in A_{\leq k} \text{ for all } k \geq 0, \lim_{k \longrightarrow +\infty} \left(k - \left|f_k\right|\right) = +\infty} \subseteq A\llbracket\hslash\rrbracket \, ,
	\end{equation}
	where the nonnegative integer $\left|f\right| \coloneqq \min \{k \geq 0 \mid f \in A_{\leq k}\}$ is the order of the element $f \in A$.
\end{prop}

The proof consists in showing that the map 
\begin{equation}
	\label{eq:semiclassical_limit}
	\sigma \colon \sum_{k \geq 0} f_k \hslash^k \longmapsto \sum_{k \geq 0} \sigma_k\left(f_k\right)
\end{equation}
is surjective on $B$ with kernel $\hslash \widehat{A} \subseteq \widehat{A}$, hence induces an isomorphism $\widehat{A} \big\slash \hslash\widehat{A} \simeq B$---noting that the right-hand side of~\eqref{eq:semiclassical_limit} is a finite sum. 
Then the compatibility with the Poisson bracket, in the sense of Definition~\ref{def:deformation_quantisation_algebras}---with $d = 2$---follows from~\eqref{eq:filtered_quantisation}. 
The morphism~\eqref{eq:semiclassical_limit} is precisely the semiclassical limit $\sigma \colon \widehat{A} \to B$ of Definition~\ref{def:deformation_quantisation_algebras}.

\begin{rem}
\label{rem:algebraic_quantisation}
	Power series are considered to recover the general setting of formal deformation quantisation as in Definition~\ref{def:deformation_quantisation_algebras}. 
	In the filtered case however one can work with polynomials, so in particular the deformation parameter $\hslash$ can take on numerical values.
\end{rem}

Hence in brief to deformation-quantise a graded commutative Poisson algebra $B$ according to Definition~\ref{def:deformation_quantisation_algebras} one may look for a filtered quantisation $A$ as in Definition~\ref{def:filtered_quantisation_algebras}.
Then replacing this with (a completion of) its Rees algebra will recover the deformation parameter $\hslash$. 
Enforcing the quantum mechanical terminology one may refer to the elements of $A$ as \emph{quantum operators}.

\subsection{The Weyl algebra}
\label{sec:weyl}

We now define a filtered quantisation of the algebra of polynomial functions on a complex symplectic vector space $(V,\omega)$---seen as an affine space over $\mathbb{C}$. 
Denote $\{\cdot,\cdot\}$ the Poisson bracket on $B \coloneqq \Sym(V^*)$ associated to the symplectic form.

\begin{defn}
	\label{def:weyl_algebra}
	We define 
	\begin{equation}
		A = W \bigl( V^*,\{\cdot,\cdot\} \bigr) \coloneqq \Tens(V^*) \big\slash I \, ,
	\end{equation}
	where $\Tens(V^*)$ is the tensor algebra of the vector space $V^*$, and $I \subseteq \Tens(V^*)$ is the two-sided ideal generated by elements of the form
	\begin{equation}
		f \otimes g - g \otimes f - \{f,g\}, \qquad \text{for} \qquad f,g \in V^* \, .
	\end{equation}
\end{defn}

The noncommutative algebra $A$ inherits the quotient filtration of $\Tens(V^*)$ induced from the natural grading (the additive/Bernstein filtration). 
Then one can show that~\eqref{eq:filtered_quantisation} holds with $d = 2$, and that there is a canonical isomorphism $\gr(A) \simeq B$ of graded commutative Poisson algebras---where the symmetric algebra has the standard grading, see e.g.~\cite{schedler_2012_deformations_of_algebras}.
Hence $A$ is a filtered quantisation of $B$.

\begin{rem}
	There is also a dual realisation of $A$, called the \emph{Weyl algebra} of the symplectic vector space $(V,\omega)$ and noted $W(V,\omega)$.
	Analogously to $A$, it is defined as the quotient of $\Tens(V)$ with respect to the two-sided ideal generated by the elements 
	\begin{equation}
		x \otimes y - y \otimes x - \omega(x,y), \qquad \text{for} \qquad x,y \in V \, .
	\end{equation}
	Strictly speaking, the Weyl algebra $W(V,\omega)$ is a filtered deformation of $\Sym(V)$ rather then $\Sym(V^*)$, and that's why we gave the dual realisation in Definition~\ref{def:weyl_algebra}.
	Nonetheless, the linear isomorphism $V \simeq V^*$ induced by the symplectic pairing tautologically provides an isomorphism $A \simeq W(V,\omega)$ of filtered associative algebras, and thus we will not distinguish among the two. 
\end{rem}

There is a natural embedding $\iota \colon V^* \hookrightarrow A$ obtained from the composition of the inclusion $V^* \hookrightarrow \Tens(V^*)$ with the canonical projection $\Tens(V^*) \to A$, and we write $\widehat{f} \coloneqq \iota(f)$ for the image of a linear function in this natural embedding.
Hereafter we want to interpret $\iota$ as a quantisation of linear functions on $V$, and to do this we define the semiclassical limit following the prescription of \S~\ref{sec:quantisation_algebras}. 

Let $\Rees(A)$ be the Rees algebra of $A$ as in definition~\ref{def:rees_algebra}, and $\widehat{A}$ its power series version as in Proposition~\ref{prop:rees}. 
The arrow $\iota$ is tautologically upgraded to an embedding $\widehat{\iota} \colon V^* \hookrightarrow \Rees(A) \subseteq \widehat{A}$ via $\widehat{\iota}(f) \coloneqq \widehat{f} \cdot \hslash$. Moreover, by definition the semiclassical limit $\sigma \colon \widehat{A} \to B$ of~\eqref{eq:semiclassical_limit} is a left inverse to $\widehat{\iota}$.

\begin{defn}
	\label{def:quantisation_functions}
	A quantisation of a function $f \in B$ is an element of the affine subspace $\sigma^{-1}(f) \subseteq \widehat{A}$.
\end{defn}

In particular the element $\widehat{f} \cdot \hslash$ is a quantisation of the linear function $f \in V^*$, and $\iota$ is given by the composition
\begin{center}
	\begin{tikzpicture}[> = to]
		\node (a) at (-0.25,0) {$V^*$};
		\node (b) at (2,0) {$\Rees(A)$};
		\node (c) at (5,0) {$A$};
		\node (e) at (5.25,-0.15) {$\mathbb{,}$};
		\path
		(a) edge node [above] {$\widehat{\iota}$} (b)
		(b) edge node [above] {$\ev_{\hslash = 1}$} (c);
	\end{tikzpicture}
\end{center}
where $\ev_{\hslash = 1}$ is the evaluation of polynomials at $\hslash = 1$---the quotient modulo the ideal generated by $\hslash - 1$.

\begin{defn}
	\label{def:canonical_quantisation_weyl}
	The element $\iota(f) = \widehat{f} \in A$ is the Weyl quantisation of the linear function $f \in V^*$.
\end{defn}

This provides a presentation of the Weyl algebra as soon as a basis of $V$ is chosen. 
Namely, if $X_1, \dotsc, X_m \colon V \to \mathbb{C}$ are linear coordinate functions associated to a basis of $V$ then their Weyl quantisations generate $A$ as a $\mathbb{C}$-algebra. 
More precisely, $A$ is isomorphic to the quotient of the free algebra generated by the symbols $\widehat{X}_i$ with commutation relations 
\begin{equation}
	\bigl[\widehat{X}_i,\widehat{X}_j\bigr] = \{X_i,X_j\} \in A_{\leq 0} = \mathbb{C} \, .
\end{equation}
Moreover, if $n \geq 0$ and $i_1, \dotsc, i_n \in \{1,\dotsc,m\}$ then the product $\prod_{j = 1}^n \widehat{X}_{i_j} \in A_{\leq n}$ has order equal to $n$. 

\begin{rem} 
	Note that by Definition~\ref{def:quantisation_functions} the element $\prod_{j = 1}^n \widehat{X}_{i_j} \cdot \hslash^n \in \widehat{A}$ is a quantisation of the monomial $f = \prod_{j = 1}^n X_{i_j} \in B$, but this is not canonically attached to $f$ as the order for the (classical) variables is immaterial.
	More explicitly, if $\Sigma_n$ is the symmetric group on $n$ objects and $\tau \in \Sigma_n$ a permutation then one can consider the element
	\begin{equation}
		\prod_{j = 1}^n \widehat{X}_{i_{\tau(j)}} \cdot \hslash^n \in \widehat{A} \, ,
	\end{equation}
	which is also a quantisation of $f$, and the lack of a natural way to pick one ordering results in no-go obstructions to quantisation à la Groenewold--Van Hove. 

	If one tries to bypass this issue by considering the symmetrisation map 
	\begin{equation}
		f \longmapsto \frac{1}{n!} \sum_{\tau \in \Sigma_n} \prod_{j = 1}^n \widehat{X}_{i_{\tau(j)}} \cdot \hslash^n \, ,
	\end{equation}
	which a priori solves the problem of choosing an ordering, then the crucial property of flatness of Theorem~\ref{thm:classical_flatness} is not preserved (as the ``classical'' Poisson bracket and the ``quantum'' commutator do not match up).
\end{rem}

In the next section we will describe a procedure to quantise the simply-laced isomonodromy systems~\eqref{eq:imd_connection} by preserving (strong) flatness, overcoming the aforementioned obstructions.

\section{Quantisation of potentials}
\label{sec:quantisation_potentials}

We now discuss how to quantise the traces of the isomonodromy cycles of Definition~\ref{def:isomonodromy_cycles}, which leads to a quantisation of the simply-laced isomonodromy system~\eqref{eq:imd_connection}. 

Start by applying all the material of the previous \S~\ref{sec:quantisation_algebras} to the situation of the simply-laced isomonodromy systems. 
Let $\mathcal{G}$ be the $k$-complete graph on nodes $I = \coprod_{j \in J} I^j$, $\{V_i\}_{i \in I}$ a family of finite-dimensional vector spaces and $a \colon J \hookrightarrow \mathbb{C} \cup \{\infty\}$ a reading.
Define the symplectic vector space $(\mathbb{M},\omega_a)$ as in \S~\ref{sec:classical_systems}. 
Then the filtered quantisation $A = W(\mathbb{M}^*,\{\cdot,\cdot\})$ of the classical algebra $A_0 = \Sym(\mathbb{M}^*)$ is defined in~\ref{def:weyl_algebra}---where $\{\cdot,\cdot\}$ is the symplectic Poisson bracket---together with the semiclassical limit $\sigma \colon \widehat{A} \to A_0$ of~\eqref{eq:semiclassical_limit}, whose domain is the deformation quantisation of $A_0$ as in Definition~\ref{def:deformation_quantisation_algebras}).

To get an explicit description of $A$ we may do as instructed above. 
Choose bases for the spaces $V_i$ for all $i \in I$, and consider the coordinate functions $X \longmapsto X^{\alpha}_{ij}$ on $\mathbb{M}$, where $\alpha$ runs through the arrows of $\mathcal{G}$ and $i,j$ are indices for the coefficients of the matrix $X^{\alpha} \colon V_{t(\alpha)} \to V_{h(\alpha)}$ in the given bases. 
Then a generic element of $A$ is a polynomial in the Weyl quantisations $\widehat{X}^{\alpha}_{ij}$.

We now push the interpretation of the classical Hamiltonians as traces of potentials one step further: we attach quantum operators to decorated oriented cycles, in the same way in which (invariant) functions on $\mathbb{M}$ are attached to ordinary oriented cycles. 

\begin{defn}
	An anchored cycle $\widehat{C}$ is an oriented cycle in $\mathcal{G}$ with a starting arrow fixed, called the anchor of $\widehat{C}$. 
	This will be denoted by underlining the anchor:
	\begin{equation}
		\widehat{C} = \alpha_n \dotsm \underline{\alpha_1} \, ,
	\end{equation}
	where $\alpha_1, \dotsc, \alpha_n$ are composable arrows in $\mathcal{G}$. 
\end{defn}

Now let $C = \alpha_n \dotsm \alpha_1$, be an oriented cycle, and consider its trace
\begin{equation}
	\Tr(C) = \sum_{k_1,\dotsc,k_n} X^{\alpha_n}_{k_nk_{n-1}} \dotsm X^{\alpha_1}_{k_1k_n} \in A_0 \, ,
\end{equation}
where $k_1,\dotsc,k_n$ are suitable indices. 
If one anchors $C$ at $\alpha_1$, i.e. if one considers the anchored cycle $\widehat{C} \coloneqq \alpha_n \dotsm \underline{\alpha_1}$, then the element
\begin{equation}
	\sum_{k_1,\dotsc,k_n} \widehat{X}^{\alpha_n}_{k_nk_{n-1}} \dotsm \widehat{X}^{\alpha_1}_{k_1k_n} \in A
\end{equation}
is uniquely determined. 
In turn we define this element to be the trace of the anchored cycle $\widehat{C}$.

\begin{defn}
	The trace of the anchored cycle $\widehat{C} = \alpha_n \dotsm \underline{\alpha_1}$ is 
	\begin{equation}
		\Tr(\widehat{C}) \coloneqq \sum_{k_1,\dotsc,k_n} \widehat{X}^{\alpha_n}_{k_nk_{n-1}} \dotsm \widehat{X}^{\alpha_1}_{k_1k_n} \in A \, .
	\end{equation}
	The $\hslash$-deformed trace of the anchored cycle $\widehat{C}$ is $\Tr_{\hslash}(\widehat{C}) \coloneqq \Tr(\widehat{C}) \cdot \hslash^n \in \Rees(A) \subseteq \widehat{A}$, where $\hslash$ is a formal variable.
\end{defn}

\begin{rem}
	More intrinsically, consider the $A$-valued matrix 
	\begin{equation}
		\widehat{X}^{\alpha_l} \dotsm \widehat{X}^{\alpha_1} \in A \otimes \End(V_i) \, ,
	\end{equation}
	where $i \coloneqq t(\alpha_1) \in I$ is the tail of the anchor of $\widehat{C}$, i.e. the starting node of the anchored cycle. 
	The coefficients of the matrices $\widehat{X}^{\alpha_l}, \dotsc, \widehat{X}^{\alpha_1}$ are Weyl quantisations of coordinate functions on $\mathbb{M}$, and taking a trace amounts to contracting $V_i$ against $V_i^*$. 
	Finally we multiply the result by the correct power of $\hslash$, i.e. by the length of the cycle.
\end{rem}

Importantly, two different anchored cycles may define the same element of the Weyl algebra. 
This happens when their two underlying cycles coincide under a cyclic permutation of their arrows in which no arrow is swapped with its opposite, because all the coefficients of the matrix $\widehat{X}^{\alpha}$ commute with all those of $\widehat{X}^{\beta}$ if and only if $\alpha$ and $\beta$ are not opposite arrows. 

This observation motivates the next definitions.

\begin{defn}
	Let $\widehat{C} = \alpha_n \dotsm \underline{\alpha_1}$ be an anchored cycle in $\mathcal{G}$. 
	An admissible permutation of the arrows of $\widehat{C}$ consists in dividing the word $\alpha_n \dotsm \alpha_1$ in two subwords 
	\begin{equation}
		A = \alpha_n \dotsm \alpha_{n-i}, \qquad B = \alpha_{n-i-1} \dotsm \alpha_1
	\end{equation}
	such that no arrow in $A$ has its opposite in $B$, and in swapping $A$ and $B$. 
	This yields a new anchored cycle $\widehat{C}' = \alpha_{n-i-1} \dotsm \alpha_1 \alpha_l \dotsm \underline{\alpha_{n-i}}$ which is said to be equivalent to $\widehat{C}$.
\end{defn}

\begin{defn}
	A quantum cycle in $\mathcal{G}$ is an anchored cycle defined up to admissible permutations of its arrows. We define $\widehat{\mathbb{C}\mathcal{G}}_{\cycl}$ to be the complex vector space spanned by quantum cycles in $\mathcal{G}$, and we call its elements quantum potentials. 
	Furthermore, we denote $\sigma_{\mathcal{G}} \colon \widehat{\mathbb{C}\mathcal{G}}_{\cycl} \to \mathbb{C}\mathcal{G}_{\cycl}$ the map that forgets the anchor, and we say that a quantum potential $\widehat{W}$ is a quantisation of the potential $W$ if $\sigma_{\mathcal{G}}(\widehat{W}) = W$. 

	The length $l(C) \geq 0$ of a quantum cycle $\widehat{C}$ is defined as that of its underlying oriented cycle $C = \sigma_{\mathcal{G}}(\widehat{C})$.
\end{defn}

There exists now a well defined linear map $\Tr_{\hslash} \colon \widehat{\mathbb{C}\mathcal{G}}_{\cycl} \to \widehat{A}$ landing inside $\Rees(A) \subseteq \widehat{A}$, together with a square which is commutative by construction: 
\begin{center}
	\begin{tikzpicture}[> = to]
		\node (a) at (0,0) {$\mathbb{C}\mathcal{G}_{\cycl}$};
		\node (b) at (3,0) {$A_0$};
		\node (c) [above of = a] {$\widehat{\mathbb{C}\mathcal{G}}_{\cycl}$};
		\node (d) [above of = b] {$\widehat{A}$};
		\node (i) at (1.5,1) {$\circlearrowleft$};
		\path
		(a) edge node [below]{$\Tr$} (b)
		(c) edge node [left]{$\sigma_{\mathcal{G}}$} (a)
		(c) edge node [above]{$\Tr_{\hslash}$} (d)
		(d) edge node [right]{$\sigma$} (b);
	\end{tikzpicture}
\end{center}

Hence the semiclassical limit of $\Tr_{\hslash}(\widehat{C}) \in \widehat{A}$ equals $\Tr(C) \in A_0$ whenever $\sigma_{\mathcal{G}}(\widehat{C}) = C$, and thus one may quantise the cycle $C$ to quantise the function $\Tr(C)$. 
In turn, quantising a cycle means by definition constructing a quantum potential which projects back to it by forgetting anchors.

This is what we now do for the isomonodromy cycles of Figure~\ref{fig:isomonodromy_cycles}. 

\begin{defn}
	\label{def:quantisation_imd_potentials}
	The quantisation of a 3-cycle and a nondegenerate 4-cycles is the quantum cycle obtained from any choice of anchoring.\footnote{Since 3-cycles and nondegenerate 4-cycles do not contain pairs of opposite arrows, the actual choice of anchor is immaterial: all the resulting anchored cycles will be equivalent, i.e. they will define the same quantum cycle inside $\widehat{\mathbb{C}\mathcal{G}}_{\cycl}$.}~The quantisation of a degenerate 4-cycles is the quantum cycle obtained by anchoring at either of the two arrows coming out of the centre. The quantisation of a 2-cycle $C = $~
	\begin{tikzpicture}
		\vertex (a) at (0,0) {};
		\vertex (b) at (1,0) {};
		\path
		(a) edge [bend left = 20] (b)
		(b) edge [bend left = 20] (a);
	\end{tikzpicture} 
	is the following combination of quantum 2-cycles, where we draw the tails of the anchors as black nodes: 
	\begin{figure}[H]
		\begin{center}
			$\widehat{C} = \frac{1}{2}\Bigg($\begin{tikzpicture}
				\vertex (a) at (0,0) [fill = black] {};
				\vertex (b) at (1,0) {};
				\path
				(a) edge [bend left = 20] (b)
				(b) edge [bend left = 20] (a);
				\end{tikzpicture} $+$ \begin{tikzpicture}
					\vertex (a) at (0,0) {};
					\vertex (b) at (1,0) [fill = black] {};
					\path
					(a) edge [bend left = 20] (b)
					(b) edge [bend left = 20] (a);
					
				\end{tikzpicture}$\Bigg)$
			\caption{Quantisation of 2-cycles. \label{fig:quantisation_2_cycle}}
		\end{center}
	\end{figure}
	
	The quantum cycles thus defined in $\mathcal{G}$ are called the \emph{quantum isomonodromy cycles}.
\end{defn}

Figure~\ref{fig:quantum_isomonodromy_cycles} shows the quantum isomonodromy cycles, where as above (and hereafter) we draw the tails of the anchors as black nodes.

\begin{figure}[H]
	\begin{center}
		\begin{tikzpicture}
			\vertex (a) at (0,0) [fill = black] {};
			\vertex (b) [right of = a] {};
			\path
			(a) edge [bend left = 20] (b)
			(b) edge [bend left = 20] (a);
		\end{tikzpicture} 
		\qquad 
		\begin{tikzpicture}
			\vertex (a) at (0,0) [fill = black] {};
			\vertex (b) [right of = a] {};
			\vertex (c) at (1,1.7172) {};
			\path
			(a) edge (b)
			(b) edge (c)
			(c) edge (a);
		\end{tikzpicture} 
		\qquad 
		\begin{tikzpicture}
			\vertex (a) at (0,0) [fill = black] {};
			\vertex (b) [right of = a] {};
			\vertex (c) [above of = b] {};
			\vertex (d) [above of = a] {};
			\path
			(a) edge (b)
			(b) edge (c)
			(c) edge (d)
			(d) edge (a);
		\end{tikzpicture} \qquad 
		\begin{tikzpicture}
			\vertex(a) at (0,0) {};
			\vertex (b) [right of = a] {};
			\vertex (c) at (1,1.7172) [fill = black] {};
			\path
			(a) edge [bend left = 20] (c)
			(c) edge [bend left = 20] (a)
			(b) edge [bend left = 20] (c)
			(c) edge [bend left = 20] (b);
		\end{tikzpicture}
		\caption{Quantum isomonodromy cycles. \label{fig:quantum_isomonodromy_cycles}}
	\end{center}
\end{figure}

From left to right one sees quantum 2-cycles, quantum 3-cycles, nondegenerate quantum 4-cycles and degenerate quantum 4-cycles. 
This should be compared with Figure~\ref{fig:isomonodromy_cycles}, which shows the (classical) isomonodromy cycles. 

\begin{rem}
	The quantisation of degenerate 4-cycles is well defined, since changing the order of the arrows coming out of the central node amounts to an admissible permutation of the arrows. 

	To see this explicitly, denote again $\gamma^*$ the opposite of any arrow $\gamma$ in $\mathcal{G}$. Then one may write a degenerate 4-cycle as $C = \beta^*\beta\alpha^*\alpha$, where $\alpha,\beta$ are the two distinct arrows of $C$ coming out of the centre. 
	The two possible anchors at the centre yield $\widehat{C}_1 = \beta^*\beta\alpha^*\underline{\alpha}$ and $\widehat{C}_2 = \alpha^*\alpha\beta^*\underline{\beta}$, which are equivalent under the admissible permutation swapping the two 2-cycles $\alpha^*\alpha$ and $\beta^*\beta$: doing this does not change the relative order of the opposite pairs $(\alpha,\alpha^*), (\beta,\beta^*)$. 
\end{rem}

\section{Universal simply-laced quantum connection}
\label{sec:quantum_connections}
    
Consider again the isomonodromy potentials $W_i \in \mathbb{C}\mathcal{G}_{\cycl}$ such that $H_i = \Tr(W_i)$ is the simply-laced Hamiltonian at the node $i \in I$.	

\begin{defn}
	\label{def:quantum_Hamiltonians}
	The quantum isomonodromy potential $\widehat{W}_i \colon \mathbf{B} \to \widehat{\mathbb{C}\mathcal{G}}_{\cycl}$ at the node $i \in I$ is defined by quantising all the isomonodromy cycles of $W_i$ according to Definition~\ref{def:quantisation_imd_potentials}. 
	The universal simply-laced quantum Hamiltonian $\widehat{H}_i \colon \mathbf{B} \to \widehat{A}$ is the $\hslash$-deformed trace of the quantum isomonodromy potential at the node $i \in I$.
\end{defn}

By construction $\widehat{H}_i = \Tr_{\hslash}(\widehat{W}_i) \colon \mathbf{B} \to \widehat{A}$ is a quantisation of the simply-laced Hamiltonian $H_i \colon \mathbf{B} \to A_0$ according to Definition~\ref{def:quantisation_functions}, because the identity $\sigma(\widehat{H}_i) = H_i$ is true pointwise on $\mathbf{B}$. 

\begin{defn}
	\label{def:slqc}
	The universal simply-laced quantum connection $\widehat{\nabla}$ is the connection on the trivial bundle of noncommutative algebras $\widehat{A} \times \mathbf{B} \to \mathbf{B}$ defined by
	\begin{equation}
		\widehat{\nabla} \coloneqq d - \widehat{\varpi}, \qquad \text{where} \qquad  \widehat{\varpi} \coloneqq \sum_{i \in I} \widehat{H}_i dt_i \in \Omega^1(\mathbf{B},\widehat{A}) \, ,
	\end{equation}
	letting $\widehat{A}$ act on itself by left multiplication.
\end{defn}

The main result of this paper is the following.

\begin{thm}
\label{thm:quantum_flatness}
	The universal simply-laced quantum connection is strongly flat, i.e.
	\begin{equation}
		\bigl[\widehat{H}_i,\widehat{H}_j\bigr] = 0 = \frac{\partial \widehat{H}_i}{\partial t_j} - \frac{\partial \widehat{H}_j}{\partial t_i}, \qquad \text{for all} \qquad i,j \in I \, .
	\end{equation}
\end{thm}

The connection of Definition~\ref{def:slqc} is universal since one may replace the regular representation by any complex left $\widehat{A}$-module $\rho \colon \widehat{A} \to \End_{\mathbb{C}}(\mathcal{H})$, and thus obtain a new connection 
\begin{equation}
	\widehat{\nabla}_{\rho} = d - \sum_{i \in I} \rho(\widehat{H}_i) dt_i \, ,
\end{equation}
on the vector bundle $\mathcal{H} \times \mathbf{B} \to \mathbf{B}$. Because of Theorem~\ref{thm:quantum_flatness} all connections obtained from $\widehat{\nabla}$ in this way will be strongly flat. 

Since the quantum Hamiltonians live in $\Rees(A) \subseteq \widehat{A}$, an important class of modules arises from left modules for the Weyl algebra $A$, letting the Rees algebra act via the evaluation morphism $\ev_{\hslash = 1} \colon \Rees(A) \to A$. 
A particular important example is obtained by considering the regular representation of the Weyl algebra on itself.

\begin{defn}
\label{def:level_1_slqc}
	The simply-laced quantum connection is the connection on the trivial bundle of noncommutative algebras $A \times \mathbf{B} \to \mathbf{B}$ induced from the universal simply-laced quantum connection from the representation $\rho_1 \colon \Rees(A) \to \End_{\mathbb{C}}(A)$ defined by
	\begin{equation}
		\rho_1 \left(\sum_k f_k \hslash^k\right) (X) \coloneqq \left(\sum_k f_k \right) \cdot X \, .
	\end{equation}
	The functions $\rho_1(\widehat{H}_i) = \Tr(\widehat{W}_i)$ are the simply-laced quantum Hamiltonians.
\end{defn}

The simply-laced quantum connection will be later related to the KZ Hamiltonians (in \S~\ref{sec:reduction_KZ}), the DMT Hamiltonians (in \S~\ref{sec:reduction_DMT}) and the FMTV Hamiltonians (in \S~\ref{sec:reduction_FMTV}).

\begin{rem}
	For the sake of a geometric example consider a Lagrangian splitting $\mathbb{M} \simeq T^*M$, where $M \coloneqq \Rep(\mathcal{Q},V)$ is the space of representation of a subquiver $\mathcal{Q} \subseteq \mathcal{G}$ containing the arrows with a given positive orientation. 
	Taking the real part of the symplectic from produces a real cotangent bundle, and one may take $\mathcal{H} \coloneqq C^{\infty}L^2(M,\mathbb{C})$ to be the space of smooth square-summable complex functions on $M$ (with respect to the Lebesgue measure on $M$), on which $A$ acts with the standard \emph{Schr\"odinger position representation}: if a set of real Darboux coordinates $(q_j,p_j)$ is chosen on $T^*M$, then the position $q_j$ acts via the function multiplication $\mu_{q_j}$, and the momentum $p_j$ via the derivative $-i\partial_{q_j}$.

	In the language of geometric quantisation this means considering the prequantisable symplectic manifold $(T^*M,\sum_j dq_j \wedge dp_j)$, and taking prequantum data consisting of the trivial complex line bundle $T^*M \times \mathbb{C} \to T^*M$ equipped with the tautological Hermitian metric and the prequantum connection $\nabla = d - i\sum_j q_jdp_j$ defined by the Liouville potential of the exact symplectic form. 
	Then one performs geometric quantisation (at quantum level $\hslash = 1$) with respect to the real polarisation defined by the vertical cotangent fibres. 
\end{rem}

The next section is dedicated to the proof of Theorem~\ref{thm:quantum_flatness}.

\section{Proof of strong flatness}
\label{sec:flatness}

\subsection{Time-dependent term}
\label{sec:flatness_I}

We show here that $d \widehat{\varpi} = 0$. 

\begin{prop}
	\label{prop:quantum_flatness_I}
	One has:
	\begin{equation}
		\partial_{t_i}\widehat{H}_j - \partial_{t_j}\widehat{H}_i = 0, \qquad \text{for all } i,j \in I \, .
	\end{equation}
\end{prop}

This follows from a lemma.

\begin{lem}
	\label{lem:flatness_I}
	Let $W_i \colon \mathbf{B} \to \mathbb{C}\mathcal{G}_{\cycl}$ be a classical isomonodromy potential. Then
	\begin{equation}
		\partial_{t_j}\widehat{W}_i = \widehat{\partial_{t_j}W_i}, \qquad \text{and} \qquad \partial_{t_j}\Tr_{\hslash}\big(\widehat{W}_i\big) = \Tr_{\hslash}\big(\partial_{t_j}\widehat{W}_i\big) \, ,
	\end{equation}
	for all $i,j \in I$.
\end{lem}

\begin{proof}[Proof of Lemma~\ref{lem:flatness_I}]
	These identities follow from construction: the quantisation of classical cycles and the $\hslash$-deformed trace of quantum cycles do not depend on $\mathbf{B}$. 
	Note that taking derivatives does not change the type of cycles that make up the isomonodromy potential $W_i$, but only modifies the coefficients of the linear combination; hence the quantisation $\widehat{\partial_{t_j}W_i}$ of $\partial_{t_j}W_i$ is well defined in the sense of Definition~\ref{def:quantisation_imd_potentials}. 
\end{proof}

\begin{proof}[Proof of Proposition~\ref{prop:quantum_flatness_I}]
	Using the second set of identities of Lemma~\ref{lem:flatness_I}, it is enough to verify that one has $\partial_{t_i}\widehat{W_j} - \partial_{t_j}\widehat{W_i} = 0$ for all $i,j \in I$, because the $\hslash$-deformed trace of the left-hand side is precisely the difference $\partial_{t_i}\widehat{H}_j - \partial_{t_j}\widehat{H}_i$. 
	To prove this we appeal to Theorem~\ref{thm:classical_flatness}, borrowing the identity $\partial_{t_i}H_j = \partial_{t_j}H_i$, which is equivalent to $\partial_{t_i}W_j = \partial_{t_j}W_i$ because $\Tr \colon \mathbb{C}\mathcal{G}_{\cycl} \to A_0$ is injective. 
	This implies
	\begin{equation}
		\widehat{\partial_{t_i}W_j} = \widehat{\partial_{t_j}W_i} \, ,
	\end{equation}
	and the first set of identity of Lemma~\ref{lem:flatness_I} concludes the proof.
\end{proof}

\subsection{Vanishing commutators}
\label{sec:flatness_II}

We now show here that $\bigl[\widehat{\varpi},\widehat{\varpi}\bigr] = 0$, i.e. that the universal simply-laced quantum Hamiltonians commute. 
By bilinearity, this reduces to the problem of computing commutators of the form
\begin{equation}
	\bigl[\Tr_{\hslash}(\widehat{C}_1),\Tr_{\hslash}(\widehat{C}_2)\bigr] = \bigl[\Tr(\widehat{C}_1),\Tr(\widehat{C}_2)\bigr] \cdot \hslash^{l(C_1) + l(C_2)} \in \widehat{A} \, ,
\end{equation}
where $\widehat{C}_1, \widehat{C}_2$ are quantum isomonodromy cycles of lengths $l(C_1)$ and $l(C_2)$ respectively. 

To understand this we focus on the element $\bigl[\Tr(\widehat{C}_1),\Tr(\widehat{C}_2)\bigr]$ inside the Weyl algebra, and check that it can still be expressed as the trace of a quantum potential.
Importantly, to show this one can no longer use the commutativity of the associative product, but nonetheless this nontrivial property holds for quantum isomonodromy cycles.

\begin{prop}
	\label{prop:quantum_commutators}
	If $\widehat{C}_1$ and $\widehat{C}_2$ are quantum isomonodromy cycles then there exists a quantum potential $\widehat{W} \in \widehat{\mathbb{C}\mathcal{G}}_{\cycl}$ such that 
	\begin{equation}
		\bigl[\Tr(\widehat{C}_1),\Tr(\widehat{C}_2)\bigr] = \Tr(\widehat{W}) \in A \, .
	\end{equation}
	Moreover, $\widehat{W}$ can be found so that it is obtained from a suitable anchoring of all cycles of the Necklace Lie bracket 
	\begin{equation}
		\big\{\sigma_{\mathcal{G}}(\widehat{C}_1),\sigma_{\mathcal{G}}(\widehat{C}_2)\big\}_{\mathcal{G}} \in \mathbb{C}\mathcal{G}_{\cycl} \, .
	\end{equation}
\end{prop}

The quantum potential $\widehat{W}$ in the statement is then uniquely determined by $\widehat{C}_1$ and $\widehat{C}_2$ as an element of $\widehat{\mathbb{C}\mathcal{G}}_{\cycl}$, since we consider it up to admissible permutations of its arrows.

\begin{defn}
	The commutator of the quantum isomonodromy cycles $\widehat{C}_1$ and $\widehat{C}_2$ is the quantum potential $\bigl[\widehat{C}_1,\widehat{C}_2\bigr] \in \widehat{\mathbb{C}\mathcal{G}}_{\cycl}$ defined by Proposition~\ref{prop:quantum_commutators}. 
	It satisfies
	\begin{equation}
		\bigl[\Tr(\widehat{C}_1),\Tr(\widehat{C}_2)\bigr] = \Tr\Big(\bigl[\widehat{C}_1,\widehat{C}_2\bigr]\Big) \, .
	\end{equation}
\end{defn}

Because of Proposition~\ref{prop:quantum_commutators} the commutator is a quantisation of the necklace Lie bracket of the underlying cycles, i.e. 
\begin{equation}
	\sigma_{\mathcal{G}} \big(\bigl[\widehat{C}_1,\widehat{C}_2\bigr]\big) = \big\{\sigma_{\mathcal{G}}(\widehat{C_1}),\sigma_{\mathcal{G}}(\widehat{C_2})\big\}_{\mathcal{G}} \, .
\end{equation}
This identity should be compared with~\eqref{eq:filtered_quantisation}, replacing degrees with lengths of cycles. 
This shows how the construction of quantum potentials points towards a filtered quantisation of $\Sym(\mathbb{C}\mathcal{G}_{\cycl})$ (cf. Remark~\ref{rem:necklace_bracket}).

The proof of Proposition~\ref{prop:quantum_commutators} breaks up into three Lemmas, whose proofs have been postponed to the Appendix~\ref{sec:appendix}.

\begin{lem}
	\label{lem:quantum_commutators}
	Pick two quantum cycles $\widehat{C}_1$, $\widehat{C}_2$, with underlying cycles $C_1$, $C_2$. Assume that one of $\widehat{C}_1$, $\widehat{C}_2$ is a 2-cycle, or that one of them does not contain pairs of opposite arrows. Then Proposition~\ref{prop:quantum_commutators} holds.
\end{lem}

This includes all the commutators of quantum isomonodromy cycles, except if both are degenerate quantum 4-cycles. 
Hence we now address that case, decomposing it into two subcases: either the degenerate 4-cycles have different centres, or their centres are one and the same node of $\mathcal{G}$.

\begin{rem}
	As before, all forthcoming pictures sketch a ``local'' situation on the graph $\mathcal{G}$, meaning that we only draw the nodes and arrows involved in the commutators and forget about the rest of $\mathcal{G}$. 
	These sketches summarise longer computations in noncommutative variables, which are explicitly given in the Appendix~\ref{sec:appendix}. 
	The simplicity of these figures should be contrasted with the complexity of the explicit computations; indeed simplifying such computations was one of our main motivations for developing this cycle-theoretic calculus.

	Secondarily, a choice of Darboux coordinates on $(\mathbb{M},\omega_a)$ can be made in order to simplify all formul\ae{}: doing so one can assume that all coefficients in the necklace Lie bracket either equal $+1$ or $-1$, and the sign can be fixed by taking an orientation for $\mathcal{G}$.
\end{rem}

The first case we consider is the commutator of two degenerate quantum 4-cycles with different centres.

\begin{lem}
	\label{lem:vanishing_commutator_degenerate_4_cycles}
	Pick distinct nodes $a,b,c,d \in I$ so that the sequences of nodes $(a,b,a,c)$ and $(a,c,d,c)$ define two degenerate quantum 4-cycles with centres $a$ and $c$---respectively. Then their commutator vanishes, as shown in figure~\ref{fig:vanishing_commutator_degenerate_4_cycles}.

	\begin{figure}[H]
		\begin{center}
			\begin{tikzpicture}
				\draw (0.5,0) -- (0,0) -- (0,4) -- (0.5,4) [-];
				\vertex (a) at (1,1) [label = below:$a$, fill = black] {};
				\vertex (b) [right of = a, label = below:$b$] {};
				\vertex (c) [above of = a, label = above:$c$] {};
				\vertex (d) at (4,1) [label = below:$a$] {};
				\vertex (e) [above of = d, label = above:$c$, fill = black] {};
				\vertex (f) [right of = e, label = above:$d$] {};
				\node (g) at (3.5,1) {\textbf{,}};
				\draw (6.5,0) -- (7,0) -- (7,4) -- (6.5,4) [-];
				\draw (7.7,1.9) -- (8.2,1.9) [-]; 
				\draw (7.7,2.1) -- (8.2,2.1) [-];	
				\draw (9,2) ellipse (5pt and 6.5pt);
				\path
				(a) edge [bend left = 20] (b)
				(b) edge [bend left = 20] (a)
				(a) edge [bend left = 20] (c)
				(c) edge [bend left = 20] (a)
				(e) edge [ bend left = 20] (d)
				(d) edge [bend left = 20] (e)
				(e) edge [bend left = 20] (f)
				(f) edge [bend left = 20] (e);
			\end{tikzpicture}
			\caption{Vanishing commutator of two degenerate quantum 4-cycles centred at different nodes. \label{fig:vanishing_commutator_degenerate_4_cycles}}
		\end{center}
	\end{figure}
\end{lem}

The commutator is indeed obtained as by a suitable anchoring of the necklace Lie bracket, since this necklace Lie bracket also vanishes. 
Hence Proposition~\ref{prop:quantum_commutators} is verified in this case.

The remaining situation is that in which the centres of the degenerate quantum 4-cycles coincide, and exactly one peripheral node is in common.

\begin{lem}
	\label{lem:commutator_degenerate_4_cycles}
	Pick two degenerate quantum 4-cycles centred at $j \in I$, with exactly one peripheral node in common. Then one can choose an orientation of $\mathcal{G}$ so that their commutator equals the quantum potential shown in Figure~\ref{fig:commutator_degenerate_4_cycles}.

	\begin{figure}[H]
		\begin{center}
			\begin{tikzpicture}
				\draw (0.5,0) -- (0,0) -- (0,4) -- (0.5,4) [-];
				\vertex (a) at (1,1) {};
				\vertex (b) at (2,2) [label=right:$j$, fill = black] {};
				\vertex (c) at (2,3.4) {};
				\vertex (d) [right of = c] {};
				\vertex (e) at (4,2) [label = left:$j$, fill = black] {};
				\vertex (f) at (5,1) {};
				\draw (5.5,0) -- (6,0) -- (6,4) -- (5.5,4) [-];
				\draw (6.8,1.9) -- (7.3,1.9) [-]; 
				\draw (6.8,2.1) -- (7.3,2.1) [-];
				\vertex (g) at (8,1) [label = left:$3$] {};
				\vertex (h) at (9,2) [fill = black] {};
				\vertex (i) at (9,3.4) [label = above:$2$] {};
				\vertex (j) at (10,1) [label = right:$1$] {};
				\draw (10.8,2) -- (11.2,2) [-];
				\vertex (k) at (12,1) [label = left:$2$] {};
				\vertex (l) at (13,2) [fill = black] {};
				\vertex (m) at (13,3.4) [label = above:$3$] {};
				\vertex (n) at (14,1) [label = right:$1$] {};
				\node (o) at (3,2) {\textbf{,}};
				\path
				(a) edge [bend left = 20] (b)
				(b) edge [bend left = 20] (a)
				(c) edge [bend left = 20] (b)
				(b) edge [bend left = 20] (c)
				(d) edge [bend left = 20] (e)
				(e) edge [bend left = 20] (d)
				(f) edge [bend left = 20] (e)
				(e) edge [bend left = 20] (f)
				(g) edge [bend left = 20] (h)
				(h) edge [bend left = 20] (g)
				(i) edge [bend left = 20] (h)
				(h) edge [bend left = 20] (i)
				(j) edge [bend left = 20] (h)
				(h) edge [bend left = 20] (j)
				(k) edge [bend left = 20] (l)
				(l) edge [bend left = 20] (k)
				(m) edge [bend left = 20] (l)
				(l) edge [bend left = 20] (m)
				(n) edge [bend left = 20] (l)
				(l) edge [bend left = 20] (n);
			\end{tikzpicture}
			\caption{Commutator of two degenerate quantum 4-cycles with centres in common. \label{fig:commutator_degenerate_4_cycles}}
		\end{center}
	\end{figure}

	The numbers at the peripheral nodes in Figure~\ref{fig:commutator_degenerate_4_cycles} indicate the order in which they are touched in the cycle, starting from the centre.
\end{lem}

The cut-and-glue rule for the necklace Lie bracket then shows that Proposition~\ref{prop:quantum_commutators} is verified in this last case.

Now we conclude the proof of the main Theorem~\ref{thm:quantum_flatness}.

\begin{prop}
	\label{prop:quantum_flatness_II}
	One has $\bigl[\widehat{H}_i,\widehat{H}_j\bigr] = 0$ for all $i, j \in I$.
\end{prop}

\begin{proof}
	Note that if $\widehat{C}_1$ and $\widehat{C}_2$ are quantum isomonodromy cycles then one has 
	\begin{equation}
		\Tr_{\hslash} \Big(\bigl[\widehat{C}_1,\widehat{C}_2\bigr]\Big) = \Tr\Big(\bigl[\widehat{C}_1,\widehat{C}_2\bigr]\Big) \cdot \hslash^{l(C_1) + l(C_2) - 2} \, ,
	\end{equation}
	where $l(C_k)$ is the length of $\widehat{C}_k$ for $k = 1,2$. 
	Indeed, by Proposition~\ref{prop:quantum_commutators} all quantum cycles in the commutator have length $l(C_1) + l(C_2) - 2$, since this is the length of the cycles in the necklace Lie bracket $\big\{\sigma_{\mathcal{G}}(\widehat{C}_1),\sigma_{\mathcal{G}}(\widehat{C}_2)\big\}_{\mathcal{G}}$. 
	Hence 
	\begin{equation}
		\Tr_{\hslash} \Big(\bigl[\widehat{C}_1,\widehat{C}_2\bigr]\Big) \cdot \hslash^2 = \bigl[\Tr_{\hslash}(\widehat{C}_1),\Tr_{\hslash}(\widehat{C}_2)\bigr] \, ,
	\end{equation}
	and by bilinearity this yields
	\begin{equation}
		\Tr_{\hslash} \Big(\bigl[\widehat{W}_i,\widehat{W}_j\bigr]\Big) \cdot \hslash^2 = \bigl[\Tr_{\hslash}\big(\widehat{W}_i\big),\Tr_{\hslash}\big(\widehat{W}_j\big)\bigr] = \bigl[\widehat{H}_i,\widehat{H}_j\bigr] \in \widehat{A} \, .
	\end{equation}
	Thus it is enough to show that $\bigl[\widehat{W}_i,\widehat{W}_j\bigr] = 0$ as quantum potential, for all $i,j \in I$. 
	Moreover, in the quantum potential $\widehat{W}_i$ we can replace all quantisations of 2-cycles (as shown in Figure~\ref{fig:quantisation_2_cycle}) with a single quantum 2-cycle anchored at the node $i$, since doing this amounts to adding a constant, i.e. a central term. 
	Finally, in all computations of quantum commutators one is free to move the anchor of 3-cycles and nondegenerate 4-cycles, since they are all equivalent.

	Start then by decomposing the classical isomonodromy potentials $W_i, W_j$ into a linear combination of isomonodromy cycles, so that $W_i = \sum_k c_k C_k$ and $W_j = \sum_l d_l D_l$.
	Expanding their necklace Lie bracket by bilinearity yields a linear combination of potentials:
	\begin{equation}
	\label{eq:sum_of_potentials_classic}
		0 = \{W_i,W_j\}_{\mathcal{G}} = \sum_{k,l} c_kd_l \{C_k,D_l\}_{\mathcal{G}} \, .
	\end{equation}
	Putting together all equal cycles on the right-hand side of~\eqref{eq:sum_of_potentials_classic} then results in a vanishing linear combination of cycles:
	\begin{equation}
		0 = \{W_i,W_j\}_{\mathcal{G}} = \sum_m e_m E_m \in \mathbb{C}\mathcal{G}_{\cycl} \, ,
	\end{equation}
	where now all the $E_m$ are distinct. 
	Then one has necessarily $e_m = 0$ for all $m$, since any finite family of distinct cycles in $\mathcal{G}$ is by definition free inside $\mathbb{C}\mathcal{G}_{\cycl}$. 

	On the quantum side one finds an analogous expansion:
	\begin{equation}
		\label{eq:sum_of_potentials_quantum}
		\bigl[\widehat{W}_i,\widehat{W}_j\bigr] = \sum_{k,l} c_kd_l \bigl[\widehat{C}_k,\widehat{D}_l\bigr] \, ,
	\end{equation}
	with $\widehat{C}_k$ (resp. $\widehat{D}_l$) being the quantisations of $C_k$ (resp. of $D_l$), according to Definition~\ref{def:quantisation_imd_potentials}---with the aforementioned modification for the quantisation of 2-cycles. 
	Moreover, thanks to Proposition~\ref{prop:quantum_commutators} one knows that $\bigl[\widehat{C}_k,\widehat{D}_l\bigr]$ is obtained by suitably anchoring all the cycles of $\{C_k,D_l\}_{\mathcal{G}}$. 
	One would now like to have 
	\begin{equation}
		\bigl[\widehat{W}_i,\widehat{W}_j\bigr] = \sum_m e_m \widehat{E}_m \in \widehat{\mathbb{C}\mathcal{G}}_{\cycl} \, ,
	\end{equation}
	with the same constants $e_m \in \mathbb{C}$ and for some quantisation $\widehat{E}_m$ of $E_m$.
	This happens if the following holds: whenever two cycles are equal on the right hand-side of~\eqref{eq:sum_of_potentials_classic}, then their corresponding quantum cycles on the right-hand side of~\eqref{eq:sum_of_potentials_quantum} are also equal, i.e. they are anchored cycles with equivalent anchors. 

	A systematic way of showing this is to list all cycles coming out of the necklace Lie bracket of two isomonodromy cycles (Figure~\ref{fig:isomonodromy_cycles}) having opposite arrows in common. 
	For the resulting cycles which do not contain a pair of opposite arrows there is nothing to show, so we proceed by cataloguing all cycles which do contain such a pair, dividing them into two lists. 

	The first list appears in Figure~\ref{fig:leftover_cycles_1}.

	\begin{figure}[H]
		\begin{center}
			\begin{tikzpicture}
				\vertex (a) at (0,0) {};
				\vertex (b) [right of = a] {};
				\vertex (g) [above of = b] {};
				\vertex (d) [above of = a] {};
				\path 
				(a) edge (b)
				(b) edge (g)
				(g) edge (a) 
				(a) edge [bend left = 20] (d)
				(a) edge [<-, bend right = 20] (d);
			\end{tikzpicture} \qquad \begin{tikzpicture}
				\vertex (a) at (0,0) {};
				\vertex (b) [right of = a] {};
				\vertex (c) [above of = a] {};
				\vertex (d) [right of = c] {};
				\vertex (e) at (1,1.4) {};
				\path 
				(b) edge [bend left = 20] (e)
				(b) edge [<-, bend right = 20] (e)
				(b) edge (d) 
				(d) edge (c)
				(c) edge (a)
				(a) edge (b);
			\end{tikzpicture} \qquad \begin{tikzpicture}
				\vertex (a) at (0,0) {};
				\vertex (b) [right of = a] {};
				\vertex (c) [above of = a] {};
				\vertex (d) [right of = c] {};
				\path 
				(a) edge [bend left = 20] (c)
				(a) edge [<-, bend right = 20] (c)
				(b) edge (d) 
				(d) edge (c)
				(c) edge (a)
				(a) edge (b);
			\end{tikzpicture}
			\caption{Cycles with pairs of opposite arrows, arising from the necklace Lie bracket of isomonodromy cycles: first list. \label{fig:leftover_cycles_1}}
		\end{center}
	\end{figure}

	The cycles of Figure~\ref{fig:leftover_cycles_1} can be described as the glueing of two subcycles at some node, which we call their centre. 
	One can then compute in noncommutative variables (cf. Appendix~\ref{sec:appendix}) to show that such cycles can always be anchored at their central node after computing commutators, up to changing anchors for 3-cycles and nondegenerate 4-cycles. 
	Then the two leftmost cycles of figure~\ref{fig:leftover_cycles_1} correspond to equal quantum cycles on the right-hand side of~\eqref{eq:sum_of_potentials_quantum}, as it was to be shown. 
	
	The rightmost cycle of Figure~\ref{fig:leftover_cycles_1} requires further care, since the two anchors at its centre give different quantum 6-cycles; nonetheless, up to changing the anchor of nondegenerate 4-cycles one can show that the 2-subcycle always comes first after computing commutators, and thus such quantum cycles are equal.

	The remaining cycles which contain opposite arrows are shown in Figure~\ref{fig:leftover_cycles_2}.

	\begin{figure}[H]
		\begin{center}
			\begin{tikzpicture}
				\vertex (a) at (0,0) {};
				\vertex (b) [right of = a] {};
				\vertex (c) [above right of = a] {};
				\path 
				(a) edge [bend left = 20] (b)
				(a) edge [<-, bend right = 20] (b)
				(a) edge [bend left = 20] (c)
				(a) edge [<-, bend right = 20] (c);
			\end{tikzpicture} \qquad \begin{tikzpicture}
				\vertex (a) at (0,0) {};
				\vertex (b) [right of = a] {};
				\vertex (c) [above of = a] {};
				\vertex (d) [above of = b] {};
				\path 
				(a) edge [bend left = 20] (b)
				(a) edge [<-, bend right = 20] (b)
				(a) edge [bend left = 20] (c)
				(a) edge [<-, bend right = 20] (c)
				(d) edge [bend left = 20] (c)
				(d) edge [<-, bend right = 20] (c);
			\end{tikzpicture} \qquad \begin{tikzpicture} 
				\vertex (g) at (8,1) {};
				\vertex (h) at (9,2) {};
				\vertex (i) at (9,3.4) {};
				\vertex (j) at (10,1) {};
				\path
				(g) edge [bend left = 20] (h)
				(h) edge [bend left = 20] (g)
				(i) edge [bend left = 20] (h)
				(h) edge [bend left = 20] (i)
				(j) edge [bend left = 20] (h)
				(h) edge [bend left = 20] (j);
			\end{tikzpicture}
			\caption{Cycles with pairs of opposite arrows, arising from the necklace Lie brackets of isomonodromy cycles: second list. \label{fig:leftover_cycles_2}}
		\end{center}
	\end{figure}

	The two leftmost cycles of Figure~\ref{fig:leftover_cycles_2} are dealt with by a uniqueness argument. 
	Namely, the cycle $C =$ 
	\begin{tikzpicture}
		\vertex (a) at (0,0) {};
		\vertex (b) at (1,0) {};
		\vertex (c) at (0.7,0.7) {};
		\path 
		(a) edge [bend left = 20] (b)
		(a) edge [<-, bend right = 20] (b)
		(a) edge [bend left = 20] (c)
		(a) edge [<-, bend right = 20] (c);
	\end{tikzpicture} 
	only arises from the necklace Lie bracket of two 3-cycles $C_1$ and $C_2$ having opposite orientations. 
	The cycles $C_1$ and $C_2$ are uniquely determined by the three nodes of $C$, and thus on the right-hand side of~\eqref{eq:sum_of_potentials_classic} the cycle $C$ appears exactly twice with opposite coefficients---otherwise the sum would not vanish.
	But then up to changing the anchors of $C_1$ and $C_2$ (which are all equivalent) the same happens on the right-hand side of~\eqref{eq:sum_of_potentials_quantum}, because the quantisation does not tamper with the coefficients of 3-cycles. 
	
	Similarly, the 6-cycle $C' = $
	\begin{tikzpicture}
		\vertex (a) at (0,0) {};
		\vertex (b) at (1,0) {};
		\vertex (c) at (0,1) {};
		\vertex (d) at (1,1){};
		\path 
		(a) edge [bend left = 20] (b)
		(a) edge [<-, bend right = 20] (b)
		(a) edge [bend left = 20] (c)
		(a) edge [<-, bend right = 20] (c)
		(d) edge [bend left = 20] (c)
		(d) edge [<-, bend right = 20] (c);
	\end{tikzpicture} 
	only arises as the necklace Lie bracket of two nondegenerate 4-cycles with opposite orientations, which are moreover uniquely determined by the nodes of $C'$, and the same argument applies.
	
	Finally we have to consider the rightmost cycle of Figure~\ref{fig:leftover_cycles_2}, i.e. the 6-cycle coming out the commutator of two degenerate 4-cycles with centre in common, as depicted in Lemma~\ref{lem:commutator_degenerate_4_cycles}. 
	The existence of a common centre implies that $i$ and $j$ lie in the same part of $I$, which we denote $\overline{I} \coloneqq I_i = I_j$. 
	Let then $k \in I$ be the common centre and consider the complete bipartite subgraph $\mathcal{G}' \subseteq \mathcal{G}$ on nodes $I' = \{k\} \coprod \overline{I}$. 
	If one defines $V' \coloneqq V_k \oplus \bigoplus_{l \in \overline{I}} V_l$, then a choice of Darboux coordinates on the representation space 
	\begin{equation}
		\Rep(\mathcal{G}',V') \simeq T^*\left(\bigoplus_{l \in \overline{I}} \Hom(V_l,V_k)\right)
	\end{equation}
	turns this situation into that of \S~\ref{sec:reduction_KZ}, where we consider the simply-laced quantum connection of a star-shaped quiver. 
	Hence the proof of Proposition~\ref{prop:flatness_SLQC_star} concludes the proof of Proposition~\ref{prop:quantum_flatness_II}.
\end{proof}

\begin{rem}
	If one denotes $\widehat{C}_{ij}$ the degenerate quantum 4-cycle with centre $k$ that passes through the nodes $i \neq j \in \overline{I}$, then one can show that 
	\begin{equation}
		\bigl[\widehat{C}_{ij},\widehat{C}_{il} + \widehat{C}_{jl}\bigr] = 0                                                                                                                                                                   \end{equation}
	for all distinct nodes $i,j,l \in \overline{I}$. 
	We see these relations as a lift of Kohno's relations for the flatness of the Knizhnik--Zamolodchikov connection, that is 
	\begin{equation}
		\bigl[\widehat{\Omega}^{(ij)},\widehat{\Omega}^{(il)} + \widehat{\Omega}^{(jl)}\bigr] = 0 \, ,
	\end{equation}
	where the operators $\widehat{\Omega}^{(ij)}$ are defined in \S~\ref{sec:KZ}. 
\end{rem}

This remark prompts us to explaining why the simply-laced quantum connection is a lifted version of some important quantum connections including Knizhnik--Zamolodchikov.

	\section{The KZ connection and the star}
\label{sec:reduction_KZ}

In this section we show that the Knizhnik--Zamolodchikov connection (KZ)~\cite{knizhnik_zamolodchikov_1984_wess_zumino_witten} is a reduction of the simply-laced quantum connection for the degenerate reading of a star with no irregular times attached. 
By definition a \emph{star} (or a \emph{star-shaped graph}) is a complete bipartite graph having one part with a single node, called the \emph{centre}; the other nodes are \emph{peripheral}.

\subsection{Simply-laced quantum connection of a star}
\label{sec:slqc_star}

We start by specialising the general construction of \S\S~\ref{sec:classical_systems} and~\ref{sec:quantum_connections} to define the universal simply-laced quantum connection of a star with no irregular times. 

Take the set $J$ to have cardinality $k = 2$, with degenerate reading 
\begin{equation}
	a \colon J \longmapsto a(J) = \{\infty,0\} \subseteq \mathbb{C} \cup \{\infty\} \, .
\end{equation}
This means that the highest irregular coefficient $A$ of the meromorphic connections~\eqref{eq:meromorphic_connection_slqc} vanishes, and that $T = T^0$. 
We further assume that $T^0 = 0$, so that~\eqref{eq:meromorphic_connection_slqc} becomes a logarithmic connection. 
The base space of times is now $\mathbf{B} = \mathbb{C}^m \setminus \{\diags\}$, i.e. the space of variations of the positions of $m \geq 1$ simple poles in $\mathbb{C}$.

If we identify $J$ with $a(J)$, then the part $I^0 \subseteq I = I^0 \coprod I^{\infty}$ within the set of nodes $I$ of the complete bipartite graph $\mathcal{G}$ is a singleton, and we denote $0$ its only element.
Hence $\mathcal{G}$ is a star with centre $0$, and with $m = |I^{\infty}|$ peripheral nodes, such as in Figure~\ref{fig:KZ_DMT}. 

\begin{figure}[H]
	\begin{center}
		\begin{tikzpicture}
			\vertex (a) at (0,0) {};
			\vertex (b) [above of = a] {};
			\vertex (c) [right of = a] {};
			\vertex (d) [left of = a] {};
			\vertex (e) [below of = a] {};
			\path
			(a) edge [-] (b)
			(a) edge [-] (c)
			(a) edge [-] (d)
			(a) edge [-] (e);
		\end{tikzpicture}
		\caption{A star-shaped graph (here with $m = 4$ peripheral nodes). \label{fig:KZ_DMT}}
	\end{center}
\end{figure}

Attach finite dimensional vector spaces $\{V_0,V_i\}_{i \in I^{\infty}}$ to the nodes of $\mathcal{G}$, and set $W^{\infty} = \bigoplus_{i \in I^{\infty}} V_i$, $U^{\infty} = W^0 = V_0$, and $V = W^{\infty} \oplus W^0$. 
The space of representation is
\begin{equation}
	\mathbb{M} = \Rep(\mathcal{G},V) = \Hom(W^{\infty},W^0) \oplus \Hom(W^0,W^{\infty}) \, ,
\end{equation}
and comes equipped with the symplectic form $\omega_a = \Tr(dQ \wedge dP)$, where $Q \colon W^{\infty} \longrightarrow W^0$ and $P \colon W^0 \longrightarrow W^{\infty}$. 
We write $Q_i$ for the component of $Q$ in $V_i^* \otimes W^0$, and $P_i$ for the component of $P$ in $(W^0)^* \otimes V_i$. 

As explained above, these data parametrise particular examples of the meromorphic connections~\eqref{eq:meromorphic_connection_slqc}, namely the logarithmic connections
\begin{equation}
	\label{eq:kz_meromorphic_connection}
	\nabla = d - \big(Q(z - T^{\infty})^{-1}P\big)dz = d - \sum_{i \in I^{\infty}} \frac{Q_iP_i}{z - t_i}dz \, ,
\end{equation}
on the trivial vector bundle $W^0 \times \mathbb{C} P^1 \longrightarrow \mathbb{C} P^1$, where $\{t_i\}_{i \in I^{\infty}} \in \mathbf{B}$, and the off-diagonal term $\Gamma = 
\begin{pmatrix} 
	& Q \\
	P & 
\end{pmatrix}
\in \mathbb{M}$ encodes the residues.

The isomonodromic deformations of the Fuchsian systems~\eqref{eq:kz_meromorphic_connection} are controlled by the simply-laced Hamiltonian system~\eqref{eq:imd_connection}. 
If one sets $\widetilde{PQ} = \ad_{T^{\infty}}^{-1}\bigl[dT^{\infty},PQ\bigr]$, then it becomes
\begin{equation}
	\varpi = \frac{1}{2}\Tr\big(\widetilde{PQ}PQ\big) \, .
\end{equation}
This is a nonautonomous Hamiltonian system on the trivial symplectic fibration $\mathbb{M} \times \mathbf{B} \longrightarrow \mathbf{B}$, which spells out as
\begin{equation}
	\label{eq:slims_star} 
	H_i \coloneqq \langle \varpi, \partial_{t_i} \rangle = \sum_{j \neq i \in I^{\infty}} \frac{\Tr(P_iQ_jP_jQ_i)}{t_i - t_j} \, .
\end{equation}

The universal simply-laced quantum connection of Definition~\ref{def:slqc} specialises to
\begin{equation}
	\widehat{\nabla} = d - \widehat{\varpi} = d - \sum_{i \in I^{\infty}} \left(\sum_{i \neq j \in I^{\infty}} \frac{\Tr(\widehat{Q}_j\widehat{P}_j\widehat{Q}_i\widehat{P}_i)}{t_i - t_j} \cdot \hslash^4\right)dt_i \, ,
\end{equation}
where $\widehat{Q}_j$ and $\widehat{P}_j$ are matrices whose coefficients are the Weyl quantisation of the coordinate functions $(Q_j)_{kl}$ and $(P_j)_{kl}$, and the order of these noncommutative variables is fixed by anchoring all degenerate 4-cycles at their centre---following Definition~\ref{def:quantisation_imd_potentials}. 

The expansion of the universal simply-laced quantum Hamiltonian at the node $i \in I^{\infty}$ is then
\begin{equation}
	\label{eq:slqc_star}
	\widehat{H}_i \coloneqq \langle \widehat{\varpi}, \partial_{t_i} \rangle = \sum_{i \neq j} \frac{\Tr(\widehat{Q}_j\widehat{P}_j\widehat{Q}_i\widehat{P}_i)}{t_i - t_j} \cdot \hslash^4 \, ,
\end{equation}
whereas the simply-laced quantum Hamiltonians are the functions $\rho_1(\widehat{H}_i) \colon \mathbf{B} \longrightarrow A$ as in Definition~\ref{def:level_1_slqc}. 

We now prove directly that the universal quantum Hamiltonians commute. 

\begin{prop}
	\label{prop:flatness_SLQC_star}
	One has $\bigl[\widehat{H}_i,\widehat{H}_j\bigr] = 0$ for all $i,j \in I^{\infty}$.
\end{prop}

\begin{proof}
	Assume $i \neq j \in I^{\infty}$. The trick is to split the commutator in the following sum:
	\begin{equation}
		\begin{split}
			\bigl[\widehat{H}_i,\widehat{H}_j\bigr] = \hslash^8 \cdot &\sum_{k \in I^{\infty} \setminus \{i,j\}} \frac{1}{(t_i - t_k)(t_j - t_k)}\Big[\Tr(\widehat{Q}_k\widehat{P}_k\widehat{Q}_i\widehat{P}_i),\Tr(\widehat{Q}_k\widehat{P}_k\widehat{Q}_j\widehat{P}_j)\Big] \\
			&+ \frac{1}{(t_i - t_j)(t_j - t_k)}\Big[\Tr(\widehat{Q}_j\widehat{P}_j\widehat{Q}_i\widehat{P}_i),\Tr(\widehat{Q}_j\widehat{P}_j\widehat{Q}_k\widehat{P}_k)\Big] \\
			&+ \frac{1}{(t_i - t_k)(t_j - t_i)}\Big[\Tr(\widehat{Q}_i\widehat{P}_i\widehat{Q}_k\widehat{P}_k),\Tr(\widehat{Q}_i\widehat{P}_i\widehat{Q}_j\widehat{P}_j)\Big] \, .
		\end{split}
	\end{equation}
	This decomposition is suggested by enumerating the possible degenerate 4-cycles with exactly one peripheral node in common, because this is the only situation leading to nonvanishing commutators. 
	Such nonvanishing commutators can be depicted as follows (for $k \in I^{\infty} \setminus \{i,j\}$):
	\begin{center}
		\begin{tikzpicture}
			\draw (-0.5,-1) -- (-1,-1) -- (-1,3.3) -- (-0.5,3.3) [-];
			\vertex (a) at (0,0) [label = left:$k$] {};
			\vertex (b) at (1,1) [fill = black] {};
			\vertex (c) at (1,2.4) [label = above:$i$] {};
			\vertex (d) at (2.5,0) [label = below:$k$] {};
			\vertex (e) at (3.5,1) [fill = black] {};
			\vertex (f) at (4.7,1) [label = right:$j$] {};
			\node (g) at (2,1) {\textbf{,}};
			\node (h) at (6,1) {\textbf{,}};
			\draw (5.2,-1) -- (5.7,-1) -- (5.7,3.3) -- (5.2,3.3) [-];
			\path
			(a) edge [bend left = 20] (b)
			(b) edge [bend left = 20] (a)
			(c) edge [bend left = 20] (b)
			(b) edge [bend left = 20] (c)
			(d) edge [bend left = 20] (e)
			(e) edge [bend left = 20] (d)
			(f) edge [bend left = 20] (e)
			(e) edge [bend left = 20] (f);
		\end{tikzpicture}
	\end{center}
	when $k$ is in common; then
	\begin{center}
		\begin{tikzpicture}
			\draw (-0.5,-2) -- (-1,-2) -- (-1,2.5) -- (-0.5,2.5) [-];
			\vertex (a) at (0,0) [fill = black] {};
			\vertex (b) at (0,1.4) [label = above:$i$] {};
			\vertex (c) at (1.3,0) [label = right:$j$] {};
			\vertex (d) at (3,-1) [label = below:$k$] {};
			\vertex (e) at (4,0) [fill = black] {};
			\vertex (f) at (5.3,0) [label = right:$j$] {};
			\node (g) at (2.3,0) {\textbf{,}};
			\node (h) at (6.6,0) {\textbf{,}};
			\draw (5.8,-2) -- (6.3,-2) -- (6.3,2.5) -- (5.8,2.5) [-];
			\path
			(b) edge [bend left = 20] (a)
			(a) edge [bend left = 20] (b)
			(c) edge [bend left = 20] (a)
			(a) edge [bend left = 20](c)
			(d) edge [bend left = 20] (e)
			(e) edge [bend left = 20] (d)
			(f) edge [bend left = 20] (e)
			(e) edge [bend left = 20] (f);
		\end{tikzpicture}
	\end{center}
	when $j$ is in common; finally
	\begin{center}
		\begin{tikzpicture}
			\draw (-0.5,-1) -- (-1,-1) -- (-1,3.3) -- (-0.5,3.3) [-];  
			\vertex (a) at (0,0) [label = left:$k$] {};
			\vertex (b) at (1,1) [fill = black] {};
			\vertex (c) at (1,2.4) [label = above:$i$] {};
			\vertex (d) at (3,2.4) [label = above:$i$] {};
			\vertex (e) at (3,1) [fill = black] {};
			\vertex (f) at (4.3,1) [label = right:$j$] {};
			\node (g) at (2,1) {\textbf{,}};
			\node (h) at (5.6,1) {\textbf{,}};
			\draw (4.8,-1) -- (5.3,-1) -- (5.3,3.3) -- (4.8,3.3) [-];
			\path
			(a) edge [bend left = 20] (b)
			(b) edge [bend left = 20] (a)
			(c) edge [bend left = 20] (b)
			(b) edge [bend left = 20] (c)
			(d) edge [bend left = 20] (e)
			(e) edge [bend left = 20] (d)
			(f) edge [bend left = 20] (e)
			(e) edge [bend left = 20] (f);
		\end{tikzpicture}
	\end{center}
	when $i$ is in common. 
	Now, using Lemma~\ref{lem:commutator_degenerate_4_cycles} one finds:
	\begin{equation}
		\begin{split}
			\bigl[&\widehat{H}_i,\widehat{H}_j\bigr] = \hslash^8 \cdot \sum_{k \neq i,j} \frac{1}{(t_i - t_k)(t_j - t_k)}\Tr\Big([\widehat{Q}_i\widehat{P}_i,\widehat{Q}_k\widehat{P}_k]\widehat{Q}_j\widehat{P}_j\Big) \\
			& \qquad \qquad + \frac{1}{(t_i - t_j)(t_j - t_k)}\Tr\Big([\widehat{Q}_i\widehat{P}_i,\widehat{Q}_j\widehat{P}_j]\widehat{Q}_k\widehat{P}_k\Big) \\
			& \qquad \qquad + \frac{1}{(t_i - t_k)(t_j - t_i)}\Tr\Big([\widehat{Q}_k\widehat{P}_k,\widehat{Q}_i\widehat{P}_i]\widehat{Q}_j\widehat{P}_j\Big) \\
			&= \hslash^8 \cdot \sum_{k \neq i,j} \left[\frac{1}{(t_i - t_k)(t_j - t_k)} - \frac{1}{(t_i - t_j)(t_j - t_k)} - \frac{1}{(t_i - t_k)(t_j - t_i)}\right] \\
			& \qquad \qquad \cdot \Tr\Big([\widehat{Q}_j\widehat{P}_j,\widehat{Q}_i\widehat{P}_i]\widehat{Q}_k\widehat{P}_k\Big) \, .
		\end{split}
	\end{equation}
	Each addend of this sum vanishes because of the cyclic relation
	\begin{equation}
		\frac{1}{(t_i - t_k)(t_j - t_k)} - \frac{1}{(t_i - t_j)(t_j - t_k)} - \frac{1}{(t_i - t_k)(t_j - t_i)} = 0.
	\end{equation}
\end{proof}

\subsection{KZ connection}
\label{sec:KZ}

Here we recall the definition of the KZ connection~\cite{knizhnik_zamolodchikov_1984_wess_zumino_witten}, and we derive an explicit formula for the KZ Hamiltonians.

Let $m \geq 1$ be an integer, and consider the space $\mathbf{B} = \mathbb{C}^m \setminus \{\diags\} \subseteq \mathbb{C}^m$ of configuration of ordered $m$-tuples of points in $\mathbb{C}$, with global complex coordinates $\{t_i\}_i$ defined by the restriction of the standard coordinates on $\mathbb{C}^m$. 
Choose a Lie algebra $\mathfrak{g}$ equipped with an invariant nondegenerate symmetric bilinear form $K \in \mathfrak{g}^* \otimes \mathfrak{g}^*$ identifying $\mathfrak{g} \simeq \mathfrak{g}^*$. 
Consider the trivial bundle $U(\mathfrak{g})^{\otimes m} \times \mathbf{B} \longrightarrow \mathbf{B}$ with fibre the $m$-fold tensor power of the universal enveloping algebra of $\mathfrak{g}$. 
Then KZ is a connection $\widehat{\nabla}_{\KZ}$ on this trivial bundle of noncommutative algebras. 

To write it down, use the duality $\mathfrak{g} \simeq \mathfrak{g}^*$ to turn the identity morphism $\Id_{\mathfrak{g}} \in \mathfrak{g} \otimes \mathfrak{g}^*$ into an element $\Omega \in \mathfrak{g} \otimes \mathfrak{g}$. 
Then applying the universal inclusion $\iota_U \colon \mathfrak{g} \hookrightarrow U(\mathfrak{g})$ on each factor yields an element of $U(\mathfrak{g}) \otimes U(\mathfrak{g})$, abusively denoted $\Omega$ as well.
Then let $\Omega^{(ij)} \in U(\mathfrak{g})^{\otimes m}$ be its natural embedding on the $i$-th and $j$-th slot of the $m$-fold tensor product, for $i \neq j \in \{1,\dotsc,m\}$.

The \emph{universal KZ connection} is defined by
\begin{equation}
	\label{eq:KZ}
	\widehat{\nabla}_{\KZ} = d - \widehat{\varpi}_{\KZ} \coloneqq d - \sum_{i \neq j} \Omega^{(ij)} \frac{dt_i - dt_j}{t_i - t_j} \, ,
\end{equation}
where $\Omega^{(ij)} \in U(\mathfrak{g})^{\otimes m}$ acts by left multiplication. 
The \emph{KZ Hamiltonians} are then given by
\begin{equation}
	\widehat{H}^{\KZ}_i \coloneqq \langle \widehat{\varpi}_{\KZ},\partial_{t_i} \rangle = \sum_{j \neq i} \frac{\Omega^{(ij)}}{t_i - t_j} \, ,
\end{equation}
and constitute a strongly integrable nonautonomous quantum Hamiltonian system. 

\begin{rem}
	\label{rem:formula_2_tensor_kz}
	The above definition of $\Omega$ is intrinsic. 
	One may equivalently choose an orthonormal basis $(x_i)_i$ of $\mathfrak{g}$ and find $\Omega = \sum_i x_i \otimes x_i$, $K = \sum_i dx_i \otimes dx_i$. 

	Furthermore, the connections $d - \hslash\widehat{\varpi}_{\KZ}$ are strongly flat as well for every choice of the parameter $\hslash$.
	Using the construction of Proposition~\ref{prop:rees} one can add such a quantum parameter in~\eqref{eq:KZ}, matching it with the expression usually found in the literature ($\hslash^{-1}$ will be the sum of the dual Coxeter number $h^{\vee}$ of a simple Lie algebra $\mathfrak{g}$, and of the level $\kappa \in \mathbb{Z}_{\geq 0}$ for the action of the central element on integrable highest-weight modules for the affine Lie algebra $\widehat{\mathfrak{g}}$; see e.g.~\cite{etingov_frenkel_kirillov_1998_lectures_on_representation_theory_and_kz_equations, kohno_2002_conformal_field_theory_and_topology}).
\end{rem}

This construction can be specialised to $\mathfrak{g} = \mathfrak{gl}(W^0)$, where $W^0$ is a finite-dimensional vector space, taking the nondegenerate invariant trace pairing $K(A,B) \coloneqq \Tr(AB)$. 

\begin{defn}
	The isomorphism $\mathfrak{g} \simeq \mathfrak{g}^*$ induced by the nondegenerate trace pairing is called the trace-duality.
\end{defn}

The trace-duality sends the vector $e_{ij}$ of the canonical basis of $\mathfrak{g}$ to the covector $de_{ji} \in \mathfrak{g}^*$. 
Hence in these coordinates one finds
\begin{equation}
	\label{eq:kz_hamiltonians}
	\widehat{H}^{\KZ}_i = \sum_{j \neq i} \sum_{k,l} \frac{\widehat{e}^{(i)}_{kl} \cdot \widehat{e}^{(j)}_{lk}}{t_i - t_j} \, ,
\end{equation}
for all $i \in \{1,\dotsc,m\}$, where the product on $U(\mathfrak{g})^{\otimes m}$ is defined factor-wise, $\widehat{e}_{jk} \coloneqq \iota_U(e_{jk})$ and $\widehat{e}^{(i)}_{jk} \in U(\mathfrak{g})^{\otimes m}$ denotes the natural embedding of $\widehat{e}_{jk} \in U(\mathfrak{g})$ on the $i$th slot of $U(\mathfrak{g})^{\otimes m}$.

\subsection{Schlesinger system}
\label{sec:schlesinger_system}

In this section we write explicit formul\ae{} for the Schlesinger Hamiltonians~\cite{schlesinger_1905_ueber_die_loesungen_gewisser_linearer_differentialgleichungen}.

Fix again a finite-dimensional vector space $W^0$, and consider the trivial holomorphic vector bundle $W^0 \times \mathbb{C} P^1 \longrightarrow \mathbb{C} P^1$.
Take $z$ to be a holomorphic coordinate identifying $\mathbb{C} P^1 \simeq \mathbb{C} \cup \{\infty\}$ and choose complex numbers $t_1, \dotsc t_m \in \mathbb{C}$ and endomorphisms $R_1, \dotsc R_m \in \mathfrak{g} = \mathfrak{gl}(W^0)$ such that $\sum_i R_i = 0$. 
Then the Fuchsian system with poles at the points $t_i$ and residues $R_i$ is the system of linear first order differential equations
\begin{equation}
	\frac{\partial{\psi}}{\partial z} = \sum_{i = 1}^m \frac{R_i}{z - t_i} \psi \, ,
\end{equation}
for a local holomorphic section $\psi$ of $W^0 \times \mathbb{C} P^1 \longrightarrow \mathbb{C} P^1$. 
The isomonodromy equations for the residues under variations of the positions of the poles yield the Schlesinger system:
\begin{equation}
	dR_i = -\sum_{i \neq j} \bigl[R_i,R_j\bigr] \frac{dt_i - dt_j}{t_i - t_j} \, .
\end{equation}
These are nonlinear first order differential equations defining an integrable Ehresmann connection on the trivial Poisson fibration $\mathfrak{g}^m \times \mathbf{B} \longrightarrow \mathbf{B}$, where $\mathbf{B} = \mathbb{C}^m \setminus \{\diags\}$ parametrises the choices of the positions of the poles, and where we equip $\mathfrak{g}$ with the linear Poisson bracket coming from the Lie--Poisson structure of $\mathfrak{g}^*$ under the trace-duality. 

It can now be shown that there exist smooth functions $H^{\Sch}_i \colon \mathfrak{g}^m \times \mathbf{B} \longrightarrow \mathbb{C}$ such that the isomonodromy equations become
\begin{equation}
	\frac{\partial R_j}{\partial t_i} = \{H^{\Sch}_i,R_j\} \, ,
\end{equation}
for $1 \leq i,j \leq m$. 
These functions are called the \emph{Schlesinger Hamiltonians}, and are explicitly given by 
\begin{equation}
	\label{eq:schlesinger_hamiltonians}
	H^{\Sch}_i = \sum_{i \neq j} \frac{\Tr(R_iR_j)}{t_i - t_j} \, .
\end{equation}

\subsection{KZ is a quantisation of Schlesinger}
\label{sec:quantum_Schlesinger=KZ} 

It is shown in~\cite{reshetikhin_1992_kz_deformation_isomonodromy, harnad_1996_quantum_imd_and_the_kz_equations} that the KZ connection is a quantisation of the Schlesinger system. 
In this section we give a proof adapted to our notation.

There is a standard quantisation machinery for the dual of a Lie algebra $\mathfrak{g}^*$. 
Namely, the algebra of regular functions on $\mathfrak{g}^*$ is isomorphic to $\Sym(\mathfrak{g})$, which is a graded commutative Poisson algebra endowed with its Lie--Poisson structure. 
The material of \S~\ref{sec:quantisation_algebras} applies, and one can look for a filtered quantisation of $\Sym(\mathfrak{g})$. 

An explicit filtered quantisation is described by the Poincaré--Birkhoff--Witt theorem, stating that there is an isomorphism of graded commutative algebras $\Sym(\mathfrak{g}) \simeq \gr U(\mathfrak{g})$, taking the standard filtration on the universal enveloping algebra of $\mathfrak{g}$. 
Moreover, the identity~\eqref{eq:filtered_quantisation} for Poisson brackets holds, and thus $U(\mathfrak{g})$ is a filtered quantisation of $\Sym(\mathfrak{g})$, which constitutes a Poisson analogue of the symplectic construction of the Weyl algebra of \S~\ref{sec:weyl}. 

We now extend the universal inclusion $x \longmapsto \widehat{x}$ to a map defined on the whole of the symmetric algebra, by full symmetrisation (in characteristic zero).

\begin{defn}
	\label{def:pbw_quantisation}
	The PBW-quantisation $\mathcal{Q}_{\PBW} \colon \Sym(\mathfrak{g}) \longrightarrow U(\mathfrak{g})$ is the map defined on monomials of degree $n$ by
	\begin{equation}
		\mathcal{Q}_{\PBW} \colon \prod_{i = 1}^n x_i \longmapsto \frac{1}{n!} \sum_{\tau \in \Sigma_n} \prod_{i = 1}^n \widehat{x}_{\tau(i)} \, ,
	\end{equation}
	where $x_i \in \mathfrak{g}$ for all $i$ and $\Sigma_n$ is the symmetric group on $n$ objects.
\end{defn}

To apply this to the KZ connection take $\mathfrak{g} = \mathfrak{gl}(W^0)$ as in \S~\ref{sec:KZ}, and extend the PBW-quantisation to a map $\Sym(\mathfrak{g})^{\otimes m} \longrightarrow U(\mathfrak{g})^{\otimes m}$ in the natural way. 
Now, the Schlesinger Hamiltonians~\eqref{eq:schlesinger_hamiltonians} are global sections
\begin{equation}
	H^{\Sch}_i \colon \mathbf{B} \longrightarrow \Sym(\mathfrak{g}^*)^{\otimes m} \, ,
\end{equation}
and thanks to the trace-duality $\mathfrak{g}^* \simeq \mathfrak{g}$ one may consider them as taking values in $\Sym(\mathfrak{g})^{\otimes m}$. 
Similarly, the KZ Hamiltonians~\eqref{eq:kz_hamiltonians} are global sections 
\begin{equation}
	\widehat{H}^{\KZ}_i \colon \mathbf{B} \longrightarrow U(\mathfrak{g})^{\otimes m} \, ,
\end{equation}
and thus it makes sense to compare the KZ Hamiltonians with the PBW-quantisation of the Schlesinger Hamiltonians.

\begin{thm}
	One has $\mathcal{Q}_{\PBW}(H^{\Sch}_i) = \widehat{H}_i^{\KZ}$ for all $i = 1, \dotsc, m$, pointwise on $\mathbf{B}$.
\end{thm}

\begin{proof}
	By linearity it is enough to show that 
	\begin{equation}
		\mathcal{Q}_{\PBW}\big(\Tr(R_iR_j)\big) = \Omega^{(ij)}, \qquad \text{for all } i \neq j \in \{1, \dotsc, m\} \, .
	\end{equation}
	Note that the function $(R_1,\dotsc,R_m) \longmapsto \Tr(R_iR_j)$ is the embedding of the invariant nondegenerate bilinear form $K \in \mathfrak{g}^* \otimes \mathfrak{g}^*$ on the on the $i$th and $j$th slot of $\Sym(\mathfrak{g}^*)^{\otimes m}$, and that the PBW-quantisation reduces to the universal inclusion $\iota_U$ on elements of degree one. Hence it is enough to show that the trace-dual $K^*$ of $K$ equals $\Omega$ inside $\mathfrak{g} \otimes \mathfrak{g}$, which is clear by computing in a $K$-orthonormal basis (cf. Remark~\ref{rem:formula_2_tensor_kz}).
\end{proof}

\subsection{Classical Hamiltonian reduction}
\label{sec:reduced_star=schlesinger}

In this section we show that the classical Hamiltonian reduction of the simply-laced Hamiltonians system~\eqref{eq:slims_star} yields the Schlesinger Hamiltonians~\eqref{eq:schlesinger_hamiltonians}. 
The formal procedure is to replace the matrix product $Q_iP_i$ that appears in~\eqref{eq:slims_star} with the residue $R_i$ that appears in~\eqref{eq:schlesinger_hamiltonians}, and we will turn this into an algebraic statement.

Consider again the vector spaces $W^0$ and $W^{\infty} = \bigoplus_{i \in I^{\infty}} V_i$, and keep the notation $\mathfrak{g} = \mathfrak{gl}(W^0)$. 
Using the nondegenerate pairing provided by the trace, one has for all $i \in I^{\infty}$ a canonical identification
\begin{equation}
	L_i \coloneqq \Hom(V_i,W^0) \simeq \Hom(W^0,V_i)^* \, ,
\end{equation}
and thus
\begin{equation}
	\Hom(V_i,W^0) \oplus \Hom(W^0,V_i) \simeq T^*L_i \, .
\end{equation}
Consider now the product map, that is 
\begin{equation}
	\mu_i \colon T^*L_i \longrightarrow \mathfrak{g}, \qquad (Q_i,P_i) \longmapsto Q_iP_i \, .
 \end{equation}
Up to using the trace-duality this is the moment map for the restriction of the $\GL(W^0)$-action on the invariant symplectic subspace $T^*L_i \subseteq \mathbb{M}$; in particular it is a Poisson map. 
The direct sum of the maps $\mu_i$ thus yields a Poisson map
\begin{equation}
	\mu \colon \mathbb{M} = \bigoplus_{i \in I^{\infty}} T^*L_i \longrightarrow \mathfrak{g}^m, \qquad \mu \colon (Q_i,P_i)_{i \in I^{\infty}} \longmapsto (Q_iP_i)_{i \in I^{\infty}} \, ,
\end{equation}
and the pull-back with respect to $\mu$ yields a Poisson morphism 
\begin{equation}
	\mu^* \colon \Sym(\mathfrak{g}^*)^{\otimes m} \longrightarrow \Sym(\mathbb{M}^*) = A_0 \, .
\end{equation}
To write an explicit formula for it, let $(R_i)_{jk}$ be the coordinate functions on the $i$th factor of the product $\mathfrak{g}^m$, and $(Q_i)_{jm}$, $(dP_i)_{ml}$ the coordinate functions on $\mathbb{M}$---once bases of $W^0$ and $V_i$ have been chosen.
Then for all $i \in I^{\infty}$ one has 
\begin{equation}
	\label{eq:global_classical_moment}
	\mu^*\big((R_i)_{jk}\big) = \sum_m (Q_i)_{jm}(P_i)_{mk} \, .
\end{equation}
We see that $\mu^*$ is a morphism turning polynomial functions on $\mathfrak{g}^m$ into polynomial functions on $\mathbb{M}$, and we now prove that this natural correspondence matches~\eqref{eq:slims_star} with~\eqref{eq:schlesinger_hamiltonians}.

\begin{prop}
	\label{prop:pull_back_classical}
	One has $\mu^*(H^{\Sch}_i) = H_i$ for all $i \in I^{\infty}$.
\end{prop}

\begin{proof}
	By linearity, it is enough to check that
	\begin{equation}
		\mu^*\Tr(R_iR_j) = \Tr(Q_iP_iQ_jP_j), \qquad \text{for } i \neq j \in I^{\infty} \, .
	\end{equation}
	The definition of $\mu^*_i$ implies precisely
	\begin{equation}
		\begin{split}
			\mu^*&\Tr(R_iR_j) = \mu^*\sum_{k,l} (R_i)_{kl}(R_j)_{lk} = \sum_{k,l} \big(\mu_i^* (R_i)_{kl}\big) \big(\mu_j^* (R_j)_{lk}\big) \\
			&= \sum_{k,l,m,n} (Q_i)_{km}(P_i)_{ml}(Q_j)_{ln}(P_j)_{nk} = \Tr(Q_iP_iQ_jP_j) \, .
		\end{split}
	\end{equation}
\end{proof}

This gives an algebraic meaning to performing the formal change of variable $R_i \coloneqq Q_iP_i$. This can be used to show that~\eqref{eq:schlesinger_hamiltonians} is the Hamiltonian reduction of~\eqref{eq:slims_star} under the action of change of bases in the spaces at the peripheral nodes of the star $\mathcal{G}$. 
With this end in mind we now recall the definition of the Hamiltonian reduction of a Poisson algebra  (see~\cite{etingov_2007_calogero_moser_systems}; what is there called ``moment map'' we call ``comoment map'', following the usual convention that a moment map takes values in the dual of a Lie algebra). 

Start abstractly: let $(B,\{\cdot,\cdot\})$ be a commutative Poisson algebra and $\mathfrak{h}$ a Lie algebra. 

\begin{defn}
	An $\mathfrak{h}$-action on $B$ is a morphism $\xi \colon \mathfrak{h} \longrightarrow \Der(B)$ of Lie algebras. A comoment map for $\xi$ is a morphism $\mu^* \colon \Sym(\mathfrak{h}) \longrightarrow B$ of Poisson algebras whose restriction to $\mathfrak{h}$ lifts $\xi$ through the adjoint action of $B$ on itself:
	\begin{center}
		\begin{tikzpicture}[> = to]
			\node (a) at (0,0) {$\mathfrak{h}$};
			\node (b) at (3,0) {$\Der(B)$};
			\node (c) [above of = b] {$B$};
			\path
			(a) edge node [below]{$\xi$} (b)
			(a) edge [densely dashed] node [above left]{$\left.\mu^*\right|_{\mathfrak{h}}$} (c)
			(c) edge node [right]{$\ad$} (b);
		\end{tikzpicture}
	\end{center}
\end{defn}

The action is then uniquely determined by $\xi(x).a = \{\mu(x),a\}$, for $x \in \mathfrak{h}$ and $a \in B$.

Assume now to have a comoment $\mu^* \colon \Sym(\mathfrak{h}) \longrightarrow B$ for an $\mathfrak{h}$-action on $B$, and let $\mathfrak{I} \subseteq \Sym(\mathfrak{h})$ be an ideal.

\begin{defn}
	\label{def:hamiltonian_reduction}
	The Hamiltonian reduction of $B$ with respect to the comoment map $\mu^*$ and the ideal $\mathfrak{I}$ is the quotient ring 
	\begin{equation}
		R(B,\mathfrak{h},\mathfrak{I}) \coloneqq B^{\mathfrak{h}} \big\slash \mathfrak{J}^{\mathfrak{h}} \, ,
	\end{equation}
	where 
	\begin{equation}
		B^{\mathfrak{h}} \coloneqq \Set{f \in B | \{\mu^*(\mathfrak{h}),f\} = 0} \, ,
	\end{equation}
	is the $\mathfrak{h}$-invariant part, $\mathfrak{J} \subseteq B$ is the ideal generated by $\mu^*(\mathfrak{I})$ inside $B$, and $\mathfrak{J}^{\mathfrak{h}} \coloneqq \mathfrak{J} \cap B^{\mathfrak{h}}$.
\end{defn}

To apply this to the case at hand set $G^{\infty} \coloneqq \prod_{i \in I^{\infty}} \GL(V_i)$, with Lie algebra $\mathfrak{g}^{\infty}$. 
The group acts on $(\mathbb{M},\omega_a)$, and the infinitesimal action of $\mathfrak{g}^{\infty}$ by vector fields on $\mathbb{M}$ yields by definition a $\mathfrak{g}^{\infty}$-action on $A_0$, where $A_0 = \Sym(\mathbb{M}^*)$ plays the role of the Poisson algebra $B$ in the above definitions. 
There is a comoment map $\mu^*_{\infty} \colon \Sym(\mathfrak{g}^{\infty}) \longrightarrow A_0$ for the $\mathfrak{g}^{\infty}$-action, and we may choose an ideal $\mathfrak{I} \subseteq \Sym(\mathfrak{g}^{\infty})$. 
Then the Hamiltonian reduction $R(A_0,\mathfrak{g}^{\infty},\mathfrak{I})$ is defined, together with the canonical projection $\pi_{\mathfrak{I}} \colon A_0^{\mathfrak{g}^{\infty}} \longrightarrow R(A_0,\mathfrak{g}^{\infty},\mathfrak{I})$. 
By Corollary~\ref{cor:invariance} the simply-laced Hamiltonian $H_i$ of~\eqref{eq:slims_star} is $\mathfrak{g}^{\infty}$-invariant.

\begin{defn}
	The element $\pi_{\mathfrak{I}}(H_i) \in R(A_0,\mathfrak{g}^{\infty},\mathfrak{I})$ is the Hamiltonian reduction of the function $H_i$ at the ideal $\mathfrak{I}$.
\end{defn}
 
The geometric counterpart of this algebraic constriction is the following. 
Fix a Zariski-closed coadjoint $G^{\infty}$-orbit $\mathcal{O} \subseteq (\mathfrak{g}^{\infty})^*$, and let $\mathfrak{I}$ be the associated ideal of vanishing functions
\begin{equation}
	\mathfrak{I} \coloneqq I[\mathcal{O}] = \Set{ f \in \Sym(\mathfrak{g}^{\infty}) | \left.f\right|_{\mathcal{O}} = 0} \, .
\end{equation}
If $\mathfrak{J}$ is the ideal generated by $\mu^*_{\infty}(\mathfrak{I}) \subseteq A_0$, then the quotient ring $A_0 \big\slash \mathfrak{J}$ corresponds to the algebra of regular functions on the preimage $\mu_{\infty}^{-1}(\mathcal{O}) \subseteq \mathbb{M}$.
Finally, the $G^{\infty}$-invariant part yields the ring of functions on the quotient
\begin{equation}
	\mu_{\infty}^{-1}(\mathcal{O}) \big\slash G^{\infty} \eqqcolon \mathbb{M} \sslash_{\mathcal{O}} G^{\infty} \, ,
\end{equation}
which is the usual Marsden-Weinstein symplectic reduction (a symplectic leaf of the Poisson scheme $\mathbb{M} \slash G^{\infty}$). 

Analogously, one can take a product of coadjoint $\GL(W^0)$-orbits $\mathcal{O}' \coloneqq \mathcal{O}_1 \times \dotsm \times \mathcal{O}_m \subseteq (\mathfrak{g}^*)^m$. 
This is a symplectic manifold endowed with the sum of the Kirillov--Konstant--Souriau symplectic forms, and the Schlesinger system restricts to time-dependent Hamiltonians $\left.H^{\Sch}_i\right|_{\mathcal{O}'}$ defined on the trivial symplectic fibration $\mathcal{O}' \times \mathbf{B} \longrightarrow \mathbf{B}$. 
Hence at any fixed time these Hamiltonians are elements of the quotient ring $\Sym(\mathfrak{g}^*)^{\otimes m} \big\slash \mathfrak{I}'$, where $\mathfrak{I}' = I\bigl[\mathcal{O}'\bigr]$ is the ideal of functions vanishing on $\mathcal{O}'$.

We thus take a last restriction on the simply-laced isomonodromy system of the star to reduce it to the Schlesinger system: assume that the vector spaces $V_i$ are all equal to $W^0$, i.e. that we attach one and the same vector space to all nodes of $\mathcal{G}$. 
Then there are canonical identifications $G^{\infty} \simeq \GL(W^0)^m$ and $\mathfrak{g}^{\infty} \simeq \mathfrak{g}^m$, under which the product $\mathcal{O}' \subseteq \mathfrak{g}^m$ of $\GL(W^0)$-orbits becomes a $G^{\infty}$-orbit $\mathcal{O} \subseteq (\mathfrak{g}^{\infty})^*$---using the trace-duality.
Moreover, there is an induced isomorphism $\Sym(\mathfrak{g}^*)^{\otimes m} \simeq \Sym(\mathfrak{g}^{\infty})$ which sends the $\mathcal{O}$-vanishing ideal $\mathfrak{I} \subseteq \Sym(\mathfrak{g}^{\infty})$ onto the $\mathcal{O}'$-vanishing ideal $\mathfrak{I}' \subseteq \Sym(\mathfrak{g}^*)^{\otimes m}$, providing a ring isomorphism 
\begin{equation}
	\Sym(\mathfrak{g}^{\infty}) \big\slash \mathfrak{I} \simeq \Sym(\mathfrak{g}^*)^{\otimes m} \big\slash \mathfrak{I}' \, .
\end{equation}

The following proposition then gives a uniform way of comparing functions defined on $\mathcal{O}'$ with functions defined on the symplectic reduction $\mathbb{M} \sslash_{\mathcal{O}} G^{\infty}$.

\begin{prop}
	There is a natural injective ring morphism
	\begin{equation}
		\varphi \colon \Sym(\mathfrak{g}^*)^{\otimes m} \big\slash \mathfrak{I}' \longrightarrow R(A_0,\mathfrak{g}^{\infty},\mathfrak{I}) \, ,
	\end{equation}
	induced by composing the map $\mu^*$ of~\eqref{eq:global_classical_moment} with the canonical projection $\pi_{\mathfrak{I}}$ on the Hamiltonian reduction.
\end{prop}

\begin{proof}
	The only thing to show is that the image of $\mu^*$  is contained in the invariant part $A_0^{\mathfrak{g}^{\infty}}$. Indeed if that is true then the composition 
	\begin{equation}
		\pi_{\mathfrak{I}} \circ \mu^* \colon \Sym(\mathfrak{g}^*)^{\otimes m} \longrightarrow R(A_0,\mathfrak{g}^{\infty},\mathfrak{I})
	\end{equation}
	is well defined, and by construction the preimage of $\mathfrak{J}$ under $\mu^*$ is the ideal $\mathfrak{I}'$, up to the aforementioned identification $\Sym(\mathfrak{g}^*)^{\otimes m} \simeq \Sym(\mathfrak{g}^{\infty})$. 
	Hence the morphism $\varphi$ induced to the quotient has no kernel.

	Finally, the fact that $\mu^*$ takes values in the invariant algebra is due to a general fact about comoment maps for commutative Hamiltonian actions of Lie groups: their images are Poisson-commutative subalgebras. 
	This means that if $\mu_0 \colon \mathbb{M} \longrightarrow \mathfrak{g} \simeq \mathfrak{g}^*$ is the moment for the $\GL(W^0)$-action on $\mathbb{M}$, then one has $\mu^*_0\big(\Sym(\mathfrak{g}^*)\big) \subseteq A_0^{\mathfrak{g}^{\infty}}$. 
	The same holds for $\mu_i$, which is the restriction of $\mu_0$ to $T^*\Hom(V_i,W^0)$. 
	Since $\mu^*$ is defined as $\mu^*_i$ on each factor $T^*\Hom(V_i,W^0) \subseteq \mathbb{M}$, the proof is complete.
\end{proof}

Proposition~\ref{prop:pull_back_classical} immediately yields the final result of this section.

\begin{thm}
	One has $\varphi\left(\left.H_i^{\Sch}\right|_{\mathcal{O}'}\right) = \pi_{\mathfrak{I}}(H_i)$ for all $i \in I^{\infty}$.
\end{thm}

Hence indeed the Schlesinger system corresponds to the Hamiltonian reduction of this particular case of the simply-laced isomonodromy system, for every choice of ideal/coadjoint orbit. 

\begin{rem}
	The residual $\GL(W^0)$-action on $\mathbb{M} \sslash_{\mathcal{O}} G^{\infty}$---the action at the central node of $\mathcal{G}$---can further be modded out to yield the moduli space $\mathcal{M}^*_{\dR}$ of isomorphism classes of meromorphic connections~\eqref{eq:kz_meromorphic_connection} on a trivial vector bundle over $\mathbb{C}P^1$ (cf. Remark~\ref{rem:quantisation_local_systems}). 
	Then the totally reduced simply-laced Hamiltonians correspond to the reduction of the Schlesinger Hamiltonians on the symplectic quotient $\mathcal{O}' \sslash_0 \GL(W^0)$ at the zero level of the moment map for the diagonal coadjoint $\GL(W^0)$-action on $\mathcal{O}'$. 
	In particular we recover the realisation of the moduli space of logarithmic connection on the Riemann sphere as the complex symplectic quotient of a product of coadjoint orbits (see~\cite{hitchin_1997_frobenius_manifolds}).
\end{rem}

\subsection{Quantum Hamiltonian reduction}
\label{sec:reduced_quantumstar=KZ}

In this section we show the natural quantum analogues of the results of the previous one: the quantum Hamiltonian reduction of the simply-laced quantum Hamiltonians $\rho_1(\widehat{H}_i)$ yields the KZ Hamiltonians~\eqref{eq:kz_hamiltonians}. 

As before, we start by explaining the algebraic meaning of the ``quantum'' change of variables $\widehat{R}_i = \widehat{Q}_i\widehat{P}_i$. 
Namely, we understand it as a morphism  
\begin{equation}
	\widehat{\mu}^* \colon U(\mathfrak{g}^*)^{\otimes m} \longrightarrow A
\end{equation}
of associative algebras, where $A = W(\mathbb{M}^*,\{\cdot,\cdot\})$ is the Weyl algebra. 

\begin{prop}
	\label{prop:quantum_comoment}
	The auxiliary morphism $\widetilde{\mu}^*_i \colon \Tens(\mathfrak{g}^*) \longrightarrow \Tens(T^*L_i)$, defined on $\mathfrak{g}^*$ by 
	\begin{equation}
		\widetilde{\mu}^*_i \colon (R_i)_{jk} \longmapsto \sum_m (Q_i)_{jm} \otimes (P_i)_{mk} \, ,
	\end{equation}
	induces an associative morphism $\widehat{\mu}^*_i \colon U(\mathfrak{g}^*) \longrightarrow W(T^*L_i,\{\cdot,\cdot\})$.
\end{prop}

\begin{proof}
	It is simpler to prove this fact for the function $\alpha \colon \Tens(\mathfrak{g}) \longrightarrow \Tens(T^*L_i)$ obtained by precomposing $\widetilde{\mu}^*_i$ with the trace-duality $\mathfrak{g} \longrightarrow \mathfrak{g}^*$. 
	This composition reads 
	\begin{equation}
		\alpha \colon (e_i)_{jk} \longmapsto \sum_m (Q_i)_{km} \otimes (P_i)_{mj} \, .
	\end{equation}

	Now the basis elements satisfy the commutation relations
	\begin{equation}
		\bigl[(e_i)_{jk},(e_i)_{lm}\bigr] = \delta_{kl}(e_i)_{jm} - \delta_{jm}(e_i)_{lk}, \qquad \big\{(Q_i)_{jk},(P_i)_{lm}\big\} = \delta_{kl}\delta_{jm} \, ,
	\end{equation} 
	which directly imply
	\begin{equation}
		\alpha\Big(\bigl[(e_i)_{jk},(e_i)_{lm}\bigr]\Big) = \Big\{\alpha\big((e_i)_{jk}\big),\alpha\big((e_i)_{lm}\big)\Big\} \, .
	\end{equation}
	This yields 
	\begin{equation}
		\widetilde{\mu}^*_i\Big(\bigl[(R_i)_{jk},(R_i)_{lm}\bigr]\Big) = \Big\{\widetilde{\mu}^*_i\big((R_i)_{jk}\big),\widetilde{\mu}^*_i\big((R_i)_{lm}\big)\Big\} \, ,
	\end{equation}
	and thus the two-sided ideal generated by 
	\begin{equation}
		x \otimes y - y \otimes x - [x,y] \in \Tens(\mathfrak{g}^*), \qquad \text{for } x,y \in \mathfrak{g}^* \, ,
	\end{equation}
	lands in the two-sided ideal generated by 
	\begin{equation}
		f \otimes g - g \otimes f - \{f,g\} \in \Tens(T^*L_i), \qquad \text{for } f,g, \in T^*L_i \, .
	\end{equation}
	Then the universal property of the quotient concludes the proof.
\end{proof}

Proposition~\ref{prop:quantum_comoment} shows that the formula
\begin{equation}
	\label{eq:global_quantum_moment}
	\widehat{\mu}^*_i \big((\widehat{R}_i)_{jk}\big) \coloneqq \sum_m (\widehat{Q}_i)_{jm} \cdot (\widehat{P}_i)_{mk} 
\end{equation}
defines a morphism $\widehat{\mu}_i \colon U(\mathfrak{g}^*) \longrightarrow W(T^*L_i,\{\cdot,\cdot\})$, where $(\widehat{Q}_i)_{jm}$, $(\widehat{P}_i)_{mk}$ are the Weyl quantisations of the corresponding coordinate functions. 
Now we collect~\eqref{eq:global_quantum_moment} into a morphism $\widehat{\mu}^* \colon U(\mathfrak{g}^*)^{\otimes m} \longrightarrow A$, defined in the natural way:
\begin{equation}
	\widehat{\mu}^*\left(\bigotimes_{i = 1}^m \widehat{f}_i\right) \coloneqq \prod_{i = 1}^m \widehat{\mu}^*_i(\widehat{f}_i), \qquad \text{where} \widehat{f}_i \in U(\mathfrak{g}^*) \text{ for all } i \, ,
\end{equation}
and where we use the associative product of $A$ on the right-hand side. 
This is indeed a morphism of associative algebras if one endows $U(\mathfrak{g}^*)^{\otimes m}$ with the factor-wise product, since inside $A$ one has
\begin{equation}
	\bigl[W(T^*L_i,\{\cdot,\cdot\}),W(T^*L_j,\{\cdot,\cdot\})\bigr] = (0), \qquad \text{for } i \neq j \, .
\end{equation} 

We now prove that this morphism matches up the KZ Hamiltonians~\eqref{eq:kz_hamiltonians} with the simply-laced quantum Hamiltonians.

\begin{prop}
	\label{prop:pull_back_quantum}
	One has $\widehat{\mu}^*\big(\widehat{H}^{\KZ}_i\big) = \rho_1(\widehat{H}_i)$ for all $i \in I^{\infty}$.
\end{prop}

\begin{proof}
	By linearity, it will be enough to show that
	\begin{equation}
		\widehat{\mu}^* \big(\Tr(\widehat{R}_i\widehat{R}_j)\big) = \Tr(\widehat{Q}_i\widehat{P}_i\widehat{Q}_j\widehat{P}_j), \qquad \text{for } i \neq j \in I^{\infty} \, .
	\end{equation}
	This follows from the straightforward expansion
	\begin{equation}
		\begin{split}
			&\widehat{\mu}^*\Tr(\widehat{R}_i\widehat{R}_j) = \widehat{\mu}^*\sum_{k,l} (\widehat{R}_i)_{kl} \cdot (\widehat{R}_j)_{lk} = \sum_{k,l} \big(\widehat{\mu}_i^* (\widehat{R}_i)_{kl}\big) \cdot \big(\widehat{\mu}_j^* (\widehat{R}_j)_{lk}\big) = \\
			&= \sum_{k,l,m,n} (\widehat{Q}_i)_{km} \cdot (\widehat{P}_i)_{ml} \cdot (\widehat{Q}_j)_{ln} \cdot (\widehat{P}_j)_{nk} = \Tr(\widehat{Q}_i\widehat{P}_i\widehat{Q}_j\widehat{P}_j) \, . \qedhere
		\end{split} 
	\end{equation}
\end{proof}

Hence one has lifted the classical correspondence of Proposition~\ref{prop:pull_back_classical} to a quantum correspondence. 
This can be used to show that the quantum Hamiltonian reduction of $\rho_1(\widehat{H}_i)$ with respect to the action of the group $G^{\infty} = \prod_{i \in I^{\infty}} \GL(V_i)$ yields~\eqref{eq:kz_hamiltonians}. 
With this end in mind we now recall the definition of the quantum Hamiltonian reduction of an associative algebra (see~\cite{etingov_2007_calogero_moser_systems}). 

Start from an abstract viewpoint: let $A$ be an associative algebra, and $\mathfrak{h}$ a Lie algebra. 

\begin{defn}
	An $\mathfrak{h}$-action on $A$ is morphism $\widehat{\xi} \colon \mathfrak{h} \longrightarrow \Der(A)$ of Lie algebras. 
	A quantum comoment map for $\widehat{\xi}$ is a morphism $\widehat{\mu}^* \colon U(\mathfrak{h}) \longrightarrow A$ of associative algebras whose restriction to $\mathfrak{h}$ lifts $\widehat{\xi}$ through the adjoint action of $A$ on itself:
	\begin{center}
		\begin{tikzpicture}[> = to]
			\node (a) at (0,0) {$\mathfrak{h}$};
			\node (b) at (3,0) {$\Der(A)$};
			\node (c) [above of = b] {$A$};
			\path
			(a) edge node[below]{$\widehat{\xi}$} (b)
			(a) edge [densely dashed] node [above left]{$\left.\widehat{\mu}^*\right|_{\mathfrak{h}}$} (c)
			(c) edge node [right]{$\ad$}(b);
		\end{tikzpicture}
	\end{center}
\end{defn}

The quantum action is then uniquely determined by $\widehat{\xi}(x).a = [\widehat{\mu}(x),a]$, for $x \in \mathfrak{h}$ and $a \in A$.

Let now $\widehat{\mu^*}$ be the quantum comoment map for a $\mathfrak{h}$ action on the associative algebra $A$, and $\widehat{\mathfrak{I}} \subseteq U(\mathfrak{h})$ a two-sided ideal.

\begin{defn}
	The quantum Hamiltonian reduction of $A$ with respect to the quantum comoment map $\widehat{\mu}^*$ and the ideal $\widehat{\mathfrak{I}}$ is the quotient ring 
	\begin{equation}
		R_q(A,\mathfrak{h},\widehat{\mathfrak{I}}) \coloneqq A^{\mathfrak{h}} \big\slash \widehat{\mathfrak{J}}^{\mathfrak{h}} \, ,
	\end{equation}
	where $A^{\mathfrak{h}} \coloneqq \left.\big\{b \in A \right| \bigl[\widehat{\mu}^*(\mathfrak{h}),b\bigr] = 0\big\}$ is the $\mathfrak{h}$-invariant part, $\widehat{\mathfrak{J}} \subseteq A$ is the left ideal generated by $\widehat{\mu}^*(\widehat{\mathfrak{I}})$ inside $A$, and $\widehat{\mathfrak{J}}^{\mathfrak{h}} \coloneqq \widehat{\mathfrak{J}} \cap A^{\mathfrak{h}}$.
\end{defn}

To apply this to the case at hand, consider the Lie group $G^{\infty} = \prod_{i \in I^{\infty}} \GL(V_i)$ and its Lie algebra $\mathfrak{g}^{\infty} \simeq (\mathfrak{g}^{\infty})^*$. 
Then analogous computations to those of Proposition~\ref{prop:quantum_comoment} provides a quantum comoment map $\widehat{\mu}^*_{\infty} \colon U\big((\mathfrak{g}^{\infty})^*\big) \longrightarrow A = W(\mathbb{M}^*,\{\cdot,\cdot\})$. 
Namely, if one denotes $(S_i)_{jk} \in \mathfrak{gl}(V_i)$ the coordinate functions, then one can consider
\begin{equation}
	\label{eq:global_quantum_moment_2}
	\widehat{\mu}^*_{\infty}\big((\widehat{S}_i)_{jk}\big) \coloneqq - \sum_m (\widehat{P}_i)_{jm} \cdot (\widehat{Q}_i)_{mk} \, .
\end{equation}

Choose any ideal $\widehat{\mathfrak{I}} \subseteq U\big((\mathfrak{g}^{\infty})^*\big)$, and define the quantum Hamiltonian reduction $R(A,\mathfrak{g}^{\infty},\widehat{\mathfrak{I}})$, as well as the canonical projection $\pi_{\widehat{\mathfrak{I}}} \colon A^{\mathfrak{g}^{\infty}} \longrightarrow R(A,\mathfrak{g}^{\infty},\widehat{\mathfrak{I}})$. 
To define the reduction of the quantum Hamiltonians one can provide an easy quantisation of Corollary~\ref{cor:invariance}. 
Recall $\widehat{H} = G^{\infty} \times \GL(W^0)$ is the global group acting on $\mathbb{M}$ by simultaneous base changing.

\begin{lem}
	\label{lem:quantum_invariance}
	The simply-laced quantum Hamiltonians of Definition~\ref{def:level_1_slqc} are invariant for the $\widehat{H}$-action on $A$. 
\end{lem}

\begin{proof}
	More generally let $R$ be a ring, and consider the standard trace on the space of square matrices with coefficients in $R$. Then the trace is invariant under conjugation with respect to matrices whose coefficients lies in the centre $Z(R) \subseteq R$. One can then apply this for $R = W(\mathbb{M}^*,\{\cdot,\cdot\})$ and $g \in \widehat{H}$ arbitrary, since $g$ has coefficients in $\mathbb{C} = Z(A)$.
\end{proof}

\begin{defn}
	The element $\pi_{\widehat{\mathfrak{I}}}(\rho_1(\widehat{H}_i)) \in R(A,\mathfrak{g}^{\infty},\widehat{\mathfrak{I}})$ is the quantum Hamiltonian reduction of the operator $\rho_1(\widehat{H}_i)$ at the ideal $\widehat{\mathfrak{I}}$.
\end{defn}

Now one can conclude as in the previous \S~\ref{sec:reduced_star=schlesinger}. 
Assume again to have chosen $V_i = W^0$ for all $i \in I^{\infty}$, getting the identifications $G^{\infty} \simeq \GL(W^0)^m$, $\mathfrak{g}^{\infty} \simeq \mathfrak{g}^m$, $U(\mathfrak{g}^*)^{\otimes m} \simeq U\big((\mathfrak{g}^{\infty})^*\big)$. 
Also, an ideal $\widehat{\mathfrak{I}} \subseteq U\big((\mathfrak{g}^{\infty})^*\big)$ now corresponds to $\widehat{\mathfrak{I}}' \subseteq U(\mathfrak{g}^*)^{\otimes m}$, providing an isomorphism 
\begin{equation}
	U\big((\mathfrak{g}^{\infty})^*\big) \big\slash \widehat{\mathfrak{I}} \simeq U(\mathfrak{g}^*)^{\otimes m} \big\slash \widehat{\mathfrak{I}}' \, .
\end{equation}

\begin{prop}
	\label{prop:natural_quantum_morphism_kz}
	There is a natural injective ring morphism
	\begin{equation}
		\widehat{\varphi} \colon U(\mathfrak{g}^*)^{\otimes m} \big\slash \widehat{\mathfrak{I}}' \longrightarrow R(A,\mathfrak{g}^{\infty},\widehat{\mathfrak{I}}) \, ,
	\end{equation}
	induced by the composition of the map $\widehat{\mu}^*$ of~\eqref{eq:global_classical_moment} with the canonical projection $\pi_{\widehat{\mathfrak{I}}}$ on the quantum Hamiltonian reduction.
\end{prop}

\begin{proof}
	Again, the nontrivial point is that the image of $\widehat{\mu}^*$  is contained in the invariant part $A^{\mathfrak{g}^{\infty}}$. 
	If this were so, then
	\begin{equation}
		\pi_{\widehat{\mathfrak{I}}} \circ \widehat{\mu}^* \colon U(\mathfrak{g}^*)^{\otimes m} \longrightarrow R(A,\mathfrak{g}^{\infty},\widehat{\mathfrak{I}})
	\end{equation}
	would be well defined and induce the desired injective morphism $\widehat{\varphi}$.

	Finally, the proof that $\bigl[\Imm(\widehat{\mu}^*),\Imm(\widehat{\mu}_{\infty}^*)\bigr] = 0$ can be given in coordinates, using formul\ae{}~\eqref{eq:global_quantum_moment} and~\eqref{eq:global_quantum_moment_2}. 
	If one fixes $i \in I^{\infty}$, then 
	\begin{equation}
		\begin{split}
			\Big[\widehat{\mu}^*\big((\widehat{R}_i)_{jk}\big)&,\widehat{\mu}_{\infty}^*\big((\widehat{S}_i)_{lm}\big)\Big] = -\sum_{n,o} \bigl[(\widehat{Q}_i)_{jn} \cdot (\widehat{P}_i)_{nk},(\widehat{P}_i)_{lo} \cdot (\widehat{Q}_i)_{om}\bigr] \\
			&= \sum_{n,o} \delta_{jo}\delta_{nl}(\widehat{P}_i)_{nk} \cdot (\widehat{Q}_i)_{om} - \delta_{nm}\delta_{ok}(\widehat{P}_i)_{lo} \cdot (\widehat{Q}_i)_{jn} \\
			&= (\widehat{P}_i)_{lk} \cdot (\widehat{Q}_i)_{jm} - (\widehat{P}_i)_{lk} \cdot (\widehat{Q}_i)_{jm} = 0 \, .
		\end{split}
	\end{equation}
	The case where $i \neq j \in I^{\infty}$ is trivial, as $\{(Q_i)_{kl},(P_j)_{mn}\} = 0$ for all $i \neq j$ and for all $k,l,m,n$, hence their Weyl quantisations commute in the Weyl algebra.
\end{proof}

All the preparation has been done to show that~\eqref{eq:kz_hamiltonians} is the quantum Hamiltonian reduction of the simply-laced quantum Hamiltonians.

\begin{thm}
	One has 
	\begin{equation}
		\widehat{\varphi}\big([\widehat{H}_i^{\KZ}]_{\widehat{\mathfrak{I}'}}\big) = \pi_{\widehat{\mathfrak{I}}}(\rho_1(\widehat{H_i})) \, ,
	\end{equation}
	for all $i \in I^{\infty}$, where $[\widehat{H}_i^{\KZ}]_{\widehat{\mathfrak{I}}'}$ is the projection of the KZ Hamiltonian to the quotient ring $U(\mathfrak{g}^*)^{\otimes m} \big\slash \widehat{\mathfrak{I}}'$.
\end{thm}

This follows from Proposition~\ref{prop:pull_back_quantum}. 
Hence indeed the quantum Hamiltonian reduction of this particular case of the simply-laced quantum connection yields the KZ connection---at any choice of ideal.

\section{The DMT connection and the star with dual reading}
\label{sec:reduction_DMT}

In this section we show that the connection of De Concini and Millson--Toledano~Laredo~\cite{millson_toledano_laredo_2005_casimir_connection} (DMT) is semiclassically equivalent to the quantum Hamiltonian reduction of the simply-laced quantum connection for the Harnad-dual picture~\cite{harnad_1994_dual_isomonodromic_deformations} of the previous section. 

\subsection{Simply-laced quantum connection of a star: Harnad-dual version}
\label{sec:slqc_dual_star}

Consider here $k = 2$, $a(J) = \{\infty,0\} \subseteq \mathbb{C} \cup \{\infty\}$ and $T^{\infty} = 0$, from the general setup of \S~\ref{sec:classical_systems}.
The graph $\mathcal{G}$ is a star centred at the node $\infty$, as in Figure~\ref{fig:KZ_DMT}. We attach finite-dimensional spaces $\{V_{\infty},V_i\}_{i \in I^0}$ to the nodes $I = \{\infty\} \coprod I^0$ of $\mathcal{G}$. 
Then we set $W^{\infty} = V_{\infty}$, $U^{\infty} = W^0 = \bigoplus_{i \in I^0} V_i$ and $V = W^{\infty} \oplus W^0$. 
The symplectic phase-space $(\mathbb{M},\omega_a)$ is
\begin{equation}
	\mathbb{M} = \Hom(W^0,W^{\infty}) \oplus \Hom(W^{\infty},W^0) \, ,
\end{equation}
with symplectic form $\omega_a = \Tr(dQ \wedge dP)$, and the space of times becomes $\mathbf{B} = \mathbb{C}^{I^0} \setminus \{\diags\}$. 
Write $T^0 = \sum_{i \in I^0} t_i\Id_i$, where $\Id_i$ is the idempotent for $V_i \subseteq W^0$---so that $\{t_i\}_{i \in I^0} \in \mathbf{B}$. 
A generic element of $\mathbb{M}$ looks like $\Gamma = 
\begin{pmatrix} 
    & Q \\
    P &
\end{pmatrix} 
\in \End(V)$, and the special example of meromorphic connections~\eqref{eq:meromorphic_connection_slqc} coded by these data are 
\begin{equation}
	\label{eq:dmt_meromorphic_connection}
	\nabla = d - \left(T^0 + \frac{QP}{z}\right)dz \, .
\end{equation}

\begin{rem}
	These are indeed the Harnad-dual of the rational differential operators~\eqref{eq:kz_meromorphic_connection}, as discussed in~\cite[Appendix B]{boalch_2012_simply_laced_isomonodromy_systems}, but with a change of notation: after the permutation 
	\begin{equation}
		(W^0,W^{\infty},T^0,T^{\infty},Q,P) \longmapsto (W^{\infty},W^0,-T^{\infty},T^0,-P,Q) 
	\end{equation}
	we rename all terms to keep consistency with~\eqref{eq:meromorphic_connection_slqc}. 
	Namely we insist that $W^0 = U^{\infty}$ be the fibre of the trivial vector bundle on which the meromorphic connections are defined, and that the spectrum of $T^0$ carries the irregular times (cf.~\cite[\S~8.3]{boalch_2012_simply_laced_isomonodromy_systems}). 
\end{rem}

If one sets $\widetilde{QP} = \ad_{T^0}^{-1}\bigl[dT^0,QP\bigr]$, then the isomonodromic deformations of~\eqref{eq:dmt_meromorphic_connection} are controlled by the simply-laced Hamiltonian system
\begin{equation}
	\varpi = \frac{1}{2}\Tr\big(\widetilde{QP}QP\big) \, .
\end{equation}
This Hamiltonian system is defined on the trivial symplectic fibration $\mathbb{M} \times \mathbf{B} \to \mathbf{B}$, and spells out as
\begin{equation}
	\label{eq:slims_dual_star} 
	H_i \coloneqq \langle \varpi,\partial_{t_i}\rangle = \sum_{i \neq j \in I^0} \frac{\Tr(P_jQ_jP_iQ_i)}{t_i - t_j} \, .
\end{equation}

The universal simply-laced quantum connection of Definition~\ref{def:slqc} specialises to
\begin{equation}
	\widehat{\nabla}= d - \widehat{\varpi} = d - \sum_{i \in I^0} \left(\sum_{j \neq i \in I^0} \frac{\Tr(\widehat{P}_j\widehat{Q}_j\widehat{P}_i\widehat{Q}_i)}{t_i - t_j} \cdot \hslash^4\right)dt_i \, .
\end{equation}
This is a connection on the trivial bundle $\widehat{A} \times \mathbf{B} \to \mathbf{B}$ of noncommutative algebras, where as above $A \coloneqq W(\mathbb{M}^*,\{\cdot,\cdot\})$ and $\widehat{A}$ is obtained from the Rees algebra of $A$ as in Proposition~\ref{prop:rees}. 
The main Theorem~\ref{thm:quantum_flatness} assures that $\widehat{\nabla}$ is strongly flat. 

Finally, one has the explicit development
\begin{equation}
	\label{eq:slqc_dual_star}
	\widehat{H}_i \coloneqq \langle \widehat{\varpi},\partial_{t_i} \rangle = \sum_{i \neq j} \frac{\Tr(\widehat{P}_j\widehat{Q}_j\widehat{P}_i\widehat{Q}_i)}{t_i - t_j} \cdot \hslash^4 \, ,
\end{equation}
for the universal simply-laced quantum Hamiltonian at the node $i \in I^0$. 
The simply-laced quantum Hamiltonians are instead the functions $\rho_1(\widehat{H}_i) \colon \mathbf{B} \to A$, as in Definition~\ref{def:level_1_slqc}.

\subsection{DMT connection}
\label{sec:DMT}

In this section we recall the construction of the DMT connection~\cite{millson_toledano_laredo_2005_casimir_connection}, and we give an explicit development of the associated quantum Hamiltonians in the case of $\mathfrak{g} = \mathfrak{gl}_n(\mathbb{C})$. 

Consider a simple Lie algebra $\mathfrak{g}$, and choose a Cartan subalgebra $\mathfrak{t} \subseteq \mathfrak{g}$ with associated root system $\mathcal{R} = \mathcal{R}(\mathfrak{g},\mathfrak{t}) \subseteq \mathfrak{t}^*$. 
Let
\begin{equation}
	\mathfrak{t}_{\reg} \coloneqq \mathfrak{t} \setminus \bigcup_{\alpha \in \mathcal{R}} \Ker(\alpha)
\end{equation}
be the regular part of the Cartan algebra, and denote $K \colon \mathfrak{g} \otimes \mathfrak{g} \to \mathbb{C}$ the Cartan--Killing form of $\mathfrak{g}$. 

One now defines a strongly flat connection $\widehat{\nabla}^{\DMT}$ on the bundle $U(\mathfrak{g}) \times \mathfrak{t}_{\reg} \to \mathfrak{t}_{\reg}$, as follows. 
For all $\alpha \in \mathcal{R}$ choose an $\mathfrak{sl}_2$-triplet of vectors $e_{\alpha} \in \mathfrak{g}_{\alpha}$, $f_{\alpha} \in \mathfrak{g}_{-\alpha}$, $h_{\alpha} = [e_{\alpha},f_{\alpha}] \in \mathfrak{t}$, and then define the \emph{DMT connection} as
\begin{equation}
	\label{eq:dmt_connection}
	\widehat{\nabla}^{\DMT} = d - \widehat{\varpi}^{\DMT} \coloneqq d - \sum_{\alpha \in R} \frac{K(\alpha,\alpha)}{2}(\widehat{e}_{\alpha} \cdot \widehat{f}_{\alpha} + \widehat{f}_{\alpha} \cdot \widehat{e}_{\alpha})\frac{d\alpha}{\alpha} \, ,
\end{equation}
where $K(\alpha,\alpha) \in \mathbb{R}_{> 0}$ is the length squared of the root $\alpha$, computed using the dual of the Killing form (which we also denote $K$). 

We now specialise the DMT connection to the case of $\mathfrak{g} \coloneqq \mathfrak{gl}_n(\mathbb{C})$, even though this algebra is not simple. 
Indeed, if we consider an invariant nondegenerate symmetric bilinear form $K$ on $\mathfrak{g}$ then we may perform the same construction as above, and we will use $K(A,B) \coloneqq \frac{1}{2}\Tr(AB)$. 
Next we choose $\mathfrak{t} \subseteq \mathfrak{g}$ to be the subalgebra of diagonal matrices, so that the roots are given by $\alpha_{ij}(\diag(t_1, \dotsc, t_n)) \coloneqq t_i - t_j$ for $1 \leq i \neq j \leq n$. 
Moreover, if $(e_{ij})_{ij}$ is the canonical basis then 
\begin{equation}
	\mathfrak{g}_{\alpha_{ij}} = \spann_{\mathbb{C}} \{e_{ij}\}, \qquad \mathfrak{g}_{-\alpha_{ij}} = \mathfrak{g}_{\alpha_{ji}} = \spann_{\mathbb{C}} \{e_{ji}\} \, , 
\end{equation}
so that $h_{\alpha_{ij}} = [e_{ij},e_{ji}] = e_{ii} - e_{jj} \in \mathfrak{t}$, and the length squared of all roots equals 2. 

Finally, notice that if one introduces the global coordinates $\{t_i\}_i$ on $\mathfrak{t}_{\reg}$ obtained from the restriction of the standard coordinates on $\mathfrak{t} \simeq \mathbb{C}^n$, then 
\begin{equation}
	d\log(\alpha_{ij}) = d\big(\log(t_i - t_j)\big) = \frac{dt_i - dt_j}{t_i - t_j} \, .
\end{equation}
Hence one has the following expansion of the \emph{DMT Hamiltonians} for $\mathfrak{g} = \mathfrak{gl}_n(\mathbb{C})$:
\begin{equation}
	\label{eq:dmt_hamiltonians}
	\widehat{H}_i^{\DMT} \coloneqq \langle \widehat{\varpi}^{\DMT}, \partial_{t_i} \rangle = \frac{1}{2}\sum_{j \neq i} \frac{\widehat{e}_{ij} \cdot \widehat{e}_{ji} + \widehat{e}_{ji} \cdot \widehat{e}_{ij}}{t_i - t_j} \, .
\end{equation}

\subsection{Dual Schlesinger system}

In this section we define the dual Schlesinger system and provide an explicit expansion of its Hamiltonians.

Let $G$ be a reductive group with Lie algebra $\mathfrak{g}$. 
Equip $\mathfrak{g}$ with a nondegenerate invariant symmetric bilinear form $K \in \mathfrak{g}^* \otimes \mathfrak{g}^*$, whence $\mathfrak{g} \simeq \mathfrak{g}^*$ carries the Lie--Poisson structure. 
Consider the trivial Poisson fibration $\mathfrak{g} \times \mathfrak{t}_{\reg} \to \mathfrak{t}_{\reg}$, where $\mathfrak{t} \subseteq \mathfrak{g}$ is a Cartan subalgebra. 
Choose then $R, T^0 \in \mathfrak{g}$, with $T^0$ regular semisimple, and consider the meromorphic connection on the trivial principal $G$-bundle over $\mathbb{C} P^1$ defined by the (global) $\mathfrak{g}$-valued 1-form 
\begin{equation}
	\label{eq:meromorphic_connection_dmt}
	\alpha = \left(T^0 + \frac{R}{z}\right)dz \, .
\end{equation}
The isomonodromy equations of such connections admit an Hamiltonian formulation~\cite{boalch_2002_g_bundles_isomonodromy_quantum_weyl_groups}. 
To write down the Hamiltonians set $\widetilde{R} \coloneqq \ad^{-1}_{T^0}[dT^0,R]$, and define the 1-form
\begin{equation}
	\varpi^{\dSch} \coloneqq K\left(R,\widetilde{R}\right) \, .
\end{equation}
Now take the global coordinates $(t_i)_i$ on $\mathfrak{t}_{\reg}$ as an open subset of the vector space $\mathfrak{t}$, and define Hamiltonians $H^{\dSch}_i \colon \mathfrak{g} \times \mathfrak{t}_{\reg} \to \mathbb{C}$ controlling the isomonodromy deformations of~\eqref{eq:meromorphic_connection_dmt} by $H^{\dSch}_i \coloneqq \langle \varpi^{\dSch}, \partial_{t_i} \rangle$. 
These are called the \emph{dual Schlesinger Hamiltonians}.

Specialising all this to $G = \GL_n(\mathbb{C})$, $\mathfrak{g} = \mathfrak{gl}_n(\mathbb{C})$ and $K(\cdot,\cdot) = \frac{1}{2}\Tr(\cdot,\cdot)$ one finds
\begin{equation}
	\varpi^{\dSch} = \frac{1}{2}\Tr\big(\widetilde{R}R\big) = \frac{1}{2}\sum_{i \neq j} R_{ij}R_{ji} \frac{dt_i - dt_j}{t_i - t_j} \, ,
\end{equation}
and one has the following expansion of the dual Schlesinger Hamiltonians:
\begin{equation}
	\label{eq:dual_schlesinger_hamiltonians}
	H^{\dSch}_i = \sum_{j \neq i} \frac{R_{ij}R_{ji}}{t_i - t_j} \, .
\end{equation}

\subsection{DMT is a quantisation of dual Schlesinger}
\label{sec:quantum_dual_Schlesinger=DMT}

It is shown in~\cite{boalch_2002_g_bundles_isomonodromy_quantum_weyl_groups} that the DMT connection~\eqref{eq:dmt_connection} is the PBW-quantisation of the dual Schlesinger system. 
In this section we give a proof adapted to our notations in the case of $G = \GL_n(\mathbb{C})$.

\begin{thm}
	One has $\mathcal{Q}_{\PBW}(H_i^{\dSch}) = \widehat{H}_i^{\DMT}$ for all $i \in I^0$.
\end{thm}

\begin{proof}
	Composing the dual Schlesinger Hamiltonians~\eqref{eq:dual_schlesinger_hamiltonians} with the trace-duality turns them into sections of the trivial bundle $\Sym(\mathfrak{g}) \times \mathfrak{t}_{\reg} \to \mathfrak{t}_{\reg}$. 
	Explicitly, these sections read
	\begin{equation}
		(t_i)_i \longmapsto \sum_{j \neq i} \frac{e_{ji} \otimes e_{ij}}{t_i - t_j} \, .
	\end{equation}
	If we fix $i$ and $j$ then the PBW-quantisation~\ref{def:pbw_quantisation} of the numerator is 
	\begin{equation}
		\mathcal{Q}_{\PBW}(e_{ji} \otimes e_{ij}) = \frac{\widehat{e}_{ij} \cdot \widehat{e}_{ji} + \widehat{e}_{ji} \cdot \widehat{e}_{ij}}{2} \, ,
	\end{equation}
	and the result follows by linearity, looking at the expansion~\eqref{eq:dmt_hamiltonians}.
\end{proof}

\subsection{Classical Hamiltonian reduction}
\label{sec:reduced_dual_star=dual_Schlesinger}

In this section we show that the classical Hamiltonian reduction of the simply-laced isomonodromy Hamiltonians~\eqref{eq:slims_dual_star} yields the dual Schlesinger Hamiltonians~\eqref{eq:dual_schlesinger_hamiltonians}.

The first necessary restriction is the following: take $\dim(V_i) = 1$ for all $i \in I^0$. 
In this case $\dim(W^0) = |I^0| \eqqcolon n$, and indeed $\mathbf{B} = \mathbb{C}^n \setminus \{\diags\} \simeq \mathfrak{t}_{\reg}$, where $\mathfrak{t}$ is the standard Cartan subalgebra of $\mathfrak{g} \coloneqq \mathfrak{gl}(W^0) \simeq \mathfrak{gl}_n$. 
As in \S~\ref{sec:reduced_star=schlesinger}, we consider the comoment map for the $\GL(W^0)$-action on the symplectic cotangent bundle $T^*\Hom(W^{\infty},W^0) \simeq \mathbb{M}$, which is given by the matrix product $(Q,P) \longmapsto QP$. 
Up to using the usual trace-duality $\mathfrak{g }\simeq \mathfrak{g}^*$, this is a map
\begin{equation}
	\mu^* \colon \Sym(\mathfrak{g}^*) \longrightarrow A_0 \, .
\end{equation}
To find an explicit formula for it, write again $R_{ij}$ for the coordinate functions on $\mathfrak{gl}(W^0)$.
The map $Q_i \colon W^{\infty} \to V_i$ is a row vector, and similarly $P_j \colon V_j \to W^{\infty}$ is a column vector, hence one may write $(dQ_i)_k$, $(dP_i)_k$ the coordinate functions on $\mathbb{M}$. 
Then 
\begin{equation}
	\label{eq:global_classical_moment_dual}
	\mu^*(R_{ij}) = \sum_k (Q_i)_k(P_j)_k \, ,
\end{equation} 
since the coefficients of the residue $R = QP$ are the numbers $R_{ij} = Q_iP_j = \sum_k (Q_i)_k(P_j)_k$.

We now show that the map~\eqref{eq:global_classical_moment_dual} matches up the simply-laced Hamiltonians~\eqref{eq:slims_dual_star} with the dual Schlesinger Hamiltonians~\eqref{eq:dual_schlesinger_hamiltonians}.

\begin{prop}
	\label{prop:pull_back_classical_2}
	One has $\mu^*(H_i^{\dSch}) = H_i$ for all $i \in I^0$.
\end{prop}

\begin{proof} 
	By linearity, it is enough to show that $\mu^*(R_{ij}R_{ji}) = \Tr(P_jQ_jP_iQ_i)$, which follows from the above formula: 
	\begin{equation}
		\mu^*(R_{ij}R_{ji}) = Q_iP_jQ_jP_i = \Tr(Q_iP_jQ_jP_i) = \Tr(P_jQ_jP_iQ_i) \, .
	\end{equation}
	In the second identity we used the fact that the endomorphism $Q_iP_jQ_jP_i \colon V_i \to V_i$ is just a complex numbers---equal to its trace.
\end{proof}

The same exact steps of~\ref{sec:reduced_star=schlesinger} then show that~\eqref{eq:dual_schlesinger_hamiltonians} is the Hamiltonian reduction of~\eqref{eq:slims_dual_star} with respect to the action of $G^{\infty} = \GL(W^{\infty})$ on $\mathbb{M}$, at any given ideal/orbit. 
Namely, if $\mathfrak{g}^{\infty} = \mathfrak{gl}(W^{\infty})$ acts with comoment map $\mu_{\infty}^* \colon \Sym(\mathfrak{g}^{\infty}) \to A_0$ and $\mathfrak{I} \subseteq \Sym(\mathfrak{g}^{\infty})$ is an ideal, then the Hamiltonian reduction $R(A_0,\mathfrak{g}^{\infty},\mathfrak{I})$ is defined as in~\ref{def:hamiltonian_reduction}, together with a canonical projection 
\begin{equation}
	\pi_{\mathfrak{I}} \colon A_0^{\mathfrak{g}^{\infty}} \longrightarrow R(A_0,\mathfrak{g}^{\infty},\mathfrak{I}) \, .
\end{equation}
This defines the Hamiltonian reduction $\pi_{\mathfrak{I}}(H_i)$ of the simply-laced Hamiltonian~\eqref{eq:slims_dual_star}, and geometrically the choice of the ideal $\mathfrak{I}$ corresponds to fixing a coadjoint $G^{\infty}$-orbit $\mathcal{O} \subseteq (\mathfrak{g}^{\infty})^*$.

Now one needs to find a suitable ideal $\mathfrak{I}' \subseteq \Sym(\mathfrak{g}^*)$ to match up with $\mathfrak{g}$, and this can be done with a further restriction: assume that $W^{\infty} = W^0$, i.e. that $\dim(W^{\infty}) = n$. 
Then there are canonical identification $\Sym(\mathfrak{g}^*) \simeq \Sym(\mathfrak{g}^{\infty})$, and taking the ideal $\mathfrak{I}'$ corresponding to $\mathfrak{I}$ in this isomorphism one constructs a natural injective morphism 
\begin{equation}
	\varphi \colon \Sym(\mathfrak{g}^*) \big\slash \mathfrak{I}' \longrightarrow R(A_0,\mathfrak{g}^{\infty},\mathfrak{I}) \, ,
\end{equation}
induced by the composition of the projection $\pi_{\mathfrak{I}}$ after $\mu^*$. 
The fact that the image of $\Sym(\mathfrak{g}^*)$ inside $A_0$ Poisson-commute with that of $\mu_{\infty}^*$ is due to the fact that the Hamiltonian $\GL(W^0)$- and $G^{\infty}$-actions commute. 
Finally Proposition~\ref{prop:pull_back_classical_2} shows the following.

\begin{thm}
	The classes of the dual Schlesinger Hamiltonians~\eqref{eq:dual_schlesinger_hamiltonians} inside $\Sym(\mathfrak{g}) \big\slash \mathfrak{I}'$ match up with $\pi_{\mathfrak{I}}(H_i) \in R(A_0,\mathfrak{g}^{\infty},\mathfrak{I})$ under the natural correspondence $\varphi$.
\end{thm}

Hence indeed the dual Schlesinger system corresponds to the Hamiltonian reduction of this particular case of the simply-laced isomonodromy system, for every choice of ideal/coadjoint orbit. 

\begin{rem}
	After considering the reduction with respect to the $\GL(W^{\infty})$-action, there is still a residual action at the peripheral nodes. 
	This action is not that of the whole group $\GL(W^0)$, but rather of the subgroup $\prod_{i \in I^0} \GL(V_i) \subseteq \GL(W^0)$, which is a maximal torus. 
	Taking the quotient with respect to the full action yields the moduli space $\mathcal{M}^*_{\dR}$ of isomorphism classes of meromorphic connections~\eqref{eq:dmt_meromorphic_connection} on a trivial holomorphic vector bundle. 
	This moduli space is isomorphic to the moduli space for logarithmic connections of the previous section (the isomorphism being the Harnad duality~\cite{harnad_1994_dual_isomonodromic_deformations}), and thus we have two different descriptions of the same complex symplectic manifold.
\end{rem}

\subsection{Quantum Hamiltonian reduction}
\label{sec:corrected_quantum_reduced_dual_star=DMT}

Here we show that the quantum Hamiltonian reduction of the simply-laced quantum Hamiltonian~\eqref{eq:slqc_dual_star} yields a quantum Hamiltonian system whose semiclassical limit is the same as that of the DMT system~\eqref{eq:dmt_hamiltonians}. 
The idea is again to rephrase the ``quantum'' change of variable $\widehat{R}_{ij} \coloneqq \widehat{Q}_i\widehat{P}_j$ in algebraic terms. 

As in Proposition~\ref{prop:quantum_comoment}, one can construct a quantum comoment map $\widehat{\mu}^* \colon U(\mathfrak{g}^*) \to A$ by showing that the natural formula makes sense, that is:
\begin{equation}
	\widehat{\mu}^*(\widehat{R}_{ij}) = \sum_k (\widehat{Q}_i)_k \cdot (\widehat{P}_j)_k \, .
\end{equation}
Applying this morphism to the DMT Hamiltonians yields
\begin{equation}
	\label{eq:corrected_dmt_hamiltonian}
	\widehat{\mu}^*\big(\widehat{H}_i^{\DMT}\big) = \widehat{\mu}^* \left(\sum_{j \neq i} \frac{\widehat{R}_{ij}\widehat{R}_{ji}}{t_i - t_j}\right) = \sum_{j \neq i} \frac{\Tr(\widehat{Q}_i\widehat{P}_j\widehat{Q}_j\widehat{P}_i)}{t_i - t_j} \, .
\end{equation}

As in the previous section one can prove the quantum Hamiltonian reductions of~\eqref{eq:corrected_dmt_hamiltonian} with respect to the $G^{\infty}$-action---at a two-sided ideal $\widehat{\mathfrak{I}} \subseteq U(\mathfrak{g}^{\infty})$---yields the DMT Hamiltonians~\eqref{eq:dmt_hamiltonians}.
Moreover,~\eqref{eq:corrected_dmt_hamiltonian} is obtained from the simply-laced quantum Hamiltonian $\rho_1(\widehat{H}_1)$ by moving the anchor of the degenerate 4-cycles from their centre to a peripheral node. 
More precisely, if one replaces the quantum isomonodromy potential with 
\begin{equation}
	\widehat{W}'_i \coloneqq \sum_{m \in I_i \setminus \{i\}, j,l \in I \setminus I_i} \frac{\alpha_{ij}\alpha_{jm}\alpha_{ml}\underline{\alpha_{li}}}{t_i - t_m} \, ,
\end{equation}
then by construction one has $\widehat{\mu}^*\big(\widehat{H}_i^{\DMT}\big) = \Tr(\widehat{W}'_i)$.

We now proceed to show that these Hamiltonians have the same semiclassical limit of the simply-laced quantum Hamiltonians~\eqref{eq:slqc_dual_star}. 
To state this properly we introduce the formal deformation parameter $\hslash$ into the picture. 
Let then $\Rees(A) \subseteq A[\hslash]$ be the Rees algebra of $A$ as in Definition~\ref{def:rees_algebra}, and $\widehat{A} \subseteq A \llbracket \hslash \rrbracket$ the topologically free $\mathbb{C} \llbracket \hslash \rrbracket$-algebra defined in~\ref{prop:rees}. 
Then the system~\eqref{eq:corrected_dmt_hamiltonian} can be encoded into the connection
\begin{equation}
	\widehat{\nabla}' \coloneqq d - \sum_{i \in I^0} \widehat{H}'_i dt_i \, ,
\end{equation}
where $H_i' \coloneqq \Tr_{\hslash}(\widehat{W}_i')$. This connection is defined on the trivial bundle $\widehat{A} \times \mathbf{B} \to \mathbf{B}$ as the simply-laced quantum connection, and we compare the two.

\begin{thm}
	\label{prop:quantum_correction}
	The $\widehat{A}$-valued one-form $\widehat{\nabla} - \widehat{\nabla}'$ on $\mathbf{B}$ vanishes in the semiclassical limit.
\end{thm}

\begin{proof}
	We must show that the element
	\begin{equation}
		\langle\widehat{\nabla} - \widehat{\nabla}',\partial_{t_i}\rangle = \sum_{j \neq i} \frac{\Tr(\widehat{Q}_i\widehat{P}_j\widehat{Q}_j\widehat{P}_i) - \Tr(\widehat{P}_j\widehat{Q}_j\widehat{P}_i\widehat{Q}_i)}{t_i - t_j} \cdot \hslash^4 \in \Rees(A) \subseteq \widehat{A}
	\end{equation}
	lies in the kernel of the semiclassical limit~\eqref{eq:semiclassical_limit} for all $\{t_i\}_i \in \mathbf{B}$. 
	More is true: all summands have vanishing semiclassical limit because 
	\begin{equation}
		\sigma_4 \Big(\Tr(\widehat{Q}_i\widehat{P}_j\widehat{Q}_j\widehat{P}_i)\Big) = \Tr(Q_iP_jQ_jP_j) = \Tr(P_jQ_jP_iQ_i) = \sigma_4\Big(\Tr(\widehat{P}_j\widehat{Q}_j\widehat{P}_i\widehat{Q}_i)\Big) \, ,
	\end{equation}
	in the identification $\gr(A) \simeq A_0$.
\end{proof}

Hence indeed one can add a semiclassically vanishing term to this particular case of the simply-laced quantum connection, so that the quantum Hamiltonian reduction equals the DMT connection---at any choice of ideal.

\section{The FMTV connection and the generic complete bipartite graph}
\label{sec:reduction_FMTV}

In this section we combine the results of \S~\ref{sec:reduction_KZ} and \S~\ref{sec:reduction_DMT} to prove that the quantum Hamiltonian reduction of the simply-laced quantum connection for a generic complete bipartite graph is semiclassically equivalent to the FMTV connection of Felder--Markov--Tarasov--Varchenko~\cite{felder_markov_tarasov_varchenko_2000_dynamical_connection}. 

\subsection{Simply-laced quantum connection of a bipartite graph}
\label{sec:slqc_bipartite}

The constructions of \S~\ref{sec:slqc_star} and \S~\ref{sec:slqc_dual_star} generalise as follows. 
One still has $k = 2$ and $a(J) = \{\infty,0\}$, but $\mathcal{G}$ is now an arbitrary bipartite graph on nodes $I = I^0 \coprod I^{\infty}$. 
The base space of times becomes 
\begin{equation}
	\mathbf{B} = \mathbb{C}^{I^{\infty}} \setminus \{\diags\} \times \mathbb{C}^{I^0} \setminus \{\diags\} \, .
\end{equation}
Next we attach finite-dimensional vector spaces $\{V^{\infty}_i\}_{i \in I^{\infty}}$ and $\{V^0_i\}_{i \in I^0}$ to the nodes of $\mathcal{G}$, and set $W^{\infty} = \bigoplus_{i \in I^{\infty}} V^{\infty}_i$ and $W^0 = \bigoplus_{i \in I^0} V^0_i$, so that the space of representations of $\mathcal{G}$ in $V = W^{\infty} \oplus W^0$ is
\begin{equation}
	\mathbb{M} = \Hom(W^{\infty},W^0) \oplus \Hom(W^0,W^{\infty}) \, .
\end{equation}
It comes with the symplectic form $\omega_a = \Tr(dQ \wedge dP)$, where $Q \colon W^{\infty} \to W^0$ and $P \colon W^0 \to W^{\infty}$ are linear maps with components $Q_{ij} \colon V^{\infty}_j \to V^0_i$ and $P_{ij} \colon V^0_j \to V^{\infty}_i$. 
There is thus a coarser decomposition with $Q_i \colon V^{\infty}_i \to W^0$ and $P_i \colon W^0 \to V^{\infty}_i$. 
Finally, write $T^0 = \sum_{i \in I^0} t^0_i \Id^0_i$ and $T^{\infty} = \sum_{i \in I^{\infty}} t^{\infty}_i \Id^{\infty}_i$, where $\Id^0_i$ is the idempotent for $V^0_i \subseteq W^0$ and $\Id^{\infty}_i$ the idempotent for $V^{\infty}_i \subseteq W^{\infty}$. 

These data parametrise meromorphic connections as~\eqref{eq:meromorphic_connection_slqc}, which in this case specialise to
\begin{equation}
	\label{eq:fmtv_meromorphic_connection}
	\nabla = d - \left(T^0 + \sum_{i \in I^{\infty}} \frac{Q_iP_i}{z - t^{\infty}_i}\right)dz \, ,
\end{equation}
and are defined on the trivial vector bundle $W^0 \times \mathbb{C} P^1 \to \mathbb{C} P^1$. 

Setting $\widetilde{PQ} = \ad_{T^{\infty}}^{-1}\bigl[dT^{\infty},PQ\bigr]$ and $\widetilde{QP} = \ad_{T^0}^{-1}\bigl[dT^0,QP\bigr]$, then the isomonodromic deformations of~\eqref{eq:fmtv_meromorphic_connection} are coded by the simply-laced Hamiltonian system
\begin{equation}
	\varpi = \frac{1}{2}\Tr\big(\widetilde{PQ}PQ\big) + \frac{1}{2}\Tr\big(\widetilde{QP}QP\big) \, ,
\end{equation}
defined on the trivial symplectic fibration $\mathbb{M} \times \mathbf{B} \to \mathbf{B}$. 
The simply-laced Hamiltonians spell out as
\begin{equation}
	\begin{split}
		\label{eq:slims_bipartite}
		H^{\infty}_i \coloneqq \langle \varpi,\partial_{t^{\infty}_i} \rangle &= \sum_{k \in I^{\infty} \setminus \{i\}, j,l \in I^0} \frac{\Tr(P_{il}Q_{lk}P_{kj}Q_{ji})}{t^{\infty}_i - t^{\infty}_k} + \sum_{j \in I^0} t^0_j\Tr(P_{ij}Q_{ji}), \\
		H^0_j \coloneqq \langle \varpi,\partial_{t^0_j}\rangle &= \sum_{l \in I^0 \setminus \{j\},i,k \in I^{\infty}} \frac{\Tr(Q_{ji}P_{il}Q_{lk}P_{kj})}{t^0_j - t^0_l} + \sum_{i \in I^{\infty}} t^{\infty}_i\Tr(Q_{ji}P_{ij}) \, .
	\end{split}
\end{equation}

The universal simply-laced quantum connection of Definition~\ref{def:slqc} specialises to
\begin{equation}
	\widehat{\nabla} = d - \widehat{\varpi} = d - \sum_{i \in I^{\infty}} \widehat{H}^{\infty}_i dt^{\infty}_i - \sum_{j \in I^0} \widehat{H}^j_0 dt^0_j \, ,
\end{equation}
and is defined on the trivial vector bundle $\widehat{A} \times \mathbf{B} \to \mathbf{B}$, where $A \coloneqq W(\mathbb{M}^*,\{\cdot,\cdot\})$ is the Weyl algebra and $\widehat{A}$ is as in Proposition~\ref{prop:rees}. 
The main Theorem~\ref{thm:quantum_flatness} assures that $\widehat{\nabla}$ is strongly flat, and the explicit formul\ae{} for the universal simply-laced quantum Hamiltonians are 
\begin{equation}
	\begin{split}
		\label{eq:slqc_bipartite}
		\widehat{H}^{\infty}_i = &\sum_{k \in I^{\infty} \setminus \{i\}, j \neq l \in I^0} \frac{\Tr(\widehat{P}_{il}\widehat{Q}_{lk}\widehat{P}_{kj}\widehat{Q}_{ji})}{t^{\infty}_i - t^{\infty}_k} \cdot \hslash^4 + \sum_{k \in I^{\infty} \setminus \{i\},j \in I^0} \frac{\Tr(\widehat{Q}_{jk}\widehat{P}_{kj}\widehat{Q}_{ji}\widehat{P}_{ij})}{t^{\infty}_i - t^{\infty}_k} \cdot \hslash^4 \\
		&+ \frac{1}{2}\sum_{j \in I^0} t^0_j \Big(\Tr(\widehat{P}_{ij}\widehat{Q}_{ji}) + \Tr(\widehat{Q}_{ji}\widehat{P}_{ij})\Big) \cdot \hslash^2, \\
		\widehat{H}^0_j = &\sum_{l \in I^0 \setminus \{j\}, i \neq k \in I^{\infty}} \frac{\Tr(\widehat{Q}_{ji}\widehat{P}_{il}\widehat{Q}_{lk}\widehat{P}_{kj})}{t^0_j - t^0_l} \cdot \hslash^4 + \sum_{l \in I^0 \setminus \{j\},i \in I^{\infty}} \frac{\Tr(\widehat{P}_{jl}\widehat{Q}_{li}\widehat{P}_{ji}\widehat{Q}_{ji})}{t^0_j - t^0_l} \cdot \hslash^4 \\
		&+ \frac{1}{2}\sum_{i \in I^ {\infty}} t^{\infty}_i \Big(\Tr(\widehat{Q}_{ji}\widehat{P}_{ij}) + \Tr(\widehat{P}_{ij}\widehat{Q}_{ji})\Big) \cdot \hslash^2 \, .
	\end{split}
\end{equation}
The simply-laced quantum Hamiltonians are the functions $\rho_1(\widehat{H}^{\infty}_i), \rho_1(\widehat{H}^0_j) \colon \mathbf{B} \to A$, as in Definition~\ref{def:level_1_slqc}.

\subsection{Classical Hamiltonian reduction and the JMMS system}
\label{sec:reduced_bipartite=JMMS}

In this section we show that the classical Hamiltonian reduction of the simply-laced Hamiltonians~\eqref{eq:slims_bipartite} yields the system of Jimbo--Miwa--M\^ori--Sato~\cite{jimbo_miwa_mori_sato_1980_density_matrix_bose_gas} (JMMS). We start by providing explicit formul\ae{} for the Hamiltonians of the JMMS system.

It is shown in~\cite{boalch_2012_simply_laced_isomonodromy_systems} that~\eqref{eq:slims_bipartite} controls isomonodromic deformation equations which correspond to the lifted equations of~\cite{jimbo_miwa_mori_sato_1980_density_matrix_bose_gas}, i.e. Equation~A.5.9 of op. cit. 
Moreover, the change of variable $R_i = Q_iP_i$ provides the JMMS equations themselves, i.e. Equations ~4.44 and~A.5.1 of op. cit. 
We now rephrase this fact in our notation.

First, similarly to \S~\ref{sec:reduced_dual_star=dual_Schlesinger}, we restrict $T^0$ to have simple spectrum in order to recover the setup of~\cite{jimbo_miwa_mori_sato_1980_density_matrix_bose_gas}. 
Let $W^0$ be a vector space of dimension $n$, and consider elements $R_i \in \mathfrak{g} \coloneqq \mathfrak{gl}(W^0) \simeq \mathfrak{gl}_n(\mathbb{C})$. Set $\mathfrak{t} \subseteq \mathfrak{g}$ to be the standard Cartan subalgebra, and choose $T^0 = \diag(t^0_1,\dotsc,t^0_n) \in \mathfrak{t}_{\reg}$. 
Let also $\Conf_m(\mathbb{C}) \simeq \mathbb{C}^m \setminus \{\diags\}$ be the space of configurations of $m$-tuples of ordered points in the complex plane, and write $\{t^{\infty}_i\}_i$ such an $m$-tuple. 

The JMMS system is a time-dependent classical Hamiltonian system controlling the isomonodromic deformations of meromorphic connections of the form
\begin{equation}
	\nabla = d - \left(T^0 + \sum_{i = 1}^m \frac{R_i}{z - t^{\infty}_i}\right)dz \, ,
\end{equation}
on the trivial vector bundle $W^0 \times \mathbb{C} P^1 \to \mathbb{C} P^1$. 
The isomonodromy problem is the following: let $T^0$ and $\{t^{\infty}_i\}_i$ vary inside the product $\mathbf{B} \coloneqq \mathfrak{t}_{\reg} \times \Conf_m(\mathbb{C})$, and look for residues $R_i \in \mathfrak{g}$ such that the (extended) monodromy data of the new connection are the same as those of $\nabla$. 
This is the combination of the isomonodromy problems of \S~\ref{sec:reduction_KZ} and \S~\ref{sec:reduction_DMT}, and it also admits an Hamiltonian formulation via the \emph{JMMS Hamiltonians}, defined on the trivial Poisson fibration $\mathfrak{g}^m \times \mathbf{B} \to \mathbf{B}$.

The explicit formul\ae{} for the JMMS Hamiltonians are
\begin{equation}
	\label{eq:jmms_hamiltonians}
	\begin{split}
		H_i^{\JMMS,\infty} &= \sum_{k \neq i} \frac{\Tr(R_iR_k)}{t^{\infty}_i - t^{\infty}_k} + \Tr(R_iT^0) \, , \\
		H_j^{\JMMS,0} &= \sum_{k \neq j} \sum_{i,p}\frac{(R_i)_{jk}(R_p)_{kj}}{t^0_j - t^0_k} + \sum_i t^{\infty}_i \Tr(R_ie_{jj}) \, ,
	\end{split}
\end{equation}
where $e_{jj} \in \mathfrak{t}$ is a diagonal element of the canonical basis of $\mathfrak{g}$. 
Indeed, this is precisely Equation~A.5.13 of~\cite{jimbo_miwa_mori_sato_1980_density_matrix_bose_gas}, written in our notation.

\begin{rem}
	We have written down this half-expanded form to make it apparent that this is Equation~A.5.13 of op. cit., replacing $A_i, A_{\infty}, c_i, a_j$ with $R_i, T^0, t^{\infty}_i, t^0_j$, respectively. 
	The full expansion of the linear terms are
	\begin{equation}
		\Tr(R_iT^0) = \sum_{j,l} (R_i)_{jl}T^0_{lj} = \sum_{j,l} (R_i)_{jl}\delta_{lj}t^0_j = \sum_j t^0_j(R_i)_{jj} \, ,
	\end{equation}
	and
	\begin{equation}
		\sum_i t^{\infty}_i\Tr(R_ie_{jj}) = \sum_i\sum_{k,l} t^{\infty}_i(R_i)_{kl}(e_{jj})_{lk} = \sum_i\sum_{k,l} t^{\infty}_i(R_i)_{kl}\delta_{jk}\delta_{jl} = \sum_i t^{\infty}_i(R_i)_{jj} \, .
	\end{equation}

	The leading term of $H^{\JMMS,\infty}_i$ provides the Schlesinger Hamiltonian~\eqref{eq:schlesinger_hamiltonians}.
	Similarly the leading term of $H^{\JMMS,0}_j$ provides a generalisation of the dual Schlesinger Hamiltonian~\eqref{eq:dual_schlesinger_hamiltonians}. 
	This generalisation amounts to the fact that now one allows for several simple poles---instead of just one.  
\end{rem}

To relate~\eqref{eq:jmms_hamiltonians} with the Hamiltonian reduction of~\eqref{eq:slims_bipartite} we have to consider the special case of the simply-laced isomonodromy systems where $\dim(V^0_j) = 1$ for all $j \in I^0$, so that $|I^0| = \dim(W^0) = n$, and then $\mathbb{C}^{I^0} \setminus \{\diags\}$ is identified with the regular part $\mathfrak{t}_{\reg}$ of the standard Cartan subalgebra $\mathfrak{t} \subseteq \mathfrak{g} = \mathfrak{gl}(W^0)$. 
Choose now $i \in I^{\infty}$, and consider the product maps $\mu_i \colon T^*\Hom(V^{\infty}_i,W^0) \to \mathfrak{g}$ sending $(Q_i,P_i)$ to $Q_iP_i$, as done in \S~\ref{sec:reduced_star=schlesinger}. 
This is the moment map for the Hamiltonian $\GL(W^0)$-action on $T^*\Hom(V^{\infty}_i,W^0)$, and then we consider
\begin{equation}
	\mu \coloneqq \bigoplus_{i \in I^{\infty}} \mu_i \colon \mathbb{M} \longrightarrow \mathfrak{g}^m, \qquad \mu \colon (Q,P) \longmapsto (Q_iP_i)_{i \in I^{\infty}} \, .
\end{equation}
The pull-back $\mu^* \colon \Sym(\mathfrak{g}^*)^{\otimes m} \to \Sym(\mathbb{M}^*) = A_0$ is a Poisson map relating the JMMS Hamiltonians~\eqref{eq:jmms_hamiltonians} and the simply-laced Hamiltonians~\eqref{eq:slims_bipartite}.
If one keeps the above notation for the coordinate functions on $\mathfrak{g}^m$ and $T^*\Hom(V^{\infty}_i,W^0)$, then 
\begin{equation}
	\label{eq:global_classical_moment_bipartite} 
	\mu^*\big((R_i)_{jk}\big) = Q_{ji}P_{ik} , ,
\end{equation}
because the $(j,k)$-component of the residue $R_i$ is the composition $Q_{ji}P_{ik} \colon V^0_k \to V^0_j$.

\begin{prop}
	\label{prop:pull_back_classical_3}
	One has $\mu^*\big(H_i^{\JMMS, \infty}\big) = H^{\infty}_i$ and $\mu^*\big(H_j^{\JMMS,0}\big) = H^0_j$ for all indices $1 \leq i \leq m$, $1 \leq j \leq n$.
\end{prop}

\begin{proof}
	The full expansions of the linear term of $H^{\JMMS,\infty}_i$ transforms as
	\begin{equation}
		\sum_j \mu^*\big(t^0_j(R_i)_{jj}\big) = \sum_j t^0_jQ_{ji}P_{ij} = \sum_j t^0_j\Tr(P_{ij}Q_{ji}) \, ,
	\end{equation}
	where in the last passage one uses the fact that $Q_{ji}P_{ij} \colon V^0_j \to V^0_j$ is a complex number. 
	Similarly, the linear term of $H^{\JMMS,0}_j$ becomes 
	\begin{equation}
		\sum_i \mu^*\big(t^{\infty}_i(R_i)_{jj}\big) = \sum_i t^{\infty}_i Q_{ji}P_{ij} = \sum_i t^{\infty}_i \Tr(Q_{ji}P_{ij}) \, .
	\end{equation}

	Next one computes
	\begin{equation}
		\begin{split}
			\sum_{i \neq k} \mu^*\big(\Tr(R_iR_k)\big) &= \sum_{i \neq k} \sum_{j,l} \big(\mu^*_i(R_i)_{jl}\big)\big(\mu^*_k(R_k)_{lj}\big) = \sum_{i \neq k} \sum_{j,l} Q_{ji}P_{il}Q_{lk}P_{kj} \\
			&= \sum_{i \neq k} \sum_{j,l} \Tr(P_{il}Q_{lk}P_{kj}Q_{ji}) \, ,
		\end{split}
	\end{equation}
	and 
	\begin{equation}
		\begin{split}
			\sum_{k \neq j} \sum_{i,p} \mu^*\big((R_i)_{jk}(R_p)_{kj}\big) &= \sum_{k \neq j} \sum_{i,p}  \Big(\mu^*(R_i)_{jk}\Big)\Big( \mu^*(R_p)_{kj}\Big) = \sum_{k \neq j} \sum_{i,p} Q_{ji}P_{ik}Q_{kp}P_{pj} \\
			&= \sum_{k \neq j} \sum_{i,p} \Tr(Q_{ji}P_{ik}Q_{kp}P_{pj}) \, .
		\end{split}
	\end{equation}

	Changing indices one recovers the formul\ae{}~\eqref{eq:slims_bipartite}. 
\end{proof}

The same construction of \S~\ref{sec:reduced_star=schlesinger} and~\ref{sec:reduced_dual_star=dual_Schlesinger} shows that the Hamiltonian reduction of the simply-laced Hamiltonians~\eqref{eq:slims_bipartite} with respect to comoment map for the action of the Lie algebra $\mathfrak{g}^{\infty} \coloneqq \bigoplus_{i \in I^{\infty}} \mathfrak{gl}(V^{\infty}_i)$ on $A_0$ corresponds to the JMMS Hamiltonians~\eqref{eq:jmms_hamiltonians}---at any ideal $\mathfrak{I} \subseteq \Sym(\mathfrak{g}^{\infty})$. 
Denoting $R(A_0,\mathfrak{g}^{\infty},\mathfrak{I})$ the Hamiltonian reduction, and $\pi_{\mathfrak{I}} \colon A_0^{\mathfrak{g}^{\infty}} \to R(A_0,\mathfrak{g}^{\infty},\mathfrak{I})$ the canonical projection, then the elements $\pi_{\mathfrak{I}}(H^{\infty}_i)$ and $\pi_{\mathfrak{I}}(H^0_j)$ are the Hamiltonian reductions of~\eqref{eq:slims_bipartite} for $(i,j) \in I^{\infty} \times I^0$.

In order to relate this with the JMMS Hamiltonians we use a canonical identification $\mathfrak{g}^{\infty} \simeq \mathfrak{g}^m$, which we obtain by further restricting the simply-laced isomonodromy system to the case where $V^{\infty}_i = W^0$ for all $i \in I^{\infty}$. 
Hence on the whole we attach $n$ one-dimensional vector spaces $\{V^0_j\}_{j \in I^0}$ to the nodes inside the part $I^0$, and then we attach one and the same $n$-dimensional vector space $W^0 = \bigoplus_{j \in I^0} V^0_j$ to the $m$ nodes inside the part $I^{\infty}$.
Then there is a canonical isomorphism $\Sym(\mathfrak{g}^*)^{\otimes m} \simeq \Sym(\mathfrak{g}^{\infty})$, and if the Hamiltonian reduction is taken at the ideal $\mathfrak{I} \subseteq \Sym(\mathfrak{g}^{\infty})$ then the ideal $\mathfrak{I}' \subseteq \Sym(\mathfrak{g}^*)^{\otimes m}$ corresponds to it under the isomorphism.
Geometrically, the choice of $\mathfrak{I}$ is the choice of coadjoint $G^{\infty}$-orbit $\mathcal{O} \subseteq (\mathfrak{g}^{\infty})^* \simeq \mathfrak{g}^{\infty}$, where $G^{\infty} \coloneqq \prod_{i \in I^{\infty}} \GL(V^{\infty}_i) \subseteq \GL(W^{\infty})$; similarly, $\mathfrak{I}'$ corresponds to an $m$-tuple of adjoint $\GL(W^0)$-orbits $\mathcal{O}' \subseteq \mathfrak{g}^m$. 

There is now a natural injective morphism 
\begin{equation}
	\varphi \colon \Sym(\mathfrak{g}^*)^{\otimes m} \big\slash \mathfrak{I}' \longrightarrow R(A_0,\mathfrak{g}^{\infty},\mathfrak{I}) \, ,
\end{equation}
induced by the composition of the projection $\pi_{\mathfrak{I}}$ after the map $\mu^*$~\eqref{eq:global_classical_moment_bipartite}. Finally Proposition~\ref{prop:pull_back_classical_3} yields the following.

\begin{thm}
	The classes of the JMMS Hamiltonians~\eqref{eq:jmms_hamiltonians} inside $\Sym(\mathfrak{g}^*)^{\otimes m} \big\slash \mathfrak{I}'$ match up with $\pi_{\mathfrak{I}}(H^{\infty}_i), \pi_{\mathfrak{I}}(H^0_j) \in R(A_0,\mathfrak{g}^{\infty},\mathfrak{I})$ under the natural correspondence $\varphi$.
\end{thm}

Hence indeed the JMMS system corresponds to the Hamiltonian reduction of this particular case of the simply-laced isomonodromy system, for every choice of ideal/coadjoint orbit. 

\begin{rem}
	As explained in \S~\ref{sec:reduced_star=schlesinger} and~\ref{sec:reduced_dual_star=dual_Schlesinger}, one can then take the quotient for the residual action of the group acting at the nodes $I^0$, i.e. the complex torus $\prod_{j \in I^0} \GL(V^0_j) \simeq (\mathbb{C}^*)^n$---the maximal torus of $\GL(W^0)$---and the quotient is the moduli space $\mathcal{M}^*_{\dR}$ of isomorphism classes of meromorphic connections~\eqref{eq:fmtv_meromorphic_connection} on a trivial holomorphic vector bundle.

	Then the Harnad duality~\cite{harnad_1994_dual_isomonodromic_deformations} acts on on the symplectic fibration $\widetilde{\mathcal{M}}^*_{\dR} \to \mathbf{B}$ of moduli spaces as an involution that swaps over the regular and irregular times.
\end{rem}

\subsection{FMTV is a quantisation of JMMS}
\label{sec:quantum_JMMS=FMTV}

In this section we give explicit formul\ae{} for the quantum Hamiltonians of the system of Felder--Markov--Tarasov--Varchenko~\cite{felder_markov_tarasov_varchenko_2000_dynamical_connection} (FMTV), when $\mathfrak{g} = \mathfrak{gl}_n(\mathbb{C})$. 
Then we show that the FMTV Hamiltonians are semiclassically equivalent to the PBW-quantisation of the JMMS Hamiltonians~\eqref{eq:jmms_hamiltonians}. 

Let $m \geq 1$ be an integer, $\mathfrak{g}$ a simple Lie algebra and $\mathfrak{t} \subseteq \mathfrak{g}$ a Cartan subalgebra. 
Then we generalise the KZ equations for $\mathfrak{g}$ adding on an additional parameter $\mu \in \mathfrak{t}_{\reg}$. 
Define the element $\Omega^{(ij)} \in U(\mathfrak{g})^{\otimes m}$ as in \S~\ref{sec:KZ}. Similarly, let $\widehat{\mu}^{(i)} \in U(\mathfrak{g})^{\otimes m}$ be the embedding of $\widehat{\mu} = \iota_U(\mu)$ on the $i$th factor, where $\iota_U \colon \mathfrak{g} \hookrightarrow U(\mathfrak{g})$ is the universal inclusion. 
Then we construct the universal KZ bundle with fibre $U(\mathfrak{g})^{\otimes m}$, defined over the base space $\mathbf{B} = \mathfrak{t}_{\reg} \times \Conf_m(\mathbb{C})$, and we consider the following system of linear partial differential equations for a local section $\psi$: 
\begin{equation}
	\label{eq:FMTV_I}
	\frac{\partial \psi}{\partial t^{\infty}_i} = \left(\widehat{\mu}^{(i)} + \sum_{1 \leq j \neq i \leq m} \frac{\Omega^{(ij)}}{t^{\infty}_i - t^{\infty}_j}\right)\psi \, ,
\end{equation}
where $\widehat{\mu}^{(i)}$ and $\Omega^{(ij)}$ act by left multiplication. 
This is Equation~3 of~\cite{felder_markov_tarasov_varchenko_2000_dynamical_connection}.\footnote{We replaced $u$ and $z_i$ by $\psi$ and $t^{\infty}_i$, respectively. 
Moreover, we consider the universal KZ equations instead of choosing finite-dimensional simple $\mathfrak{g}$-modules $V_i$, and thus we replace $V = V_1 \otimes \dotsm \otimes V_m$ with $U(\mathfrak{g})^{\otimes m}$. 
Finally, the choice of a complex parameter $\kappa$ is not relevant here as the connections obtained by changing it will still be strongly flat.}

Then Equation~4 on the same page provides a system of differential equations which is compatible with~\eqref{eq:FMTV_I}: the \emph{dynamical equations}.
This is a system of  differential equations for $\psi$ with respect to the variable $\mu \in \mathfrak{t}_{\reg}$. 
To write them, let $\mathcal{R} = \mathcal{R}(\mathfrak{g},\mathfrak{t})$ be the root system of the pair $(\mathfrak{g},\mathfrak{t})$, and consider the PBW-quantisations $\widehat{e}_{\alpha}, \widehat{f}_{\alpha} \in U(\mathfrak{g})$ of the vectors $e_{\alpha} \in \mathfrak{g}_{\alpha}$, $f_{\alpha} \in \mathfrak{g}_{-{\alpha}}$---which compose an $\mathfrak{sl}_2$-triplet together with $h_{\alpha} = [e_{\alpha},f_{\alpha}] \in \mathfrak{t}$. 
Finally, make a choice of positive roots $\mathcal{R}_+ \subseteq \mathcal{R}$. 
Then we impose that the derivative of $\psi$ in the direction of $\mu' \in \mathfrak{t}$ be given by:
\begin{equation}
	\label{eq:FMTV_II}
	\frac{\partial \psi}{\partial \mu'} = \left(\sum_{1 \leq i \leq m} t^{\infty}_i (\widehat{\mu}')^{(i)} + \sum_{\alpha \in \mathcal{R}_+} \frac{\langle \alpha, \mu'\rangle}{\langle\alpha,\mu\rangle} \widehat{e}_{\alpha} \cdot \widehat{f}_{\alpha}\right)\psi \, ,
\end{equation}
where $\langle \cdot, \cdot \rangle \colon \mathfrak{t}^* \otimes \mathfrak{t} \to \mathbb{C}$ is the dual pairing, $(\widehat{\mu}')^{(i)} \in U(\mathfrak{g})^{\otimes m}$ acts via left multiplication and $\widehat{e}_{\alpha} \cdot \widehat{f}_{\alpha}$ denotes the left multiplication of $\sum_{i,j} \widehat{e}^{(i)}_{\alpha} \cdot \widehat{f}^{(j)}_{\alpha}\in U(\mathfrak{g})^{\otimes m}$. The full FMTV system consists of the two sets of equations~\eqref{eq:FMTV_I} and~\eqref{eq:FMTV_II}. 

One can next encode these linear differential equations into a connection $\widehat{\nabla}^{\FMTV}$ defined on the trivial bundle $U(\mathfrak{g})^{\otimes m} \times \mathbf{B} \to \mathbf{B}$, the \emph{FMTV connection}:
\begin{equation}
	\label{eq:FMTV_connection}
	\widehat{\nabla}^{\FMTV} = d - \sum_{1 \leq i \leq m} \widehat{H}_i^{\FMTV,\infty}dt^{\infty}_i - \sum_{i \leq j \leq n} \widehat{H}_j^{\FMTV,0} dt^0_j \, ,
\end{equation}
where $n \coloneqq \dim(\mathfrak{t})$ is the rank of $\mathfrak{g}$ and $\{t^0_j\}_j$ is the coordinate system on the open subset $\mathfrak{t}_{\reg} \subseteq \mathfrak{t}$ induced from global coordinates on $\mathfrak{t} \simeq \mathbb{C}^n$. 
The time-dependent quantum operators $\widehat{H}_i^{\FMTV,\infty}$, $\widehat{H}_j^{\FMTV,0}$ are by definition the \emph{FMTV Hamiltonians}, for which we give explicit formul\ae{} in the case where $\mathfrak{g} \coloneqq \mathfrak{gl}_n(\mathbb{C})$, and $\mathfrak{t}$ is the standard Cartan subalgebra of diagonal matrices. 

\begin{prop}
	The FMTV Hamiltonians for $\mathfrak{g} = \mathfrak{gl}_n(\mathbb{C})$ read:
	\begin{equation}
		\begin{split}
			\label{eq:fmtv_hamiltonians}
			\widehat{H}_i^{\FMTV,\infty} &= \sum_{i \neq k}\sum_{j,l} \frac{\widehat{e}^{(i)}_{jl} \cdot \widehat{e}^{(k)}_{lj}}{t^{\infty}_i - t^{\infty}_k} + \sum_j t^0_j \widehat{e}^{(i)}_{jj}, \\
			\widehat{H}_j^{\FMTV,0} &= \sum_{k \neq j}\sum_{i,p} \frac{\widehat{e}^{(i)}_{jk} \cdot \widehat{e}^{(p)}_{kj}}{t^0_j - t^0_k} + \sum_i t^{\infty}_i \widehat{e}^{(i)}_{jj} \, ,
		\end{split}
	\end{equation}
	for all $i \in \{1,\dotsc,m\}$ and $j \in \{1,\dotsc,n\}$, where $(e_{ij})_{ij}$ is the canonical basis of $\mathfrak{g}$.
\end{prop}

\begin{proof}
	The root system is $\mathcal{R}(\mathfrak{g},\mathfrak{t}) = \{\alpha_{ij}\}_{i \neq j}$, with $\alpha_{ij}(\diag(t_1,\dotsc,t_n)) \coloneqq t_i - t_j$. 
	A root is positive if $i > j$, and the element $\mu \in \mathfrak{t}_{\reg}$ corresponds to the matrix $T^0 = \diag(t^0_1, \dotsc, t^0_n) = \sum_j t^0_je_{jj}$. 

	The Cartan term of~\eqref{eq:FMTV_I} then becomes $\widehat{\mu}^{(i)} = \sum_j t^0_j\widehat{e}^{(i)}_{jj}$, and we know that the KZ term expands to~\eqref{eq:kz_hamiltonians}. 

	To differentiate along a generic direction $\mu' \in \mathfrak{t}$ it is enough to consider partial derivatives with respect to the system of fundamental coweights $(e_{jj})_j$, which is a basis of $\mathfrak{t}$. 
	If $j \in \{1, \dotsc, n\}$ is fixed then we let $\mu' = e_{jj}$, and the Cartan term of~\eqref{eq:FMTV_II} reads
	\begin{equation}
		\sum_{1 \leq i \leq m} t^{\infty}_i (\widehat{\mu}')^{(i)} = \sum_{1 \leq i \leq m} t^{\infty}_i \widehat{e}^{(i)}_{jj} \, .
	\end{equation}
	Finally, the rightmost term of~\eqref{eq:FMTV_II} expands as follows:
	\begin{equation}
		\begin{split}
			\sum_{\alpha > 0} &\frac{\langle \alpha,\mu' \rangle}{\langle \alpha, \mu\rangle}\widehat{e}_{\alpha} \cdot \widehat{f}_{\alpha} = \sum_{k > l}\sum_{1 \leq i,p \leq m} \frac{\langle \alpha_{kl}, e_{jj}\rangle}{\langle \alpha_{kl},T^0\rangle} \widehat{e}^{(i)}_{kl} \cdot \widehat{e}^{(p)}_{lk} \\[10 pt]
			&= \sum_{j > l}\sum_{i,p}  \frac{\langle \alpha_{jl},e_{jj}\rangle}{\langle \alpha_{jl},T^0\rangle} \widehat{e}^{(i)}_{jl} \cdot \widehat{e}^{(p)}_{lj} + \sum_{i,p} \sum_{k > j} \frac{\langle \alpha_{kj},e_{jj}\rangle}{\langle \alpha_{kj},T^0\rangle} \widehat{e}^{(i)}_{kj} \cdot \widehat{e}^{(p)}_{jk} \\[10 pt]
			&= \sum_{k \neq j}\sum_{i,p}  \frac{\widehat{e}^{(i)}_{jk} \cdot \widehat{e}^{(p)}_{kj}}{t^0_j - t^0_k} \, ,
		\end{split}
	\end{equation}
	where we used $\langle \alpha_{kl},e_{jj}\rangle = \delta_{kj} - \delta_{lj}$ and $\langle \alpha_{kl},T^0\rangle = t^0_k - t^0_l$.
\end{proof}

The next proposition compares the FMTV Hamiltonians with the PBW-quantisation of the JMMS system~\eqref{eq:jmms_hamiltonians}, as time-dependent quantum operators $\mathbf{B} \to U(\mathfrak{g})^{\otimes m}$.

\begin{thm}
	\label{thm:quantum_JMMS=FMTV}
	Take $\mathfrak{g} = \mathfrak{gl}_n(\mathbb{C})$ in the FMTV system, and fix indices $1 \leq i \le m$, $1 \leq j \leq n$. 
	Then:
	\begin{enumerate}
		\item One has $\widehat{H}_i^{\FMTV,\infty} = \mathcal{Q}_{\PBW}\big(H_i^{\JMMS,\infty}\big)$ pointwise on $\mathbf{B}$.

		\item The time-dependent quantum operator
		\begin{equation}
			\widehat{H}_j^{\FMTV,0} - \mathcal{Q}_{\PBW}\big(H_j^{\JMMS,0}\big) \colon \mathbf{B} \longrightarrow U(\mathfrak{g})^{\otimes m}
		\end{equation}
		vanishes in the semiclassical limit.
	\end{enumerate}
	Hence the FMTV connection is a quantisation of the JMMS system.
\end{thm}

\begin{proof}
	As for the linear terms of~\eqref{eq:jmms_hamiltonians}, one has $\mathcal{Q}_{\PBW}\big(t^0_je^{(i)}_{jj}\big) = t^0_j\widehat{e}^{(i)}_{jj}$ and $\mathcal{Q}_{\PBW}\big(t^{\infty}_ie^{(i)}_{jj}\big) = t^{\infty}_i\widehat{e}^{(i)}_{jj}$, since the PBW-quantisation reduces to the universal inclusion on elements of degree one, and the result follows by linearity. 
	Next, it has been shown in \S~\ref{sec:quantum_Schlesinger=KZ} that the PBW-quantisation of the Schlesinger Hamiltonians~\eqref{eq:schlesinger_hamiltonians} yields the KZ Hamiltonians~\eqref{eq:kz_hamiltonians}, whence on the whole one has the first identity in the statement.

	For the PBW-quantisation of the leading term of $H^{\JMMS,0}_j$ instead one computes:
	\begin{equation}
		\begin{split}
			\mathcal{Q}_{\PBW}\Bigg(&\sum_{i,p}\sum_{j \neq k} \frac{(R_i)_{jk}(R_p)_{kj}}{t^0_j - t^0_k}\Bigg) = \sum_{i,p}\sum_{j \neq k} \frac{1}{2(t^0_j - t^0_k)}\big(\widehat{e}^{(i)}_{jk} \cdot \widehat{e}^{(p)}_{kj} + \widehat{e}^{(p)}_{kj} \cdot \widehat{e}^{(i)}_{jk}\big) \\
			&= \sum_{i,p}\sum_{j \neq k} \frac{\widehat{e}^{(i)}_{jk} \cdot \widehat{e}^{(p)}_{kj}}{t^0_j - t^0_k} + \sum_i \sum_{j \neq k} \frac{\widehat{e}^{(i)}_{kk} - \widehat{e}^{(i)}_{jj}}{2(t^0_j - t^0_k)} \, ,
		\end{split}
	\end{equation}
	using the identity $\bigl[\widehat{e}^{(i)}_{kj},\widehat{e}^{(p)}_{jk}\bigr] = \delta_{ip}\big(\widehat{e}^{(i)}_{kk} - \widehat{e}^{(i)}_{jj}\big)$ inside $U(\mathfrak{g})^{\otimes m}$. 

	Hence $\widehat{H}_j^{\FMTV,0}$ and $\mathcal{Q}_{\PBW}\big(H_j^{\JMMS,0}\big)$ have the same leading term, which implies that they have the same semiclassical limit once one passes to deformation quantisation as explained in \S~\ref{sec:quantisation_algebras}.
\end{proof}

\begin{rem}
	One can now make sense of the statement that the DMT connection~\eqref{eq:dmt_hamiltonians} is essentially a particular case of the FMTV connection. 
	Indeed, by taking $m = 1$ the FMTV Hamiltonian $\widehat{H}_j^{\FMTV,0}$ becomes 
	\begin{equation}
		\widehat{H}_j^{\FMTV,0} = \sum_{k \neq j} \frac{\widehat{e}_{jk} \cdot \widehat{e}_{kj}}{t^0_j - t^0_k} \, ,
	\end{equation}
	for $1 \leq j \leq n$. The analogous Hamiltonian of the DMT system~\eqref{eq:dmt_hamiltonians} instead is
	\begin{equation}
		\begin{split}
			\widehat{H}_j^{\DMT} &= \sum_{k \neq j} \frac{\widehat{e}_{jk} \cdot \widehat{e}_{kj} + \widehat{e}_{kj} \cdot \widehat{e}_{jk}}{2(t^0_j - t^0_k)} = \sum_{k \neq j} \frac{\widehat{e}_{jk} \cdot \widehat{e}_{kj}}{t^0_j - t^0_k} + \sum_{k \neq j} \frac{\bigl[\widehat{e}_{kj},\widehat{e}_{jk}\bigr]}{2(t^0_j - t^0_k)} \\
			&= \widehat{H}_j^{\FMTV,0} + \sum_{k \neq j} \frac{\widehat{e}_{kk} - \widehat{e}_{jj}}{2(t^0_j - t^0_k)} \, .
		\end{split}
	\end{equation}
	The lower-order correction is the exact analogous of the difference between the simply-laced quantum connection and the DMT connection, as discussed in \S~\ref{sec:corrected_quantum_reduced_dual_star=DMT}. 
\end{rem}

\subsection{Quantum Hamiltonian reduction}
\label{sec:corrected_quantum_reduced_bipartite=FMTV}

In this section we show that the quantum Hamiltonian reduction of the simply-laced quantum Hamiltonians of a generic complete bipartite graph is semiclassically equivalent to the FMTV Hamiltonians~\eqref{eq:fmtv_hamiltonians}.

Analogously to \S\S~\ref{sec:reduced_quantumstar=KZ} and~\ref{sec:corrected_quantum_reduced_dual_star=DMT}, the ``quantum'' change of variable $(\widehat{R}_i)_{jk} = \widehat{Q}_{ji}\widehat{P}_{ik}$ has the algebraic meaning of applying a quantum comoment map $\widehat{\mu}^* \colon U(\mathfrak{g})^{\otimes m} \to A = W(\mathbb{M}^*,\{\cdot,\cdot\})$ to the FMTV Hamiltonians~\eqref{eq:fmtv_hamiltonians}. 
The quantum comoment is defined by
\begin{equation}
	\widehat{\mu}^*\big((\widehat{R}_i)_{jk}\big) = \widehat{Q}_{ji} \cdot \widehat{P}_{ik} \, ,
\end{equation}
which yields
\begin{align*}
	\widehat{\mu}^*\big(\widehat{H}_i^{\FMTV,\infty}\big) &= \sum_{i \neq k,j,l} \frac{\Tr(\widehat{Q}_{ji}\widehat{P}_{il}\widehat{Q}_{lk}\widehat{P}_{kj})}{t^{\infty}_i - t^{\infty}_k} + \sum_j t^0_j \Tr(\widehat{Q}_{ji}\widehat{P}_{ij}), \\
	\widehat{\mu}^*\big(\widehat{H}_i^{\FMTV,0}\big) &= \sum_{k \neq j,i,l} \frac{\Tr(\widehat{Q}_{ji}\widehat{P}_{ik}\widehat{Q}_{kl}\widehat{P}_{lj})}{t^0_j - t^0_k} + \sum_i t^{\infty}_i \Tr(\widehat{Q}_{ji}\widehat{P}_{ij}) \, .
\end{align*}
These operators are obtained from the simply-laced quantum Hamiltonians $\rho_1(\widehat{H}^{\infty}_i)$, $\rho_1(\widehat{H}^0_j)$ by changing anchors, since they are by definition traces of quantum cycles anchored at nodes in the part $I^0 \subseteq I$. 
Hence there exist quantum potentials $\widehat{W}^{'\infty}_i$ and $\widehat{W}^{'0}_j$ for $(i,j) \in I^{\infty} \times I^0$ such that 
\begin{equation}
	\widehat{\mu}^*\big(\widehat{H}_i^{\FMTV,\infty}\big) = \Tr\big(\widehat{W}^{'\infty}_i\big), \qquad \text{and} \qquad \widehat{\mu}^*\big(\widehat{H}_i^{\FMTV,0}\big) = \Tr\big(\widehat{W}^{'0}_j\big) \, .
\end{equation}

This quantum system has the same semiclassical limit of the simply-laced quantum connection. 
To state this we mimic the argument of \S~\ref{sec:corrected_quantum_reduced_dual_star=DMT}: consider the Rees algebra $\Rees(A) \subseteq A[\hslash]$ of $A$---of Definition~\ref{def:rees_algebra}---and the topologically free $\mathbb{C} \llbracket \hslash \rrbracket$-algebra $\widehat{A} \subseteq A \llbracket \hslash \rrbracket$---of Proposition~\ref{prop:rees}.
Then we compare the universal simply-laced connection~\eqref{eq:slqc_bipartite} with the connection
\begin{equation}
	\widehat{\nabla}' \coloneqq d - \sum_{i \in I^{\infty}} \widehat{H}_i^{'\infty} dt_i - \sum_{j \in I^0} \widehat{H}_j^{'0} dt_j \, ,
\end{equation}
where $\widehat{H}_i^{'\infty} \coloneqq \Tr_{\hslash}\big(\widehat{W}^{'\infty}_i\big)$ and $\widehat{H}_j^{'0} \coloneqq \Tr_{\hslash}\big(\widehat{W}^{'0}_j\big)$. 
Both these connections are defined on the trivial bundle $\widehat{A} \times \mathbf{B} \to \mathbf{B}$, and we inspect their difference.

\begin{thm}
	The $A$-valued one-form $\widehat{\nabla} - \widehat{\nabla}'$ on $\mathbf{B}$ vanishes in the semiclassical limit.
\end{thm}

\begin{proof}
	The proof is analogous to that of Theorem~\ref{prop:quantum_correction}. One can compute explicitly 
	\begin{equation}
		\langle \widehat{\nabla} - \widehat{\nabla}',\partial_{t^{\infty}_i}\rangle = \widehat{H}^{'\infty}_i - \widehat{H}^{\infty}_i
	\end{equation}
	and 
	\begin{equation}
		\langle \widehat{\nabla} - \widehat{\nabla}',\partial_{t^0_j}\rangle = \widehat{H}^{'0}_j - \widehat{H}^0_j
	\end{equation}
	for all $i,j$. 
	This yields polynomials in $\hslash$ in which the coefficient of $\hslash^4$ (resp. $\hslash^2$) has order strictly smaller then four (resp. two); thus by definition these elements live in the kernel of the semiclassical limit~\eqref{eq:semiclassical_limit}.
\end{proof}

Hence indeed one can add a semiclassically vanishing term to this particular case of the simply-laced quantum connection, so that the quantum Hamiltonian reduction equals the FMTV connection---at any choice of ideal.

\section*{Conclusion/Outlook}

Hence in brief we attached a flat linear connection to every choice of a complete $k$-partite graph plus some decoration. 
Moreover, we explicitly computed the particular case corresponding to a complete bipartite graph, and related it to the quantum system of~\cite{felder_markov_tarasov_varchenko_2000_dynamical_connection}. 
This subsumes the connections of Knizhnik--Zamolodchikov~\cite{knizhnik_zamolodchikov_1984_wess_zumino_witten} and De Concini--Millson--Toledano~Laredo~\cite{millson_toledano_laredo_2005_casimir_connection}---both corresponding to a star-shaped graph.

Of course there is now an infinite family of examples of flat connections beyond them, which arise for $k \geq 3$. 

For instance taking a triangle---the complete graph on $k = 3$ nodes---yields a flat connection quantising the isomonodromic deformations of meromorphic connections on the sphere with a pole of order three at infinity.
In~\cite{nagoya_sun_2011_confluent_kz_equations_with_poincare_rank_2_at_infinity} such a quantisation was constructed in the case where the leading coefficient has simple spectrum, using confluent Verma modules. 
The formula for the Hamiltonian $\mathcal{H}^{(1)}_1$ in Example~3.3 of op. cit. is given in the absence of simple poles, and consists of a linear combination of operators of order two and three: in the viewpoint of this article this corresponds to a linear combination of the traces of all quantum 2-cycles and 3-cycles at a node of the triangle, which indicates that our perspective is compatible with~\cite{nagoya_sun_2011_confluent_kz_equations_with_poincare_rank_2_at_infinity}, with the addition of a proof of flatness.

Moreover, the setup of~\cite{nagoya_sun_2011_confluent_kz_equations_with_poincare_rank_2_at_infinity} allows for adding simple poles in the complex plane, which in the language of this paper means splaying one node of a triangle equipped with a degenerate reading. 
However no other node can be splayed, since they are all one-dimensional; thus for example the simply-laced quantum connection of a triangle with two nodes splayed is beyond the scope of~\cite{nagoya_sun_2011_confluent_kz_equations_with_poincare_rank_2_at_infinity}, so looks to be new.

Finally, note that the isomonodromic deformations of arbitrary meromorphic connections on Riemann surfaces are now known to be sections of flat symplectic Ehresmann connections (see~\cite{boalch_2001_symplectic_manifolds_and_isomonodromic_deformations, boalch_2011_geometry_and_braiding_of_stokes_data, boalch_yamakawa_2015_twisted_wild_character_varieties}), so there are many more flat linear connections, beyond the simply-laced case of this article, that may be obtained by quantising them (cf. Remark~\ref{rem:quantisation_local_systems}). 
In the general case the space of times will be the space of admissible deformations of \emph{wild Riemann surface structures} in the sense of~\cite{boalch_2011_geometry_and_braiding_of_stokes_data, boalch_yamakawa_2015_twisted_wild_character_varieties}.

\section*{Acknowledgements} 

I express my gratitude to P. Boalch for suggesting this problem to me. 

After posting this paper on the arXiv, D. Yamakawa informed me that he had established an unpublished result similar to Theorem~\ref{thm:quantum_flatness}, using a different construction that was not related to the KZ connection~\cite{yamakawa_2015_talk}. I thank him for sharing his notes with me.

	\appendix

	\section{Coordinate computations}
\label{sec:appendix}

Throughout this appendix we make a choice of global Darboux coordinates on $(\mathbb{M},\omega_a)$ so that 
\begin{equation}
	\big[\widehat{X}^{\alpha}_{ij},\widehat{X}^{\alpha^*}_{kl}\big] = \{X^{\alpha}_{ij},X^{\alpha^*}_{kl}\} = \varepsilon_{\alpha\alpha^*} \delta_{il}\delta_{jk} \in \{-1,0,1\} \, ,
\end{equation}
where $\alpha$ is an arrow of $\mathcal{G}$ with opposite arrow $\alpha^*$, and $\varepsilon_{\alpha\alpha^*} = -\varepsilon_{\alpha^*\alpha} \in \{\pm 1\}$ is a sign depending on the orientation of $\mathcal{G}$: we set $\varepsilon_{\alpha\alpha^*} = 1$ if $\alpha$ has positive orientation. 
All other quantum commutators and classical Poisson brackets vanish.

\begin{proof}[Proof of Lemma~\ref{lem:quantum_commutators}]
	Write $\widehat{C}_1 = \alpha_n \dotsm \underline{\alpha_1}$ and $\widehat{C}_2 = \beta_m \dotsm \underline{\beta_1}$, where $\alpha_i,\beta_j$ are arrows of $\mathcal{G}$. 
	One wants to expand
	\begin{equation}
		\big[\Tr(\widehat{C}_1),\Tr(\widehat{C}_2)\big] = \sum_{k,l} \big[\widehat{X}^{\alpha_n}_{k_nk_{n-1}} \dotsm \widehat{X}^{\alpha_1}_{k_1k_n},\widehat{X}^{\beta_m}_{l_ml_{m-1}} \dotsm \widehat{X}^{\beta_1}_{l_1l_m}\big] \, ,
	\end{equation}
	where $k = (k_1, \dotsc, k_n)$ and $l = (l_1,\dotsc,l_m)$ are suitable multi-indices. 
	Then applying the Leibniz rule recursively yields two alternative expansions:
	\begin{equation}
		\begin{split}
			\big[\Tr(\widehat{C}_1),&\Tr(\widehat{C}_2)\big] = \sum_{k,l}\sum_{i,j : \beta_j = \alpha_i^*} \big[\widehat{X}^{\alpha_i}_{k_ik_{i-1}},\widehat{X}^{\beta_j}_{l_jl_{j-1}}\big] \widehat{X}^{\beta_m}_{l_ml_{m-1}} \dotsm \widehat{X}^{\beta_{j+1}}_{l_{j+1}l_j} \\
			&\cdot \widehat{X}^{\alpha_n}_{k_nk_{n-1}} \dotsm \widehat{X}^{\alpha_{i+1}}_{k_{i+1}k_i} \cdot \widehat{X}^{\alpha_{i-1}}_{k_{i-1}k_{i-2}} \dotsm \widehat{X}^{\alpha_1}_{k_1k_n} \cdot \widehat{X}^{\beta_{j-1}}_{l_{j-1}l_{j-2}} \dotsm \widehat{X}^{\beta_1}_{l_1l_m} \, ,
		\end{split}
	\end{equation}
	and
	\begin{equation}
		\begin{split}
			\big[\Tr(\widehat{C}_1),&\Tr(\widehat{C}_2)\big] = \sum_{k,l}\sum_{i,j : \beta_j = \alpha_i^*} \big[\widehat{X}^{\alpha_i}_{k_ik_{i-1}},\widehat{X}^{\beta_j}_{l_jl_{j-1}}\big] \widehat{X}^{\alpha_n}_{k_mk_{n-1}} \dotsm \widehat{X}^{\alpha_{i+1}}_{k_{i+1}k_i} \\
			&\cdot \widehat{X}^{\beta_m}_{l_ml_{m-1}} \dotsm \widehat{X}^{\beta_{j+1}}_{l_{j+1}l_j} \cdot \widehat{X}^{\beta_{j-1}}_{l_{j-1}l_{j-2}} \dotsm \widehat{X}^{\beta_1}_{l_1l_m} \cdot \widehat{X}^{\alpha_{i-1}}_{k_{i-1}k_{i-2}} \dotsm \widehat{X}^{\alpha_1}_{k_1k_n} \, .
		\end{split}
	\end{equation}

	The result is a linear combination of words of length $n + m - 2$ in the alphabet $\widehat{X}^{\alpha}_{ij}$. 
	We want to write these words as traces of quantum cycles, and the condition given in the statement gives a consistent way to do it. 
	Indeed, say that $\widehat{C}_1$ satisfies the hypothesis, which means that it contains no pair of opposite arrows as soon as one of its arrows is removed. 
	Then one can permute
	\begin{equation}
		\widehat{X}^{\alpha_n}_{k_nk_{n-1}} \dotsm \widehat{X}^{\alpha_{i+1}}_{k_{i+1}k_i} \cdot \widehat{X}^{\alpha_{i-1}}_{k_{i-1}k_{i-2}} \dotsm \widehat{X}^{\alpha_1}_{k_1k_n} = \widehat{X}^{\alpha_{i-1}}_{k_{i-1}k_{i-2}} \dotsm \widehat{X}^{\alpha_1}_{k_1k_n} \cdot \widehat{X}^{\alpha_n}_{k_nk_{n-1}} \dotsm \widehat{X}^{\alpha_{i+1}}_{k_{i+1}k_i} \, ,
	\end{equation}
	and all summands of the first expansion get the desired form:
	\begin{equation}
		\begin{split}
			\big[\Tr(\widehat{C}_1),&\Tr(\widehat{C}_2)\big] = \sum_{k,l}\sum_{i,j : \beta_j = \alpha_i^*} \big[\widehat{X}^{\alpha_i}_{k_ik_{i-1}},\widehat{X}^{\beta_j}_{l_jl_{j-1}}\big] \widehat{X}^{\beta_m}_{l_ml_{m-1}} \dotsm \widehat{X}^{\beta_{j+1}}_{l_{j+1}l_j} \\
			&\cdot \widehat{X}^{\alpha_{i-1}}_{k_{i-1}k_{i-2}} \dotsm \widehat{X}^{\alpha_1}_{k_1k_n} \cdot \widehat{X}^{\alpha_n}_{k_nk_{n-1}} \dotsm \widehat{X}^{\alpha_{i+1}}_{k_{i+1}k_i} \cdot \widehat{X}^{\beta_{j-1}}_{l_{j-1}l_{j-2}} \dotsm \widehat{X}^{\beta_1}_{l_1l_m} = \\\\
			&= \sum_{i,j: \beta_j = \alpha_i^*} \varepsilon_{\alpha_i\alpha_i^*}\Tr\big(\widehat{X}^{\beta_m} \dotsm \widehat{X}^{\beta_{j+1}}\widehat{X}^{\alpha_{i-1}} \dotsm \widehat{X}^{\alpha_1}\widehat{X}^{\alpha_n} \dotsm \widehat{X}^{\alpha_{i+1}}\widehat{X}^{\beta_{j-1}} \dotsm \widehat{X}^{\beta_1}\big) \, ,
		\end{split}
	\end{equation}
	where we used $\big[\widehat{X}^{\alpha_i}_{k_ik_{i-1}},\widehat{X}^{\beta_j}_{l_jl_{j-1}}\big] = \big[\widehat{X}^{\alpha_i}_{k_ik_{i-1}},\widehat{X}^{\alpha_i^*}_{l_jl_{j-1}}\big] = \varepsilon_{\alpha_i\alpha_i^*}\delta_{k_il_{j-1}}\delta_{k_{i-1}l_j}$. 

	Hence setting 
	\begin{equation}
		\widehat{C} \coloneqq \sum_{i,j: \beta_j = \alpha_i^*} \varepsilon_{\alpha_i\alpha_i^*} \beta_m \dotsc \beta_{j+1} \alpha_{i-1} \dotsc \alpha_1 \alpha_n \dotsc \alpha_{i+1} \beta_{j-1} \dotsc \underline{\beta_1} 
	\end{equation}
	recovers the quantum potential in the statement of Proposition~\ref{prop:quantum_commutators}, since the expansion of the necklace Lie bracket of the semiclassical limits of $\widehat{C}_1$ and $\widehat{C}_2$ reads 
	\begin{equation}
		\begin{split}
			\big\{\Tr(C_1)&,\Tr(C_2)\big\}_{\mathcal{G}} = \\
			&= \sum_{i,j : \beta_j = \alpha_i^*} \varepsilon_{\alpha_i\alpha_i^*}\Tr\big(X^{\beta_m} \dotsm X^{\beta_{j+1}}X^{\alpha_{i-1}} \dotsm X^{\alpha_1}X^{\alpha_n} \dotsm X^{\alpha_{i+1}}X^{\beta_{j-1}} \dotsm X^{\beta_1}\big) \, ,
		\end{split}
	\end{equation}
	with the same signs $\varepsilon_{\alpha_i\alpha_i^*}$. 
	Then indeed $\widehat{C}$ is obtained by giving an anchor to all cycles of $\big\{\Tr(C_1),\Tr(C_2)\big\}_{\mathcal{G}}$.

	If $C_2$ is the only cycle satisfying the hypothesis then one may use the second expansion to conclude in the same way.
\end{proof}

\begin{proof}[Proof of Lemma~\ref{lem:vanishing_commutator_degenerate_4_cycles}]
	Looking at Figure~\ref{fig:vanishing_commutator_degenerate_4_cycles}, let $\alpha$ be the arrow from $a$ to $b$, $\beta$ the arrow from $b$ to $c$ and $\gamma$ the arrow from $c$ to $d$. 
	Then one has:
	\begin{equation}
		\begin{split}
			\Big[&\Tr\big(\widehat{X}^{\beta^*}\widehat{X}^{\beta}\widehat{X}^{\alpha^*}\widehat{X}^{\alpha}
			\big),\Tr\big(\widehat{X}^{\beta}\widehat{X}^{\beta^*}\widehat{X}^{\gamma^*}\widehat{X}^{\gamma}\big)\Big] \\
			&= \sum_{i,j,k,l,m,n,o,p} \Big[\widehat{X}^{\beta^*}_{ij}\widehat{X}^{\beta}_{jk}\widehat{X}^{\alpha^*}_{kl}\widehat{X}^{\alpha}_{li},\widehat{X}^{\beta}_{mn}\widehat{X}^{\beta^*}_{no}\widehat{X}^{\gamma^*}_{op}\widehat{X}^{\gamma}_{pm}\Big] \\
			&= \sum_{i,j,k,l,m,n,o,p} \Big[\widehat{X}^{\beta^*}_{ij}\widehat{X}^{\beta}_{jk},\widehat{X}^{\beta}_{mn}\widehat{X}^{\beta^*}_{no}\Big]\widehat{X}^{\alpha^*}_{kl}\widehat{X}^{\alpha}_{li}\widehat{X}^{\gamma^*}_{op}\widehat{X}^{\gamma}_{pm} \\	
			&= \sum_{i,j,k,l,m,n,o,p} \widehat{X}^{\beta^*}_{ij}\widehat{X}^{\beta}_{mn}\Big[\widehat{X}^{\beta}_{jk},\widehat{X}^{\beta^*}_{no}\Big]\widehat{X}^{\alpha^*}_{kl}\widehat{X}^{\alpha}_{li}\widehat{X}^{\gamma^*}_{op}\widehat{X}^{\gamma}_{pm} + \Big[\widehat{X}^{\beta^*}_{ij},\widehat{X}^{\beta}_{mn}\Big]\widehat{X}^{\beta^*}_{no}\widehat{X}^{\beta}_{jk}\widehat{X}^{\alpha^*}_{kl}\widehat{X}^{\alpha}_{li}\widehat{X}^{\gamma^*}_{op}\widehat{X}^{\gamma}_{pm} \\
			&= \sum_{i,j,k,l,m,n,o,p} \varepsilon_{\beta\beta^*}\delta_{jo}\delta_{kn}\widehat{X}^{\beta^*}_{ij}\widehat{X}^{\beta}_{mn}\widehat{X}^{\alpha^*}_{kl}\widehat{X}^{\alpha}_{li}\widehat{X}^{\gamma^*}_{op}\widehat{X}^{\gamma}_{pm} + \varepsilon_{\beta^*\beta}\delta_{jm}\delta_{in}\widehat{X}^{\beta^*}_{no}\widehat{X}^{\beta}_{jk}\widehat{X}^{\alpha^*}_{kl}\widehat{X}^{\alpha}_{li}\widehat{X}^{\gamma^*}_{op}\widehat{X}^{\gamma}_{pm} \\
			&= \varepsilon_{\beta\beta^*} \Big(\Tr\big(\widehat{X}^{\beta^*}\widehat{X}^{\gamma^*}\widehat{X}^{\gamma}\widehat{X}^{\beta}\widehat{X}^{\alpha^*}\widehat{X}^{\alpha}\big) - \Tr\big(\widehat{X}^{\beta^*}\widehat{X}^{\gamma^*}\widehat{X}^{\gamma}\widehat{X}^{\beta}\widehat{X}^{\alpha^*}\widehat{X}^{\alpha}\big)\Big) = 0 \, . \qedhere
		\end{split}
	\end{equation}
\end{proof}

\begin{proof}[Proof of Lemma~\ref{lem:commutator_degenerate_4_cycles}]
	Denote $\beta$ and $\beta^*$ the arrows that the two degenerate 4-cycles have in common, with $t(\beta) = j$. 
	Then denote $\alpha,\alpha^*$ (resp. $\gamma,\gamma^*$) the remaining arrows of the leftmost cycle (resp. of the rightmost cycle), with $t(\alpha) = j$ (resp. $t(\gamma) = j$). 
	This means that $\alpha$, $\beta$ and $\gamma$ come out of the common centre, whereas $\alpha^*$, $\beta^*$ and $\gamma^*$ point towards the common centre. 

	Then one has:
	\begin{equation}
		\begin{split}
			\Big[&\Tr\big(\widehat{X}^{\beta^*}\widehat{X}^{\beta}\widehat{X}^{\alpha^*}\widehat{X}^{\alpha}\big),\Tr\big(\widehat{X}^{\beta^*}\widehat{X}^{\beta}\widehat{X}^{\gamma^*}\widehat{X}^{\gamma}\big)\Big] \\
			&= \sum_{i,j,k,l,m,n,o,p} \big[\widehat{X}^{\beta^*}_{ij}\widehat{X}^{\beta}_{jk}\widehat{X}^{\alpha^*}_{kl}\widehat{X}^{\alpha}_{li},\widehat{X}^{\beta^*}_{mn}\widehat{X}^{\beta}_{no}\widehat{X}^{\gamma^*}_{op}\widehat{X}^{\gamma}_{pm}\big] \\
			&= \sum_{i,j,k,l,m,n,o,p} \big[\widehat{X}^{\beta^*}_{ij}\widehat{X}^{\beta}_{jk},\widehat{X}^{\beta^*}_{mn}\widehat{X}^{\beta}_{no}\big]\widehat{X}^{\alpha^*}_{kl}\widehat{X}^{\alpha}_{li}\widehat{X}^{\gamma^*}_{op}\widehat{X}^{\gamma}_{pm} \\
			&= \sum_{i,j,k,l,m,n,o,p} \widehat{X}^{\beta^*}_{ij}\big[\widehat{X}^{\beta}_{jk},\widehat{X}^{\beta^*}_{mn}\big]\widehat{X}^{\beta}_{no}\widehat{X}^{\alpha^*}_{kl}\widehat{X}^{\alpha}_{li}\widehat{X}^{\gamma^*}_{op}\widehat{X}^{\gamma}_{pm} + \widehat{X}^{\beta^*}_{mn}\big[\widehat{X}^{\beta^*}_{ij},\widehat{X}^{\beta}_{no}\big]\widehat{X}^{\beta}_{jk}\widehat{X}^{\alpha^*}_{kl}\widehat{X}^{\alpha}_{li}\widehat{X}^{\gamma^*}_{op}\widehat{X}^{\gamma}_{pm} \\
			&= \sum_{i,j,k,l,m,n,o,p} \varepsilon_{\beta\beta^*}\delta_{jn}\delta_{km} \widehat{X}^{\beta^*}_{ij}\widehat{X}^{\beta}_{no}\widehat{X}^{\alpha^*}_{kl}\widehat{X}^{\alpha}_{li}\widehat{X}^{\gamma^*}_{op}\widehat{X}^{\gamma}_{pm} + \varepsilon_{\beta^*\beta} \delta_{io}\delta_{jk}\widehat{X}^{\beta^*}_{mn}\widehat{X}^{\beta}_{jk}\widehat{X}^{\alpha^*}_{kl}\widehat{X}^{\alpha}_{li}\widehat{X}^{\gamma^*}_{op}\widehat{X}^{\gamma}_{pm} \\
			&= \varepsilon_{\beta\beta^*}\Big(\Tr\big(\widehat{X}^{\alpha^*}\widehat{X}^{\alpha}\widehat{X}^{\beta^*}\widehat{X}^{\beta}\widehat{X}^{\gamma^*}\widehat{X}^{\gamma}\big) -  \Tr\big(\widehat{X}^{\beta^*}\widehat{X}^{\beta}\widehat{X}^{\alpha^*}\widehat{X}^{\alpha}\widehat{X}^{\gamma^*}\widehat{X}^{\gamma}\big)\Big) \, .
			\end{split}
		\end{equation}
	Now one can orient $\mathcal{G}$ so that $\beta$ has positive orientation, whence $\varepsilon_{\beta\beta^*} = 1$, recovering the situation depicted in Figure~\ref{fig:commutator_degenerate_4_cycles}.
\end{proof}

	\bibliographystyle{amsplain} 
		\bibliography{bibliography}
\end{document}